\documentclass[12pt]{article}
\usepackage{amsmath}
\usepackage{graphicx}
\usepackage{enumerate}
\usepackage{natbib}
\usepackage{url} 
\usepackage{soul}
\usepackage{caption}
\newcommand{\blind}{1}

\usepackage{amssymb}
\usepackage{algorithmicx}
\usepackage{algcompatible}
\usepackage{algorithm}
\usepackage{changebar}
\usepackage{xr-hyper}
\usepackage{hyperref}
\hypersetup{ 
  colorlinks,
  citecolor=blue,
  linkcolor=red,
  urlcolor=magenta}
\usepackage{mathtools}
\usepackage{subfigure}
\usepackage{amsmath,amsthm,bm,bbm}
\usepackage{rotating}
\usepackage{mathrsfs}
\usepackage{chngcntr}
\usepackage{apptools}
\usepackage[titletoc,title]{appendix}
\usepackage{comment}
\usepackage{xcolor}         
\usepackage{setspace,lscape,longtable}
\usepackage{fullpage}  
\usepackage{titlesec}
\usepackage[shortlabels]{enumitem}
\usepackage[english]{babel}
\usepackage{enumitem}
\setlist[enumerate]{itemsep=0mm}
\setlist[itemize]{itemsep=0mm}

\usepackage{algpseudocode}
\usepackage{algorithm}

\usepackage{xr}

\usepackage{xcolor}

\newtheorem{theorem}{Theorem}[section]
\newtheorem{corollary}{Corollary}[section]
\newtheorem{lemma}[theorem]{Lemma}

\newtheorem{assumption}{Assumption}
\theoremstyle{remark}
\newtheorem{remark}{Remark}

\newcommand{\bmu}{\bm{\mu}}
\newcommand{\by}{\bm{Y}}
\newcommand{\bl}{\bm{L}}
\newcommand{\bz}{\bm{z}}
\newcommand{\bI}{\bm{I}}
\newcommand{\bE}{\bm{E}}
\newcommand{\bX}{\bm{X}}
\newcommand{\bw}{\bm{w}}
\newcommand{\bp}{\bm{p}}
\newcommand{\bA}{\bm{A}}
\newcommand{\bB}{\bm{B}}
\newcommand{\bU}{\bm{U}}
\newcommand{\bV}{\bm{V}}
\newcommand{\bO}{\bm{O}}
\newcommand{\bR}{\bm{R}}
\newcommand{\bQ}{\bm{Q}}
\newcommand{\bD}{\bm{D}}
\newcommand{\bP}{\bm{P}}
\newcommand{\bSigma}{\bm{\Sigma}}

\newcommand{\bv}{\bm{v}}

\newcommand{\bZ}{\bm{Z}}

\DeclareMathOperator{\E}{\mathbb{E}}
\DeclareMathOperator{\prob}{\mathbb{P}}

\DeclareMathOperator*{\argmin}{arg\,min}

\newcommand{\yx}{\color{blue} \it}
\newcommand{\xx}{\color{black} \rm}

\begin{document}

\def\spacingset#1{\renewcommand{\baselinestretch}%
{#1}\small\normalsize} \spacingset{1}


\if1\blind
{
  \title{\bf Bayesian Sparse Gaussian Mixture Model in High Dimensions}
  \author{Dapeng Yao
  \\
    Department of Applied Mathematics and Statistics, Johns Hopkins University\\
    and \\
    Fangzheng Xie\\
    Department of Statistics, Indiana University\\
    and \\
    Yanxun Xu\\
    Department of Applied Mathematics and Statistics, Johns Hopkins University\\
    }
    \date{}
  \maketitle
} \fi

\if0\blind
{
  \bigskip
  \bigskip
  \bigskip
  \begin{center}
    {\LARGE\bf \color{orange}Bayesian Sparse Gaussian Mixture Model in High Dimensions}
\end{center}
  \medskip
} \fi

\bigskip

\begin{abstract}
 We study the sparse high-dimensional Gaussian mixture model when the number of clusters is allowed to grow with the sample size.
A minimax lower bound for parameter estimation is established, and we show that a constrained maximum likelihood estimator achieves the minimax lower bound.
However, this optimization-based estimator is computationally intractable because the objective function is highly nonconvex and the feasible set involves discrete structures. To address the computational challenge, we propose a Bayesian approach to estimate high-dimensional Gaussian mixtures whose cluster centers exhibit sparsity using a continuous spike-and-slab prior. 
Posterior inference can be efficiently computed using an easy-to-implement Gibbs sampler. We further prove that the posterior contraction rate of the proposed Bayesian method is minimax optimal.
The mis-clustering rate is obtained as a by-product using tools from matrix perturbation theory.
The proposed Bayesian sparse Gaussian mixture model does not require pre-specifying the number of clusters, which 
 can be adaptively estimated via the Gibbs sampler. 
The validity and usefulness of the proposed method is demonstrated through simulation studies and the analysis of a real-world single-cell RNA sequencing dataset.
\end{abstract}

\noindent%
\vfill
\addtolength{\textheight}{1in}%
\newpage
\spacingset{1.9} 
\section{Introduction}
\label{sec:intro}


Clustering is a powerful tool for detecting structures in heterogeneous data and identifying homogeneous subgroups with a wide range of applications, such as genomics \citep{gu2008bayesian}, pattern recognition \citep{diday1981clustering},  
and topic modeling \citep{blei2003latent}.  
In many scientific domains, data are often high-dimensional, i.e., the dimension of observations can be larger than the sample size. For example, an important task in the analysis of single-cell RNA-sequencing data, where the number of genes (dimension)  is usually larger than the number of cells (sample size), is to cluster cells and identify   functional cell subpopulations \citep{cao2019single}.   
A principal challenge of extending the low-dimensional clustering techniques to high dimensions is the well-known  ``curse of dimensionality." 
To overcome this issue, dimensionality reduction \citep{ding2002adaptive} or additional structural assumptions \citep{caichime2019} are usually necessary in high dimensional clustering methods. 


High dimensional clustering and mixture models have attracted attention recently from 
the frequentist perspective. 
When the dimension $p$ has at most the same order as the sample size $n$,
\citet{doss2023optimal} studied the optimal rate of estimation in a finite Gaussian location mixture model without a separation condition.
\citet{loffler2020optimality} showed that spectral clustering is minimax optimal 
in the Gaussian mixture model with isotropic covariance matrix when $p=o(n\Delta)$, where $\Delta$ is the minimal distance among cluster centers.
When $p\gg n$, \citet{azizyan2013} considered a simple case in which there are only two clusters 
with equal mixing weights and same isotropic covariance matrices. 
\citet{jin2016} and \citet{jin2017} proposed influential features principal component analysis based on feature selection and principal component analysis. A phase transition phenomenon in high dimensional clustering problem was also investigated in \cite{jin2016} and \cite{jin2017} across different sparsity and signal levels.
\citet{caichime2019} proposed a modified Expectation-Maximization (EM) algorithm 
based on sparse discriminant vectors to obtain the minimax optimal convergence rate of the excess mis-clustering error. 
In terms of density estimation, \citet{ashtiani2020} obtained a near-optimal convergence rate for high dimensional location-scale mixtures with respect to the total variation distance. 
%
%

Despite these theoretical and computational developments in high dimensional clustering, most frequentist approaches dealing with finite mixtures assume that the number of clusters $K$ is either known or needs to be estimated consistently using techniques such as cross-validation \citep{smyth2000crossvalidation} and the gap statistics \citep{gap}. 
In contrast, Bayesian methods treat $K$ as an unknown parameter and put a prior on it. 
For example, \citet{miller2018} proposed a mixture of finite mixtures model with a Gibbs sampler, and the posterior consistency of $K$ was later studied in \citet{miller2023consistency}. 
\citet{ohn2020optimal} established a near optimal rate for estimating finite Gaussian mixtures with respect to the Wasserstein distance when $K$ 
is unknown and allowed to grow with $n$. 
In the context of Bayesian model-based clustering for high-dimensional data, \cite{tadesse2005} proposed a method for uncovering cluster structure and identifying the discriminating variables simultaneously and exploited reversible jump Markov chain Monte Carlo (MCMC) for posterior sampling.  

\cite{gao2020general} proposed a Bayesian structured linear model, which encompasses the bi-clustering problem as a special case, utilizing a subset-selection prior and establishing an optimal posterior contraction rate for parameter estimation. 

\citet{chandra2021escaping} proposed a Bayesian latent
factor mixture model and investigated the behavior of the induced cluster memberships as $p$ goes to infinity whereas $n$ remains fixed. 

However, a general theoretical framework for Bayesian analysis of high-dimensional clustering in terms of both parameter estimation and mis-clustering error is yet to be established.


The Gaussian mixture model we consider lies in the regime of high dimensionality with sparsity structures.
There has been a growing interest in Bayesian inference with sparsity-enforcing priors. One commonly-used prior is the spike-and-slab prior \citep{mitchell1988}, which is a mixture of a point mass at zero and
a relatively ``flat" absolutely continuous density. 
The spike-and-slab LASSO prior \citep{rockova2018} borrows the similarity between the LASSO and Laplace prior, and combines it with a continuous version of the spike-and-slab prior. Theoretical properties of the spike-and-slab LASSO prior were well studied in the context of regression models, graphical models, and Gaussian sequence models (see \citealp{SSLreview} for a review). Another class of sparsity-enforcing priors is global-local shrinkage priors, such as the horseshoe prior \citep{carvalho2009} and the Dirichlet-Laplace prior \citep{bhattacharya2015}. We refer the readers to \citet{tadesse2021handbook} and references therein. However, when these types of priors are adapted to high-dimensional sparse clustering problems with an unknown number of clusters, their theoretical properties remain unclear.

This paper presents the Bayesian analysis of a high-dimensional sparse Gaussian mixture model using a spike-and-slab LASSO prior and establishes the optimality of the proposed estimation procedure.
Our main contribution is threefold. 
First, we fully characterize the minimax rate for parameter estimation in the high-dimensional sparse Gaussian mixture model, in which the number of clusters is allowed to grow with sample size. A frequentist constrained maximum likelihood estimator (MLE) is proved to achieve minimax optimality. Second, since the frequentist optimization-based estimator is computationally intractable, we propose a Bayesian sparse Gaussian mixture model to address this computational challenge, and establish its theoretical properties by showing that the posterior contraction rate for parameter estimation is minimax optimal. Third, we obtain a posterior contraction rate for the mis-clustering error using tools from matrix perturbation theory \citep{yu2014}. 
To the best of our knowledge, this paper represents the first effort in establishing the theoretical results of parameter estimation and clustering recovery in high-dimensional Bayesian sparse Gaussian mixture model with a growing number of clusters.

The rest of this paper is organized as follows. In Section \ref{sec:model}, we introduce the high-dimensional clustering problem and our model, establish the minimax lower bound for parameter estimation, and propose a frequentist constrained MLE that achieves the minimax lower bound. Section \ref{sec:theory} elaborates on the main theoretical results, including the optimal posterior contraction rate and the mis-clustering error. 
We demonstrate the practical performance of the proposed method through simulation studies in Section \ref{sec:simu} and a real-world application to clustering single-cell RNA sequencing data in Section \ref{sec:real}. 

\noindent\textbf{Notations:} 
Let $|S|$ denote the cardinality of $S$ if the set $S$ is finite or the volume (Lebesgue measure) of $S$ if $S$ is a Lebesgue-measurable infinite subset in Euclidean space. Denote $[n]$ as the set of all consecutive integers $\{1,\dots,n\}$. We use $\lesssim$ and $\gtrsim$ to denote the inequality up to a constant. In other words, $a\lesssim(\gtrsim\text{, resp.})$ $b$ if $a\leq(\geq\text{, resp.})$ $Cb$ for some constant $C>0$. We write $a\asymp b$ if $a\lesssim b$ and $b\lesssim a$. We use $\lfloor x \rfloor$ to denote the greatest integer less than or equal to the real number $x$ and $\lceil x \rceil$ to denote the smallest integer greater than or equal to the real number $x$.
For a $p$-dimensional vector $\bm{x}\in\mathbb{R}^p$, we denote $x_i$ as the $i$th coordinate of $x$. Also, we denote $\|\bm{x}\|_1=\sum_{i=1}^p |x_i|$ as the $\ell_1$-norm, $\|\bm{x}\|_2=\sum_{i=1}^p |x_i|^2$ as the $\ell_2$-norm, $\|\bm{x}\|_\infty=\max_{i\in[p]}|x_i|$ as the $\ell_\infty$-norm, and $\|\bm{x}\|_0=\sum_{i=1}^p\mathbbm{1}(x_i\neq 0)$. 
For any matrix $\bA\in\mathbb{R}^{n\times m}$, let $A_{ij}$ denote the $(i,j)$-entry of $\bA$ and let $\bA_{i\cdot}$ and $\bA_{\cdot j}$ be the $i$th row and $j$th column of $\bA$, respectively. We denote $\|\bA\|_F=\sqrt{\sum_{i=1}^n\sum_{j=1}^m|A_{ij}|^2}$ to be the Frobenius norm of $\bA$ and $\|\bA\|_2$ to be the spectral norm of $\bA$. 
We denote $\bA \succ \bm{0}$ if $\bA$ is a positive definite matrix.
The prior and posterior  probability distributions are denoted as $\Pi$ and the corresponding densities with respect to the underlying $\sigma$-finite measure (whenever it exists) are denoted as $\pi$. 
We denote $D_{KL}(\mathbb{P}\|\mathbb{Q})$ the Kullback–Leibler divergence between any probability measures $\mathbb{P}$ and $\mathbb{Q}$. The $\epsilon$-packing number of a metric space $\Theta$ with respect to the metric $d$, which is the maximum number of pairwise disjoint balls contained $\Theta$ with radii $\epsilon$, is denoted as $M(\epsilon,\Theta,d)$. In the rest of the paper, we will use an asterisk to represent the ground true values of the parameters that give rise to the data distribution. 

\section{Model}\label{sec:model}

\subsection{Gaussian mixture model and clustering}
Let $\by=[\by_1,\dots,\by_n]$ be a $p\times n$ data matrix, where rows represent variables or features, and columns represent observations. We assume that the data exhibits a clustering structure that can be described through a Gaussian mixture model as follows. Let $\bmu_1,\dots,\bmu_K\in\mathbb{R}^p$ be the cluster centroids of the respective clusters, where $K\geq 1$ is the number of clusters. Let $\bz=(z_1,\dots,z_n)^T\in[K]^n$ be the cluster membership vector for observations, with $z_i = k$ indicating that $\by_i$ belongs to the $k$th cluster.
The distribution of $\by_i$ is given by 
\begin{align}\label{model}
    \by_i=\bmu_{z_i}+\bm{\epsilon}_i,  
\end{align}
where $\bm{\epsilon}_i\sim N_p(\bm{0}, \bSigma)$ independently for $i\in [n]$. 
The goal is to estimate the cluster centroids $\bmu_1,\ldots,\bmu_K$ as well as to recover the latent cluster membership vector $\bz$. 
This paper considers the asymptotic regime where both $p$ and $n$ go to infinity and $p/n\to\infty$. 
When $p$ does not exceed $n$, \citet{azizyan2013} proved that the expected clustering accuracy (which will be defined formally later) depends on the dimension $p$ through the rate $\sqrt{p/n}$ in the two-cluster problem without additional structural assumptions. Under the regime that $p/n\to\infty$ considered in our framework, we posit the following sparse structure on the 
cluster mean vectors $\bmu_1,\ldots,\bmu_K$. Denote $\bmu = [\bmu_1,\ldots,\bmu_K]$ 
as the matrix concatenated by the mean vectors of all clusters and define the support of $\bmu$ as
$\text{supp}(\bmu)=\{j\in[p]:(\bmu_{j\cdot})^T\neq 0\}$.
We say that $\bmu$ is jointly $s$-sparse if $|\text{supp}(\bmu)|\leq s$. Moreover, we require that not only each $\bmu_k$ has at most $s_n$ non-zero coordinates, 
namely, $\sum_{j=1}^p\mathbbm{1}(\mu_{kj}\neq 0)\leq s_n$ for all $k\in [K]$, but also that $\bmu$ is jointly $s_n$-sparse. We assume 
$s_n\to\infty$ as $n\to\infty$. 
In the sequel, we will drop the subscript $n$ from $s_n$ and write $s = s_n$ for notation simplicity, but the readers should be reminded that $s$ depends on $n$ implicitly. 

Denote $\bm{e}_i\in\mathbb{R}^K$ the unit vector that has value 1 at the $i$th coordinate and $0$ elsewhere. Let 
$\bl=[\bm{l}_1,\dots,\bm{l}_n]^T\in\mathbb{R}^{n\times K}$ where $\bm{l}_i=\bm{e}_{z_i}$. Then $\bl$ is the matrix whose rows represent cluster memberships of the $n$ observations. It follows immediately that the expected data matrix can be written as
$\E(\by)=\bmu\bl^T$. Namely, our model can be represented as a signal-plus-noise model matrix 
$\by=\bmu\bl^T+\bE$,
 where $\bE=[\bm{\epsilon}_1,\dots,\bm{\epsilon}_n]$ is the mean-zero noise matrix where $\bm{\epsilon}_i$'s are independent normal random vectors with mean zero and covariance matrix $\bSigma$. 
 As $K$ is typically much smaller than $n$, the above representation of the model is similar to those in \citet{cape2018} and  \citet{agterberg2022entrywise} because the data matrix has a low expected rank. Nevertheless, the sparse Gaussian mixture model differs from \citet{cape2018} and
 \citet{agterberg2022entrywise} in that the columns of the expected data matrix have the clustering structure and the rows have the sparsity structure. Following the previous convention of using asterisk to indicate true parameter values, we denote $(\bmu^*, \bl^*, \bSigma^*)$ the underlying truth of $(\bmu, \bl, \bSigma)$ throughout the rest of the paper.

\subsection{Minimax lower bound}\label{ssec:minimax_lower}

One of the major theoretical contributions of this paper is to study the 
estimation error of the mean matrix $(\bmu^*)(\bl^*)^T$ in the proposed sparse Gaussian mixture model. 

This differs from most existing minimax results in the clustering literature, which predominantly focus on evaluating the mis-clustering error.
As the first step towards the complete theory, we establish the minimax lower bound.
Formally, consider the following parameter space 
\begin{align*}
\Theta^*_K = \{(\bmu, \bl, \bSigma):& \bmu\in\mathbb{R}^{p\times K}, \bl\in\mathcal{L}_K,|\mathrm{supp}(\bmu)|\leq s, \|\bmu\bl^T\|_F^2=O(sn),\\
&\bSigma \succ \bm{0}, 0 < m_{\bSigma} \leq \lambda_{\min}(\bSigma) \leq \lambda_{\max}(\bSigma) \leq M_{\bSigma} < \infty
\},
\end{align*} 
where $\mathcal{L}_K=\{\bl\in\mathbb{R}^{n\times K}:\bl=[\bm{l}_1,\dots,\bm{l}_n]^T,\bm{l}_i\in\{0,1\}^K, \|\bm{l}_i\|_0=1 \text{ for all } i\in[n]\}$ 
is the set of cluster assignment matrices
, $\mathrm{supp}(\bmu)$ is the set of indices of the non-zero rows of $\bmu$, and $\lambda_{\min}(\bSigma)$ and $\lambda_{\max}(\bSigma)$ represents the smallest and largest singular value of $\bSigma$ respectively. We also denote $\Delta =\min_{k_1\neq k_2}\|\bmu^*_{k_1}-\bmu^*_{k_2}\|_2$ as the minimum separation of the cluster centers.

We next present a collection of assumptions that are necessary in  theoretical analyses.
\begin{assumption}\label{assump1}
(Low rank) $K\log n\lesssim \log p$, $K\leq s$.
\end{assumption}
\begin{assumption}\label{assump3}
(Minimum separation) 
$\Delta \geq \frac{1}{n^q}$ for some constant $q > 0$.
\end{assumption} 
\vspace{-.1in}
\begin{assumption}\label{assump4}
(High dimensionality) $p/n\to\infty$.
\end{assumption}

Assumption \ref{assump1} is a mild low-rank assumption and can be satisfied even with increasing $K$.  
Assumption \ref{assump3} requires that the centers of different clusters are well separated and is common in high-dimensional clustering problems. 
 
It also guarantees the identifiability of $(\bmu, \bl)$ up to a permutation.

Assumption \ref{assump4} requires $p/n\to\infty$ and it describes the high-dimensional nature of the problem.
Below, Theorem \ref{thm2} establishes the minimax lower bound for estimating the mean matrix with regard to the Frobenius norm. 
\begin{theorem}\label{thm2}
Let $\by = (\bmu^*)(\bl^*)^T + \bE$  where each column of $\bE$ is normal random vector with mean zero and covariance matrix $\bSigma^*$.
Assume Assumptions \ref{assump1}-\ref{assump4} hold.
If $s\leq p/4$,
then there exists a constant $C>0$ such that
\begin{align*}
\inf_{\widehat{\bmu},\widehat{\bl}}\sup_{(\bmu^*,\bl^*,\bSigma^*)\in\Theta^*_K}\E_*\left\{\|\widehat{\bmu}\widehat{\bl}^T- (\bmu^*)(\bl^*)^T\|_F^2\right\}
\geq C\left(s \log p 
+ n\log K 
\right)
\end{align*}
for sufficiently large $n$, where $\E_*$ denotes the expected value with respect to $(\bmu^*,\bl^*,\bSigma^*)$.
\end{theorem}

The key challenge in the proof of Theorem \ref{thm2} lies in designing suitable subsets of the parameter space for $(\bmu, \bl)$. We construct three parameter subspaces, each essentially fixing $(\bmu)_S$, $S$, and $\bl$, respectively. By controlling the Kullback–Leibler diameter and entropy of each subspace, we apply Fano's lemma to derive minimax lower bounds on the convergence rate in each subspace, integrating them to obtain the final minimax lower bound.

The minimax lower bound consists of two parts: $s\log p$ and $n\log K$. The $s\log p$ term describes the logarithmic complexity of selecting $s$ non-zero coordinates among $p$ variables. It appears repeatedly in the minimax rates for high-dimensional problems where sparsity plays an important role, including the sparse normal means problem \citep{10.1214/12-AOS1029}
and the sparse linear regression 
\citep{10.1214/15-AOS1334}. 
The term $n\log K$ comes from the logarithmic complexity of assigning $n$ points into $K$ clusters and also appears in the minimax risk for parameter estimation in stochastic block models \citep{ghosh2020posterior}. 

\begin{remark}\label{remark:minimax_lower}

This result fills the gap in the literature of high dimensional low-rank matrix estimation, particularly in scenarios where both sparsity and clustering structures exist. 

When $\by$ is a $p\times n$ random matrix that can be written as $\by = \bX+\bE$, where $\bE$ is a $p\times n$ noise matrix whose entries are independent standard normal random variables and $\bX$ is a $p\times n$ rank-$K$ matrix, 
\citet{yang2016} showed that, if $\bX$ not only is low rank but also has only an $s\times l$ non-zero submatrix, then the minimax lower bound 
is $K(s+l)+s\log({ep}/{s})+l\log({en}/{l})$.
Our minimax lower bound is sharper than the above bound because the right singular subspace induced by $\bl$ contains a clustering structure, whereas the matrix $X$ considered in \cite{yang2016} does not have a structured right singular subspace. 

\end{remark}

\subsection{Minimax upper bound and a constrained maximum likelihood estimator}\label{ssec:minimax_upper}

From the frequentist perspective, an ideal method for 
parameter estimation in a well-specified statistical model is the maximum likelihood estimator (MLE). In this subsection, we propose a constrained MLE for estimating the mean matrix $(\bmu^*)(\bl^*)^T$.  We prove that the risk bound of this estimator achieves the minimax lower bound, thereby showing that the minimax lower bound coincides with the minimax risk modulus a multiplicative constant. 

Assuming the number of clusters 
$K$ is known, 
we consider the parameter space
$\Theta_K=
\{(\bmu,\bl):\bmu\in\mathbb{R}^{p\times K}, \bl\in\mathcal{L}_K,|\mathrm{supp}(\bmu)|\leq s
\} $
and define the following constrained MLE
\begin{align}\label{equ3}
    (\widehat{\bmu}, \widehat{\bl}) = \argmin_{(\bmu, \bl)\in \Theta_K}\|\by - \bmu\bl^T\|_F^2.
\end{align}
It is worth noting that the parameter space in the constrained MLE is not necessarily compact. However, by characterizing a compact neighbor of $(\bmu^*)(\bl^*)^T$ and controlling the complexity inside and outside this neighbor separately, we can establish the risk bound of the constrained MLE and show that it achieves the minimax lower bound in Theorem \ref{thm2}.
\begin{theorem}\label{thm3}
Suppose that $(\bmu^*,\bl^*,\bSigma^*)\in\Theta_K^*$ and  $(\widehat{\bmu}, \widehat{\bl})$ is defined as in (\ref{equ3}) and Assumptions \ref{assump1}-\ref{assump4} hold. Then there exists some constant $c > 0$ such that 
\begin{align*}
    \sup_{(\bmu^*,\bl^*,\bSigma^*)\in \Theta_K^*}\E_*\left\{\|(\bmu^*)(\bl^*)^T -  \widehat{\bmu}\widehat{\bl}^T\|^2_F\right\} \leq c\left(n\log K + s\log p\right),
\end{align*}
where $\E_*$ denotes the expected value with respect to $(\bmu^*,\bl^*,\bSigma^*)$.
\end{theorem}

The proof of Theorem \ref{thm3} relies on transforming the upper bound of the convergence rate into an empirical process on $\widetilde{\Theta}_K$, which is the normalization of the feasible set of the estimator $\Theta_K$. This poses challenges due to the infinite entropy of $\Theta_K$. To overcome this, we decompose $\widetilde{\Theta}_K$ into layers $(\mathcal{E}_j^K)_{j\in\mathbb{Z}}$ with manageable entropy. We classify these layers into three scenarios: small volume, large volume, and intermediate volume. For the small volume case, we establish that $\bl$ equals $\bl^*$ up to a permutation within $\mathcal{E}_j^K$, simplifying the complexity to focus on $\bmu$ alone. In the large volume case, we control the entropy separately for the spaces of $\bmu$ and $\bl$. Lastly, in the intermediate volume case, we leverage the ellipsoidal nature of $\mathcal{E}_j^K$ and its packing to bound the covering numbers efficiently.

Despite the theoretical optimality of the constrained MLE, it is computationally intractable in general since the feasible set $\Theta_K$ is nonconvex and involves discrete structures. 
In addition, the implementation of the constrained MLE requires to pre-specify the sparsity level $s$ and the number of clusters $K$, which are usually unknown in practice.
These computational challenges motivate us to develop a Bayesian method that
can be implemented conveniently using an MCMC sampler without specifying $s$ and $K$ {\it a priori}. 

\subsection{Bayesian sparse high-dimensional Gaussian mixture model}\label{ssec:prior}


As described in the previous subsection, the optimization-based constrained MLE is computationally intractable due to the non-convexity and discrete structure of the problem. One may apply the EM algorithm, which iterates between a clustering step given the recently updated parameter values and a parameter estimation step given the recently updated cluster memberships, to address this issue. For example,  \citet{caichime2019} proposed an approach that estimates the sparse discriminant vector and obtains the clustering memberships in the Expectation step to avoid the singularities of sample covariance matrices in high dimensions.
Another approach is spectral clustering \citep{luxburg2004spectralclustering}. However, the optimality of spectral clustering is only established when $p=o(n\Delta)$ without sparsity structure \citep{loffler2020optimality}. In this subsection, we propose a Bayesian approach to estimate the high-dimensional sparse Gaussian mixture model. As will be seen later, the proposed Bayesian method has a minimax-optimal posterior contraction rate.

We deliberately consider a misspecified sampling model: $\by_i = \bmu_{z_i} + \bm{\epsilon}_i$, where the error term $\bm{\epsilon}_i$ follows a multivariate normal distribution with mean vector zero and identity covariance matrix $\bI_p$. This intentional simplification is motivated by theoretical convenience, since our primary focus lies in the mean matrix $\bmu\bl^T$ . As we will prove later, as long as the spectrum of the true covariance matrix $\bSigma^*$ is bounded, the posterior distribution of $\bmu\bl$ concentrates on the true parameter $(\bmu^*)(\bl^*)^T$ at a minimax-optimal rate.

To promote sparsity, we use the spike-and-slab LASSO prior \citep{rockova2018} for the mean vectors of clusters. 
The spike-and-slab LASSO prior can be viewed as a continuous relaxation of the spike-and-slab prior \citep{mitchell1988}, which is a mixture of a point mass at zero (referred to as the ``spike'' distribution) and an absolutely continuous distribution (referred to as the ``slab'' distribution).
Formally, for $\bm{x}\in\mathbb{R}^{p}$, the spike-and-slab LASSO prior is defined as follows: for $j\in[p]$,
$\pi(x_j\mid \lambda_0,\lambda_1,\xi_j)=(1-\xi_j)\psi(x_j\mid\lambda_0)+\xi_j\psi(x_j\mid \lambda_1)$ and $(\xi_j\mid\theta)\sim\text{Bernoulli}(\theta),$
where $\psi(x\mid\lambda)=(\lambda/2)\exp(-\lambda |x|)$ is the density of Laplace distribution with mean $0$ and variance $2/\lambda^2$. By assuming $\lambda_0\gg \lambda_1$,   
$\psi(x_j\mid\lambda_0)$ closely resembles the ``spike" distribution in the spike-and-slab prior since it is highly concentrated at $0$, whereas 
$\psi(x_j\mid \lambda_1)$ plays the role of the ``slab" distribution. We follow the notation in \citet{rockova2018} and use $\text{SSL}(\lambda_0,\lambda_1,\theta)$ to denote this prior model. In the context of our proposed sparse Gaussian mixture model, we define the joint-SSL($\lambda_0$,$\lambda_1$,$\theta$) as follows to further incorporate the case where the mean vectors 
$\bmu_1,\dots,\bmu_K\in\mathbb{R}^p$ share the same sparsity pattern: given $K$, for $j\in[p]$,
\begin{align*}  \pi(\mu_{1j},\cdots,\mu_{Kj}\mid\lambda_0,\lambda_1,\xi_j)&=\prod_{k=1}^K\left((1-\xi_j)\psi(\mu_{kj}\mid\lambda_0)+\xi_j\psi(\mu_{kj}\mid\lambda_1)\right),\\
    (\xi_j\mid\theta)&\sim\text{Bernoulli}(\theta).
\end{align*}
Under this prior distribution, the random vectors 
$\bmu_1,\dots,\bmu_K$ are conditionally independent given $K$ and a sparsity indicator vector $\bm{\xi}\in\{0,1\}^p$ which controls the common sparsity structure. We further assume that $\theta\sim\mathrm{Beta}(1, \beta_\theta)$, where $\beta_\theta = p^{1 + \kappa}\log p$ for some constant $\kappa > 0$. The choice of the hyperparameter $\beta_\theta$ is selected for technical reasons.

We now specify the sparse Gaussian mixture model. 
Given $K$, for cluster membership indicators $z_1,\dots,z_n$, we assign a categorical prior with a $K$-dimensional probability vector $\bw = (w_1,\ldots,w_K)^T$, whose hyperprior distribution is a $K$-dimensional symmetric Dirichlet distribution with the shape parameter $\alpha > 0$.
We assign a joint-SSL prior for the mean vectors $\{\bmu_k\}_{k=1}^K$ to adapt to the joint sparsity. 
To allow for an unknown $K$, we further assign a truncated Poisson distribution to $K$ by letting $\pi(K) \propto e^{-\lambda}\lambda^K/K!$, $K \in [K_{\max}]$, where $K_{\max}$ is a conservative upper bound for $K$ and should be large enough in practice.
Thus, the proposed Bayesian sparse Gaussian mixture model can be expressed as follows: 
\begin{align}
    (\by_1,\ldots,\by_n\mid\bz,\bmu)&\sim N_p(\bmu_k,\bI_p)\quad\mbox{independently}, \label{prior1}\\
    (\bmu_1,\dots,\bmu_K\mid K, \theta)&\sim \text{joint-SSL}(\lambda_0,\lambda_1,\theta), \label{prior2}\\
    (z_1,\ldots,z_n\mid \bw, K)&\sim\mathrm{Categorical}(\bw)\quad\mbox{independently}, \label{prior3}\\
    (\bw\mid K)&\sim\text{Dirichlet}_K(\alpha),\label{prior4}\\ 
    \pi(K) &\propto \frac{e^{-\lambda}\lambda^k}{k!},\quad K\in [K_{\max}]\label{prior5}\\
    \theta &\sim \text{Beta}(1,\beta_\theta) \text{ where } \beta_\theta=p^{1+\kappa}\log p. \label{prior6}
\end{align}


The use of sparsity-enforcing priors in Gaussian mixture models has been widely explored in Bayesian literature. For example, \cite{tadesse2005} and \cite{gao2020general} proposed a discrete subset-selection prior for clustering high-dimensional data. In this paper, we employ a continuous spike-and-slab shrinkage prior. While \cite{tadesse2005} focused solely on computational algorithms without theoretical analysis, and \cite{gao2020general} primarily investigated the theoretical results of parameter estimation, our main contribution lies in establishing the theoretical properties of the proposed model concerning both parameter estimation and mis-clustering error. This represents the first effort in developing a general theoretical framework for Bayesian analyses of high-dimensional clustering.


\section{Theoretical Properties}\label{sec:theory}

\subsection{Posterior contraction rate}\label{ssec:posterior}



In this subsection, 
we show that the posterior contraction rate with respect to the Frobenius norm metric is minimax optimal under the propose Bayesian sparse Gaussian mixture model. All the proofs are deferred to the Appendix. 

By the Bayes formula, the posterior distribution of $\bmu$ and $\bl$ can be written as
\[
\Pi\{(\bmu, \bl)\in \mathcal{E}\mid\by) = \frac{\int_{\mathcal{E}}p_n(\by\mid\bmu, \bl)/p_n(\by\mid\bmu^*, \bl^*)\Pi(d\bmu d\bl)}{\int_\Theta p_n(\by\mid\bmu, \bl)/p_n(\by\mid\bmu^*, \bl^*)\Pi(d\bmu d\bl)},
\]
where $p_n(\by\mid\bmu, \bl) = (2\pi)^{-np/2}\exp\left(-\|\by-\bmu \bl^T\|_F^2/2\right)$
is the likelihood of the data matrix $\by$ with identity covariance matrix and $\mathcal{E}$ is any measurable subset of $\Theta = \bigcup_{K = 1}^{K_{\max}}\mathbb{R}^{p\times K}\times\mathcal{L}_K$. In Theorem \ref{thm1}, we derive the posterior contraction rate under the proposed Bayesian model.

\begin{theorem}\label{thm1}
Let $\by$ be generated from a mixture of $K^*$ Gaussian distributions as in (\ref{model}) with the true mean vectors $\bmu^* = [\bmu_1^*,\ldots,\bmu_{K^*}^*]$ and the true cluster membership matrix $\bl^*$, where $|\mathrm{supp}(\bmu^*)|\leq s$. Suppose Assumptions \ref{assump1} - \ref{assump4} hold. Let $\bmu$ and $\bl$ follow the prior specification in (\ref{prior1})-(\ref{prior6}) with some hyperparameters $\kappa>0$, $\alpha\geq 1$, 
$\lambda_0\geq 2\log (p/s)\sqrt{np/(s\log p)}$
and $1/n^{\gamma} \lesssim\lambda_1\lesssim \sqrt{s\log p/\|\bmu^*\|_F} $ for some constants $\gamma>0$. 
Then for $n$ sufficiently large, we have
$$\Pi\left\{(\bmu, \bl):\|\bmu \bl^T-(\bmu^*)(\bl^*)^T\|_F^2\geq M(s\log p+n\log K_{\max})\bigg\rvert \by\right\}\to 0$$
in $\mathbb{P}_{(\bmu^*, \bl^*, \bSigma^*)}$- probability, for every large constant $M$ and $(\bmu^*, \bl^*, \bSigma^*)\in \Theta^*_K$.
\end{theorem}

The proof of Theorem \ref{thm1} adopts a modified ``testing-and-prior-concentration" approach \citep{ghosal2000}, tailored to address the unique challenges arising from model misspecification in our setting. We rigorously prove the three conditions: (1) The prior distribution puts a sufficient mass on the neighbourhood of the true parameter $(\bmu^*)(\bl^*)^T$; (2) There exists a test function which can distinguish $(\bmu^*)(\bl^*)^T$ from the complement of its neighbourhood in a subset of the parameter space; (3) The prior puts almost all mass on the subset of parameter space in condition (2).
 



\begin{remark}
Recall from Section \ref{sec:model} that the minimax lower bound contains the true number of clusters $K^*$, which is unknown in many applications. The posterior contraction rate obtained in Theorem \ref{thm1} contains a logarithmic factor of the upper bound $K_{\max}$ for $K^*$. If we further assume that $K_{\max} \asymp (K^*)^q$ for any constant $q\geq 1$, the posterior contraction rate matches the minimax lower bound in Theorem \ref{thm2} and is optimal thereafter. For $\lambda_1$ in the  joint-SSL prior, if we further assume $\|\bmu_k^*\|_2^2=O(s)$ for any $k\in[K^*]$, then the upper bound of $\lambda_1$ can be relaxed to $\lambda_1\lesssim \sqrt{\log p/K_{\max}}$, which is a mild condition and can be easily satisfied in practice.

\end{remark}
We assume that the cluster mean vectors are jointly sparse. However,
Theorem \ref{thm1} can be easily generalized to the case where the cluster centers do not share the common sparsity structure. Specifically, each mean vector $\bmu_k$ has at most $s$ non-zero coordinates but the indices of the non-zero coordinates are not necessarily the same across $k\in [K]$. Clearly, the matrix $\bmu = [\bmu_1,\ldots,\bmu_K]$ is jointly $Ks$-sparse. To adapt to the column-wise sparsity of $\bmu$, we modify the hierarchical prior model by letting $\bmu_1,\dots,\bmu_K$ follow the SSL prior independently given $K$: 
\begin{align}\tag{6'}\label{prior2'}
    (\bmu_k\mid K,\theta)\sim\text{SSL}(\lambda_0,\lambda_1,\theta)\quad\text{for } k=1,\dots,K.
\end{align}
The following corollary gives the posterior contraction rate under such a modification.

\begin{corollary}\label{cor1}
Let $\by$ be generated from a mixture of $K^*$ Gaussian distributions as in (\ref{model}) with the true mean vectors $\bmu^* = [\bmu_1^*,\ldots,\bmu_{K^*}^*]^T$ and the true cluster membership matrix $\bl^*$, where $|\mathrm{supp}(\bmu^*)|\leq K^*s$. 
Suppose Assumptions \ref{assump1}-\ref{assump4} hold.
Let $\bmu$ and $\bl$ follow the prior specification in (\ref{prior1}), \eqref{prior2'}, \eqref{prior3}-(\ref{prior6}) with the same hyperparameters as in Theorem \ref{thm1}.
Then for $n$ sufficiently large, we have
\[
\Pi\left\{(\bmu, \bl):\|\bmu \bl^T-(\bmu^*)(\bl^*)^T\|_F^2\geq M(sK_{\max}\log p+n\log K_{\max})\bigg\rvert \by\right\}\to 0
\]
in $\mathbb{P}_{(\bmu^*, \bl^*, \bSigma^*)}$- probability, for every large constant $M$ and $(\bmu^*, \bl^*, \bSigma^*)\in \Theta^*_K$.
\end{corollary}

\begin{remark}
Corollary \ref{cor1} can be easily extended to the case when the mean vectors have not only different sparsity structures, but also distinct sparsity sizes, i.e., $|\text{supp}(\bmu_k)|\neq |\text{supp}(\bmu_{k'})|$ for some $k\neq k'$. In such a case, the first term of the posterior contraction rate becomes to $\sum_{k=1}^{K^*}|\text{supp}(\bmu_k)|K_{\max}\log p$.
\end{remark}


\subsection{Mis-clustering error}\label{ssec:clustering_error}


Recovering cluster memberships is always a focal objective for clustering problems. In this subsection, we obtain a mis-clustering error bound of the proposed Bayesian model \emph{a posteriori} based on the posterior contraction result for parameter estimation in Theorem \ref{thm1}.
For any two cluster membership vectors $\bz,\bz'\in([K])^n$, define the minimum Hamming distance
$    d_{H}(\bz,\bz')=(1/n)\inf_{\tau\in S_K}\sum_{i=1}^n\mathbbm{1}\{z_i\neq \tau(z_i')\}$
as the mis-clustering rate between $\bz$ and $\bz'$, where $S_K$ is the set of all permutations on $[K]$. Let $\sigma_{\max}(\bX)$ and $\sigma_{\min}(\bX)$ denote the largest and smallest 
 
non-zero 

singular value of matrix $\bX$, respectively. Below, we obtain the posterior contraction result for the mis-clustering error measured by $d_{H}$.

\begin{theorem}\label{thm:cluster_accuracy}
Assume the conditions in Theorem \ref{thm1} hold and $n_k^*=\sum_{i=1}^n\mathbbm{1}(z_i^*=k)\to\infty$ for all $k$. Let $n_{\min}^* \overset{\Delta}{=}\min_{k\in [K]}n^*_k$ and $n_{\max}^*\overset{\Delta} = \max_{k\in [K]}n_k^*$.
Then we have 
\begin{align*}
    \Pi\left\{nd_{H}(\bz,\bz^*)\geq\frac{M(n_{\max}^*)^3\sigma_{\max}(\bmu^*)^2}{(n_{\min}^*)^4\sigma_{\min}(\bmu^*)^4}(s\log p+n\log K_{\max})\bigg\rvert\by\right\}\to 0
\end{align*}
in $\mathbb{P}_{(\bmu^*, \bl^*, \bSigma^*)}$- probability for every large constant $M$.
\end{theorem}

The main challenge in proving Theorem \ref{thm:cluster_accuracy} is translating clustering accuracy into errors regarding the right singular subspace of $\bmu\bl^T$. By leveraging a variant of the Davis-Kahan Theorem, we bound the distance of the right singular subspace of $\bmu\bl^T$ using the distance of $\bmu\bl^T$. Through geometric analysis, we establish that correct clustering is achievable if the errors of corresponding right singular vectors are sufficiently small. This allows us to control mis-clustering error by the error of $\bmu\bl^T$, and the desired result follows by directly applying Theorem \ref{thm1}.

\begin{remark}
If we assume $(s\log p)/(n\|\bmu^*\|_F^2)\to 0$, then by Theorem \ref{thm:cluster_accuracy}, the proportion of the mis-clustered data points is asymptotically negligible with a high posterior probability provided that 
$
(n_{\max}^*)^3\sigma_{\max}(\bmu^*)^{4} = O({(n_{\min}^*)^4\sigma_{\min}(\bmu^*)^4})
$
as $n\to\infty$. Moreover, if we further assume that $\sigma_{\max}(\bmu^*)\lesssim\sigma_{\min}(\bmu^*)$ 
and $n_{\max}^*\lesssim n_{\min}^*=O(n)$ (which means that the sizes of the smallest cluster and the largest cluster are of the same order as $n$), then the number of mis-clustered data points, i.e., $nd_{H}(\bz,\bz^*)$, is asymptotically bounded by a constant with a high posterior probability because $\|\bmu^*\|_F^2\asymp\sigma_{\min}(\bmu^*)^2$ in this case. 
\end{remark}

\begin{remark}\label{remark:clustering_accuracy}
%
%
\citet{azizyan2013} and \citet{caichime2019} also studied high-dimensional clustering with the sparsity assumption. However, they only considered the case when the number of clusters was 2. 
Assuming that $\bmu^*_1 - \bmu^*_2$ was sparse, \citet{azizyan2013} showed that the minimax optimal convergence rate of mis-clustering was $\sqrt{s\log p/n}/\Delta^2$ when the two clusters had same mixing weights and isotropic covariance matrices. 
Assuming that the discriminant direction vector $\beta^* = (\bm{\Sigma}^*)^{-1}(\bmu_1^*-\bmu^*_2)$ was sparse, 
\citet{caichime2019} showed that the convergence rate of the excess mis-clustering error
, defined as the difference between the mis-clustering error and the optimal mis-classification error obtained by Fisher's linear discriminant rule when cluster-specific parameters were known,
achieved the minimax optimal rate of $s\log p / n$. 
However, the convergence rate of mis-clustering error was not investigated. 
In addition, \cite{li2017minimax}, \citet{azizyan2013} and \citet{caichime2019} focused on the two-cluster problem, but the minimax optimal result for high-dimensional sparse clustering with $K>2$ clusters was not studied.  In contrast, we allow $K$ to
grow moderately with the sample size $n$. 


\end{remark}

\section{Simulation Studies}\label{sec:simu}
We evaluate the empirical performance of the proposed Bayesian method for sparse Gaussian mixtures through analyses of synthetic data sets. Posterior inference is carried out through an MCMC sampler, the details of which are provided in Appendix \ref{sec:sampling}. 
We also compare the performance of our model with four competitors: 
principal component analysis K-means (PCA-KM), sparse K-means (SKM) \citep{witten2010}, Clustering of High-dimensional Gaussian Mixtures with the EM (CHIME), and Gaussian-mixture-model-based clustering (MClust) \citep{Fraley2012mclustV4}.
In particular, PCA-KM is a two-stage approach that first performs a PCA to reduce dimensionality and then applies a K-means algorithm to the principal components. SKM is a generalization of the K-means in high dimensions to find clusters and important features (i.e., the non-zero coordinates) simultaneously. 
CHIME is a high-dimensional clustering approach based on an EM algorithm. To overcome the issue that the sample covariance matrix may not be invertible and thus the subsequent estimate of $\bm{z}$ is not available, CHIME focuses on the so-called sparse discriminant vector and directly use it in the Fisher discriminant rule to estimate cluster memberships.
Note that the performance of CHIME is quite sensitive to the choice of initial values.
Throughout simulation examples in this section, we set the initial values of CHIME to be the output of K-means. For PCA-KM and SKM, we choose the number of clusters via Silhouette method \citep{ROUSSEEUW198753}, with the range of $K$ being from 2 to 10. For MClust and CHIME, the number of clusters is estimated via Bayesian information criterion (BIC). 

\subsection{Simulation setup}\label{ssec:simu_setup}
   
We consider three simulation scenarios.  Scenario I is designed to evaluate the proposed Bayesian method in terms of clustering accuracy with varying numbers of clusters and support sizes of the mean vectors.
   The data matrix $\by$ is of size $p\times n$ with $p = 400$ and $n = 200$.  We assume that the true number of clusters $K^*$ ranges over $\{3, 5\}$ and the support size $s$ ranges over \{6,12\}.
   We use $\mathcal{S}$ to denote the set of non-zero coordinates and let the first $s$ coordinates of the cluster means be non-zero, i.e., $\mathcal{S}=\{1, 2, \dots, s\}$.
   For each $K^*\in\{3,5\}$,
   the true cluster assignment $z_i^*$ is generated from a categorical distribution: $z^*_i\sim \mathrm{Cat}(\bp_{K^*})$, 
   where 
   $\bp_3=(0.3,0.3,0.4)$ and $\bp_5=(0.2,0.2,0.2,0.2,0.2)$.
   When $K^*=3$, the three cluster mean vectors are $(\bmu_1^*)_\mathcal{S} = 3\times(1,1,\dots,1)^T$, $(\bmu_2^*)_\mathcal{S}=-1.5\times(1,1,\dots,1)^T$ and $(\bmu_3^*)_\mathcal{S} = (0,\dots,0)^T$, where  $(\bmu_1^*)_{\mathcal{S}}, (\bmu_2^*)_{\mathcal{S}}, (\bmu_3^*)_{\mathcal{S}}\in\mathbb{R}^s$. For $K^* = 5$, the five cluster mean vectors are $(\bmu_1^*)_\mathcal{S} = 4\times (1,1,\dots,1)^T$, $(\bmu_2^*)_\mathcal{S}=-4\times(1,1,\dots,1)^T$, $(\bmu_3^*)_\mathcal{S}=(0,\dots,0)^T$, $(\bmu_4^*)_\mathcal{S}=4\times(-1,1,-1,1,\dots,-1,1)^T$ and $(\bmu_5^*)_\mathcal{S}=1.5\times (1,-1,1,-1,\dots,1,-1)^T$. 
    Given $K^*$ and $z_i^*$'s, data are generated from $\by_i\sim N(\bmu^*_{z_i^*}, \bI_p)$.

  Scenario II focuses on the case when small clusters exist.
  The data matrix $\by$ consists of $n = 200$ observations  with dimension $p = 400$. 
  We assume that the true number of clusters $K^*=3$ and 
  the support size  $s=8$.
  Similarly as Scenario I, we set $\mathcal{S}=\{1, 2, \dots, s\}$.  
  The mean vectors over the support $\mathcal{S}$ in the three clusters are $(\bmu_1^*)_{\mathcal{S}}=(5,2,\cdots,5,2)^T$, 
    $(\bmu_2^*)_{\mathcal{S}}=(10,5,\cdots,10,5)^T$, and
    $(\bmu_3^*)_{\mathcal{S}}=(15,2,\cdots,15,2)^T$, respectively.  
 For each observation $i$, its simulated true cluster assignment $z^*_i$ is generated from a categorical distribution independently: $z^*_i\sim \mathrm{Cat}(0.02,0.48,0.5)$.  
  Given $K^*$ and $z_i^*$'s, data are generated from $\by_i\sim N(\bmu^*_{z_i^*}, \bm{\Sigma}_{z_i^*})$, where $\bm{\Sigma}_1 = \bm{\Sigma}_3 = \bI_p$, and $\bm{\Sigma}_2$ is a diagonal matrix whose diagonal entries equal 4 on the coordinates in the support $\mathcal{S}$ and 1 elsewhere.  

  Scenario III aims to investigate the robustness of the proposed Bayesian method to the misspecification of the sampling distribution. 
  The true distribution of the data is assumed to be  a  mixture of multivariate $t$ distributions but we use the Gaussian mixtures as the working likelihood. The data matrix $\by$ consists of $n = 200$ observations of multivariate $t$-mixtures with  dimension $p = 400$ and a degree of freedom $5$. The number of clusters is set to $K^* = 3$,
  and the first $s=8$ coordinates of cluster mean vectors are non-zero. We generate cluster assignments $z^*_i$'s from a categorical distribution: $z^*_i\sim \text{Cat}(0.2,0.4,0.4)$ independently for $i\in [n]$, and let the cluster mean vectors $\bmu^*_1, \bmu^*_2, \bmu^*_3$ and the covariance matrices $\bm{\Sigma}_1,\bm{\Sigma}_2,\bm{\Sigma}_3$ be the same as those in Scenario II for each multivariate $t$-cluster. 


For each of the three scenarios, 
we apply the proposed Bayesian sparse Gaussian mixture model to the simulated data with 100 repeated simulations.
In each simulation, we compute posterior inference using the developed MCMC sampler with $1000$ burn-in iterations and another $4000$ iterations for post-burn-in samples. 
The upper bound of the number of clusters is set to be $K_{\max} = 20$. We set the hyperparameters $\kappa$, $\lambda_0$, and $\lambda_1$ in the  spike-and-slab LASSO prior to be $0.1$, $100$, and $1$ respectively, and $\lambda$ in the truncated Poisson prior for $K$ to be 2.
The estimated number of clusters and cluster assignments under the proposed Bayesian method are reported based on the posterior mode of $z_i$'s from post-burn-in MCMC samples. 
 
The proposed Bayesian method, Mclust, PCA-KM, and SKM are performed under R with version 4.2.1 and CHIME is performed under Matlab with version 9.11 (R2021b).
 

\subsection{Simulation results}
\label{sub:simulation_results}

We first investigate the performance of the proposed Bayesian sparse Gaussian mixture model against the four competitors in Scenario I. We focus on the following three objectives: identification of the number of clusters, the clustering accuracy, and the cluster-wise mean vector estimation accuracy. 
The proposed Bayesian method can successfully recover the simulated true number of clusters. Specifically, when $K^*=3$, the proposed method identifies 3 clusters in 85 replicates out of 100 replicates for $s=6$ and in 98 replicates for $s=12$; when $K^*=5$, the proposed method identifies 5 clusters in 83 replicates out of 100 replicates for $s=6$ and in 98 replicates for $s=12$. 
In contrast, all the four competitors underestimate the number of clusters. In particular, when $K^*=3$, the estimated number of clusters using the four competitors all equal to 2 in 100 simulation replicates. When $K^*=5$, PCA-KM, SKM, MClust, and CHIME only correctly estimate the number of clusters in 6, 0, 4, and 3 out of 100 replicates for $s=6$, and in 8, 0, 9, and 3 out of 100 replicates for $s=12$.
Figure \ref{fig:clustering_k3s6} and Appendix Figure \ref{fig:clustering_k5s6} plot the simulated true cluster memberships and the estimated clustering results under the proposed Bayesian method and the four competitors for one randomly selected simulation replicate when $K^*=3$, $s=6$, and $K^*=5$, $s=6$, respectively. 
We can see that the four competitors cannot well distinguish clusters with a certain degree of overlapping, e.g., the green and blue clusters in the upper left panel of Figure  \ref{fig:clustering_k3s6}, while the proposed Bayesian method can successfully separate them.   
 \begin{figure}[htp]
    \centering
    \includegraphics[width=.75\textwidth]{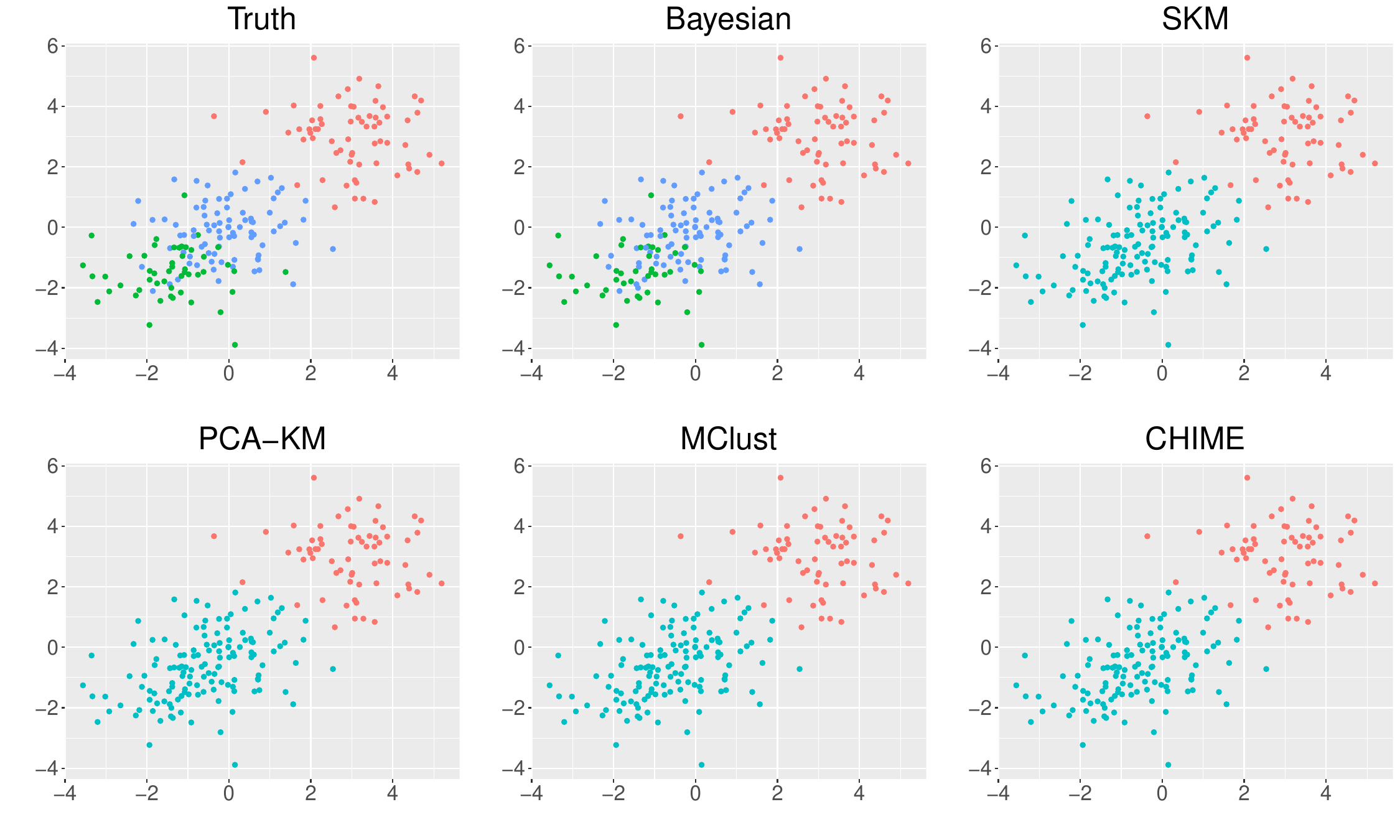}
    \caption{Clustering results of different methods compared to the true cluster memberships in Scenario I with $K^*=3$ and $s=6$ in a randomly selected simulation replicate. Data points are projected onto the subspace of the first two coordinates and different colors correspond to different estimated cluster memberships.}
    \label{fig:clustering_k3s6}
\end{figure}

In terms of clustering accuracy, we use the adjusted Rand index (ARI) \citep{rand1971} as the evaluation metric.
Specifically, let $\mathcal{C}^*$ and $\mathcal{C}^\dagger$ denote the true and estimated partition of $[n]$, respectively, and let $K^*\overset{\Delta}{=}|\mathcal{C}^*|$, $K^\dagger \overset{\Delta}{=} |\mathcal{C}^\dagger|$.
For each cluster, denote $n^*_k$ the size of $k$-th cluster in $\mathcal{C}^*$ and $n^\dagger_{k'}$ as the size of $k'$-th cluster in $\mathcal{C}^\dagger$. Furthermore, let $n_{k,k'}$ be the number of observations that are assigned to both the $k$-th cluster in $\mathcal{C}^*$ and $k'$-th cluster in $\mathcal{C}^\dagger$. Then the ARI is defined as 
$$\mathrm{ARI} (\mathcal{C}^*,\mathcal{C}^\dagger)=\frac{\sum_{k=1}^{K^*}\sum_{k'=1}^{K^\dagger}{n_{k,k'}\choose 2}-\sum_{k=1}^{K^*}{n_k\choose 2}\sum_{k'=1}^{K^\dagger}{n_{k'}\choose 2}/{n\choose 2}}{\left(\sum_{k=1}^{K_t}{n_k\choose 2}+\sum_{k'=1}^{K^\dagger}{n_{k'}\choose 2}\right)/2-\sum_{k=1}^{K^*}{n_k\choose 2}\sum_{k'=1}^{K^\dagger}{n_{k'}\choose 2}/{n\choose 2}}.$$
Table \ref{tab:simu} reports the average ARIs of the clustering results of the proposed Bayesian method against the four competitors across 100 simulation replicates under Scenario I. 
The proposed Bayesian method outperforms the four alternatives in terms of higher ARIs in all settings. 

\begin{table}[tbp]
    \centering
    \begin{tabular}{c|c c c c}
         \hline
         \hline
         & \multicolumn{2}{c}{$K^*=3$} & \multicolumn{2}{c}{$K^*=5$} \\
         & $s=6$ & $s=12$ & $s=6$ & $s=12$\\
        \hline Bayesian & 0.84 (0.19) & 0.98 (0.01) & 0.94 (0.03) & 0.99 (0.01) \\
        \hline PCA-KM & 0.54 (0.04) & 0.55 (0.04) & 0.64 (0.17) & 0.60 (0.18)\\
        \hline MClust & 0.54 (0.04) & 0.55 (0.04) & 0.81 (0.13) & 0.78 (0.05) \\
        \hline SKM & 0.55 (0.04) & 0.55 (0.04) & 0.54 (0.21) & 0.74 (0.13)\\
        \hline CHIME & 0.53 (0.10) & 0.63 (0.18) & 0.52 (0.27)  & 0.54 (0.29) \\
        \hline
        \hline
    \end{tabular}
    \caption{Average (standard deviation) adjusted Rand indices (ARIs) in Scenario I with different choice of the number of clusters $K^*$ and support size $s$.}
    \label{tab:simu}
\end{table}

We then examine the cluster-wise mean vector estimation accuracy by computing 
$\|\hat{\bmu}_1-\bmu^*_1\|_2$ under the proposed Bayesian method and alternatives, where $\hat{\bmu}_1$ is the estimated mean vector under different methods. 
Specifically, $\hat{\bmu}_1$ under the proposed Bayesian method is the posterior mean of $\bmu_1$.  CHIME and MClust directly return $\hat{\bmu}_1$ since they are model-based methods. For PCA-KM and SKM, we use the empirical means induced from their estimated clustering memberships as $\hat{\bmu}_1$ since they are based on K-means.
Figure \ref{fig:sparsity_k3s12_error} presents the boxplots of $\|\hat{\bmu}_1-\bmu^*_1\|_2$ when $K^*=3$ and $s=12$ across 100 simulation replicates under different methods, showing that the proposed Bayesian method yields the smallest error of estimating $\bmu^*_1$.
Furthermore, Appendix Figure \ref{fig:sparsity_k3s12} plots the estimated $\bmu_1^*$ under different methods in one randomly selected simulation replicate, indicating that the proposed Bayesian method recovers the sparsity pattern better than the four competitors.  Lastly, we report the running times of all methods in Scenario I in Appendix Table \ref{tab:running_time}. 
 
\begin{figure}[tbp]
    \centering
    \includegraphics[width=.75\textwidth]{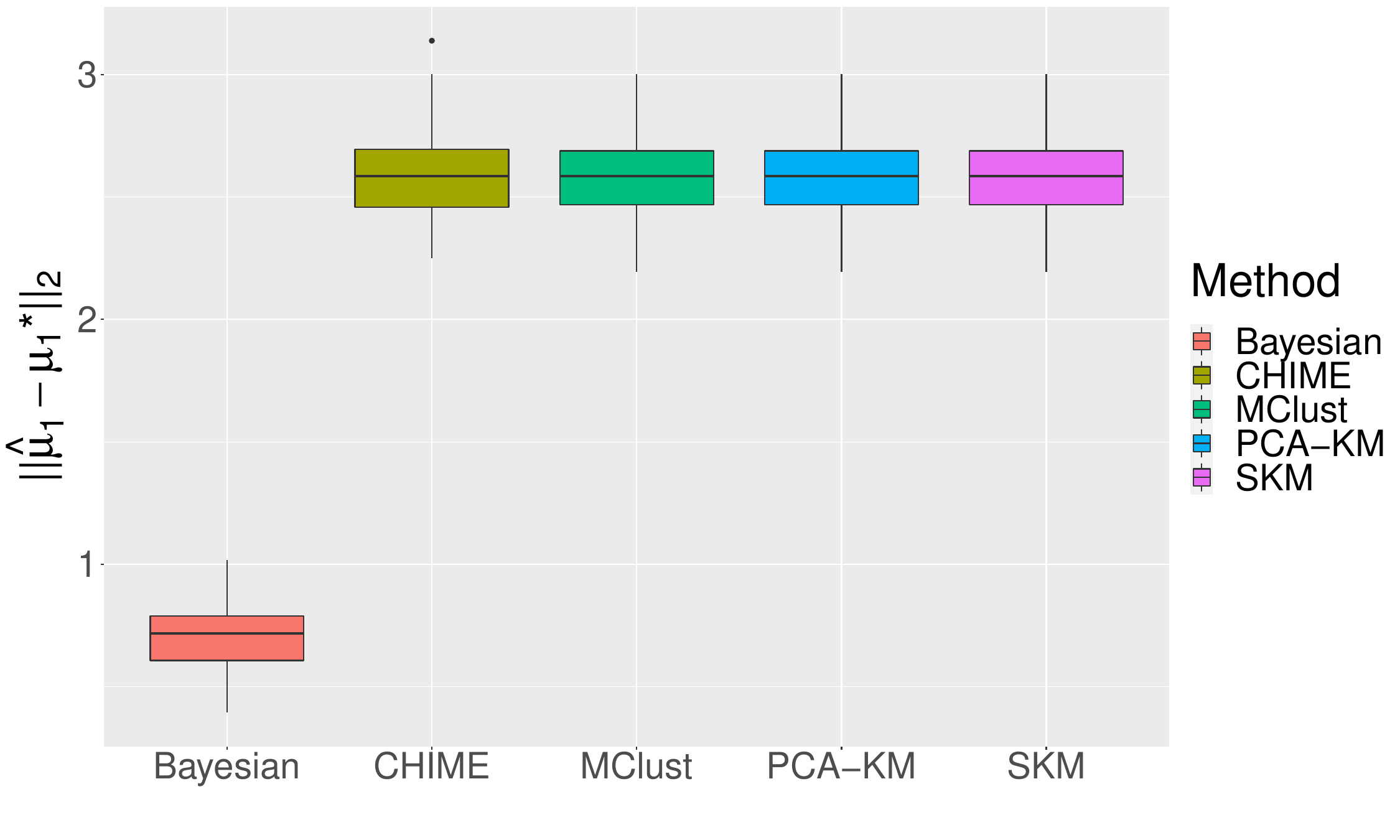}
    \caption{Numerical results of $\|\hat{\bmu}_1-\bmu^*_1\|_2$ of different methods in Scenario I with $K^*=3$ and $s=12$ across 100 simulation replicates.}
    \label{fig:sparsity_k3s12_error}
\end{figure}

Scenario II is designed to evaluate the proposed Bayesian method when small clusters exist. 
\begin{figure}[tbp]
    \centering
    \includegraphics[width=.76\textwidth]{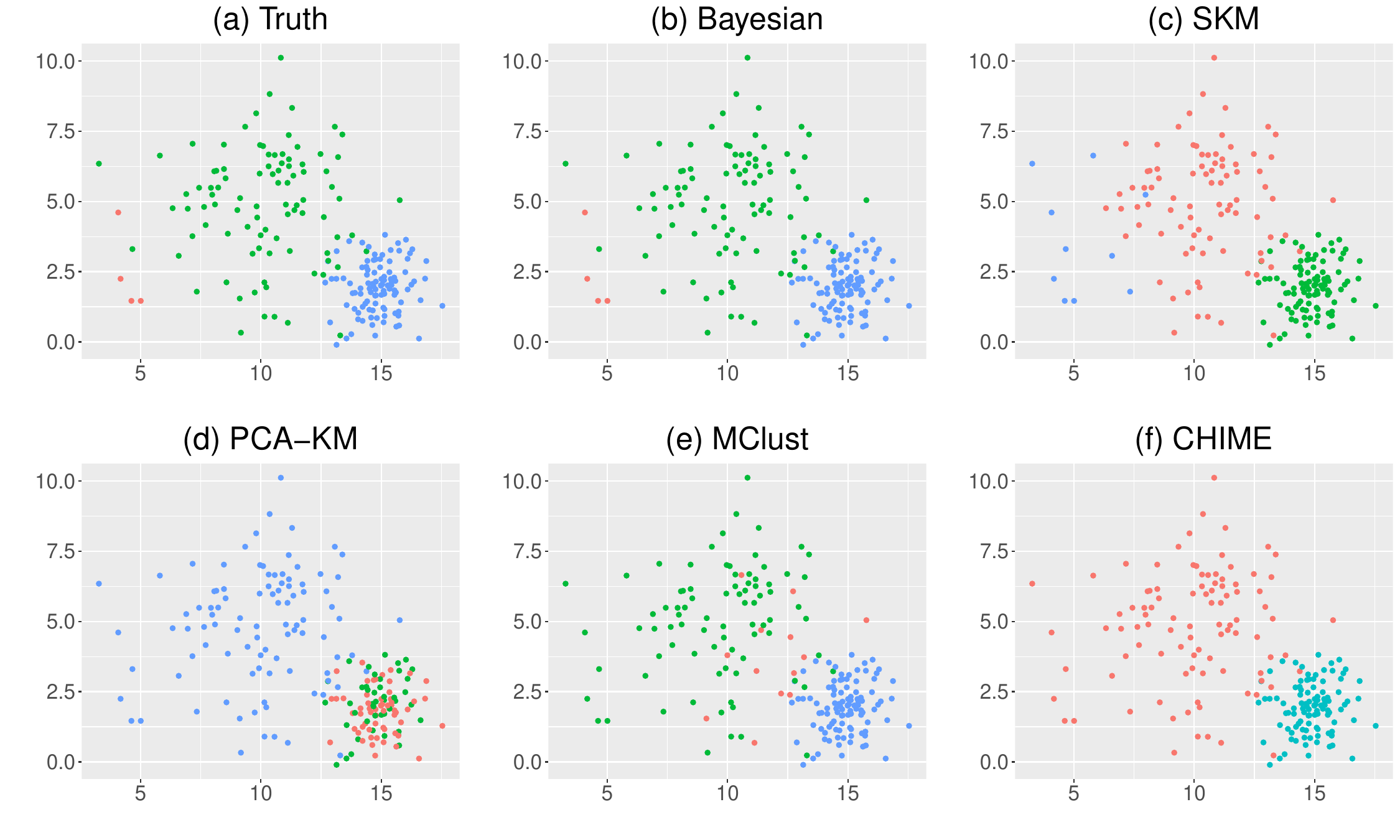}
    \caption{Clustering results of different methods compared with the truth in one randomly selected simulation replicate of Scenario II. (a) Simulated truth. (b) Clustering result under the proposed Bayesian method. (c-f) Clustering results under SKM, PCA-KM, MClust, and CHIME when the number of clusters is fixed to be truth. Data points are projected onto the subspace of the first two coordinates and different colors correspond to different estimated cluster memberships.
    }
    \label{fig:small_cluster_true_k}
\end{figure}
Figure \ref{fig:small_cluster_true_k}(a) shows the true clustering assignments in one randomly selected replicate, in which the small cluster (in red color) only contains four data points.
Our Bayesian method successfully discovers the small cluster and yields the estimated number of clusters $K=3$ in 96 out of 100 simulation replicates, resulting in an average ARI of 0.99. Figure \ref{fig:small_cluster_true_k}(b) shows the estimated clustering memberships under the proposed Bayesian method in the same simulation replicate, exactly matching the truth shown in Figure \ref{fig:small_cluster_true_k}(a). 
In contrast, all four competitors are not able to identify the small cluster and report the estimated number of clusters $K=2$ in all 100 simulation replicates. 
We further examine the performance of the four competitors when we pre-specify the number of clusters to be the truth $K=3$. Figure \ref{fig:small_cluster_true_k}(c) - (f) plot the estimated clustering memberships obtained from the four competitors.
We can see that SKM, PCA-KM, and MClust incline to return clusters with relatively balance sizes, leading to inaccurate clustering assignments with the average ARIs being 0.76, 0.84, and 0.79, respectively, across 100 simulation replicates. CHIME only returns two clusters even though we set the number of clusters to be 3, as shown in Figure \ref{fig:small_cluster_true_k}(f).

For Scenario III, the histograms of the estimated number of clusters under different methods are presented in Appendix Figure \ref{fig:hist_scen3}. 
Figure \ref{fig:mixture_t} visualizes the clustering results under different methods in a randomly selected simulation replicate.
The proposed Bayesian method successfully identifies three clusters in 76 out of 100 simulation replicates, with an average ARI of 0.97 across 100 simulation replicates.
Note that the multivariate $t$-distribution is heavy-tailed. Since we mis-specify the working model as Gaussian mixtures, it is reasonable to treat some observations as ``outliers", as shown in the upper left panel of Figure \ref{fig:mixture_t}.
Therefore, the proposed Bayesian method tends to assign these ``outliers" to small clusters when it overestimates the number of clusters. 
PCA-KM returns 10 clusters in 59 out of 100 simulation replicates with the average ARI being 0.84. 
The estimated numbers of clusters of SKM are all 2 in 100 simulation replicates, and the average ARI is 0.52. For model-based methods, i.e., MClust and CHIME, which also utilize Gaussian mixtures as the working likelihood, their performance are much worse than others as they only identify one cluster in 93 out of 100 replicates, resulting in the average ARIs less than 0.05.
\begin{figure}[htp]
    \centering
    \includegraphics[width = .76\textwidth]{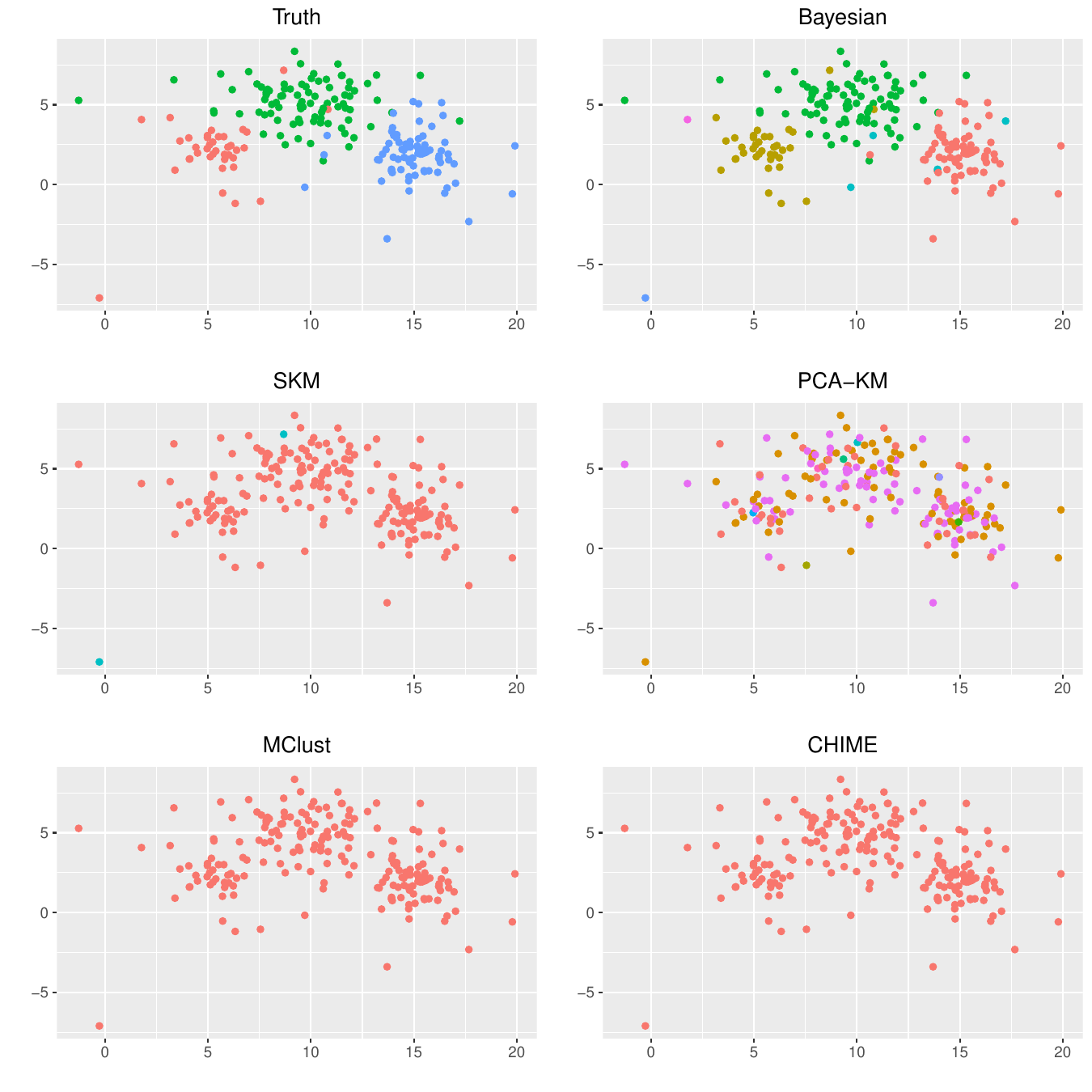}
    \caption{Clustering results under different methods   in one randomly selected replicate of Scenario III. Data points are projected onto the subspace of the first two coordinates and different colors correspond to different estimated cluster memberships.}
    \label{fig:mixture_t}
\end{figure}

\section{Single-cell Sequencing Data Analysis} 
\label{sec:real}

Recent advances in high-throughput single-cell RNA sequencing (scRNA-Seq) technologies greatly enhance our understanding of cell-to-cell heterogeneity and cell lineages trajectories in development \citep{cao2019single}. One important goal of analyzing scRNA-Seq data is to cluster cells to identify cell subpopulations with different functions and gene  expression patterns. The large number of detected genes in scRNA-Seq data makes the task of clustering cells a high-dimensional problem. In this section, we evaluate the proposed Bayesian sparse Gaussian mixture model using a benchmark scRNA-Seq data set \citep{darmanis2015survey}, which is available at the data repository Gene Expression Omnibus (GSE67835, \url{https://www.ncbi.nlm.nih.gov/geo/query/acc.cgi?acc=GSE67835}). 
After excluding hybrid cells and filtering out lowly expressed genes (i.e., the total number of RNA-Seq counts over all non-hybrid cells is less than or equal to 10), we have $p=18568$ genes and $n=420$ cells in $8$ cell types including fetal
quiescent cells (110 cells), fetal replicating cells (25
cells), astrocytes cells (62 cells), neuron cells (131
cells), endothelial (20 cells), oligodendrocyte cells (38 cells), microglia cells (16 cells), and OPCs (16
cells). The original count data $y_{ji}$ for gene $j$ in cell $i$ is transformed into continuous type by taking base-2 logarithm and pseudo count 1, i.e., $\log_2 (y_{ji}+1)$. Then we divide each $y_{ji}$ by the total expression of each cell, i.e., $y_{ji}/\sum_{j=1}^py_{ji}$. Lastly we normalize the data such that the standardized expression levels have zero mean and unit variance for each gene. 

We apply the proposed Bayesian method to the scRNA-Seq data with
the same hyperparameters as in the simulation study. We run the MCMC sampler for 10000 iterations and  discard the first 5000 iterations as burn-in. 
For comparison, we implement several alternatives, including K-means (KM), PCA-KM, MClust, SKM, K-means after non-negative matrix factorization (NMF-KM) \citep{zhu2017detecting}, and K-means after t-distributed stochastic neighbor embedding algorithm (tSNE-KM) \citep{linderman2019fast}.  For PCA-KM and NMF-KM, we first project the data onto the top 10-dimensional feature space, then apply the KM algorithm to cluster the cells. For the KM-based method, the optimal number of clusters is determined by the Silhouette method. 
To measure the performance of clustering results, we use not only the aforementioned ARI but another commonly-used criteria in the single-cell literature: normalized mutual information (NMI) \citep{ghosh2011cluster}. Formally, with the same notations as in Section \ref{sub:simulation_results}, the NMI is defined as
$$\text{NMI}(\mathcal{C}^*,\mathcal{C}^\dagger)=\frac{\sum_{k=1}^{K^*}\sum_{k'=1}^{K^\dagger}\frac{n_{k,k'}}{n}\log\left(\frac{n_{k,k'}}{n}\right)-\sum_{k=1}^{K^*}\frac{n_k^*}{n}-\sum_{k'=1}^{K^\dagger}\frac{n_{k'}^\dagger}{n}\log\left(\frac{n_{k'}^\dagger}{n}\right)}{\sqrt{\sum_{k=1}^{K^*}\frac{n_k^*}{n}\log\left(\frac{n_k^*}{n}\right)\sum_{k=1}^{K^\dagger}\frac{n_{k'}^\dagger}{n}\log\left(\frac{n_{k'}^\dagger}{n}\right)}}.$$
Table \ref{tab:real_data_est_k} reports the ARIs and NMIs under the proposed method and alternatives, showing that the proposed model results in the highest ARI and NMI and achieves the best clustering accuracy. 
Figure \ref{fig:real_data_est_k} plots the true cell types and the estimated cluster memberships under all the methods.
Although the proposed Bayesian method underestimates the number of cell types by 1 and yields $K=7$,  it can identify most cell types except for fetal quiescent and fetal replicating cells. The KM method correctly estimates $K=8$ cell types. However, it cannot recognize OPC cells and gives two additional small clusters that are not interpretable. 
Other methods tend to underestimate $K$. In particular, SKM and MClust estimate $K=3$ and $K=4$ respectively, and perform worse than others in terms of much lower ARIs and NMIs, as shown in Table \ref{tab:real_data_est_k}. Both PCA-KM and tSNE-KM estimate $K=4$ by correctly identifying the astrocytes cell type and merging fetal quiescent and fetal replicating cell types into one cluster. For the other five cell types, PCA-KM identifies microglia cell type and merges oligodendrocytes, OPC, endothelial, and neuron cell types into one cluster, while tSNE-KM identifies the neuron cell type and merges oligodendrocytes, OPC, microglia, and endothelial cells as one cluster. NMF-KM is able to identify neuron, fetal quiescent, and fetal replicating cell types but cannot distinguish others. tSNE-KM identifies oligodendrocytes, OPC, microglia, and endothelial cells as one cluster. 

\begin{table}[t]
    \centering
    \begin{tabular}{c|c c c c c c c }
        \hline
        \hline Methods & Bayesian & SKM & PCA-KM & MClust & NMF-KM & tSNE-KM & KM \\
        \hline Estimate of $K$ & 7 & 3 & 4 & 4 & 6 & 4 & 9\\
        \hline ARI & 0.84 & 0.15 & 0.59 & 0.33 & 0.63 & 0.78 & 0.79\\
        \hline NMI & 0.80 & 0.22 & 0.58 & 0.35 & 0.61 & 0.70 & 0.77\\
        \hline\hline
    \end{tabular}
    \caption{Estimate of the number of clusters $K$, ARIs, and NMIs obtained by applying different methods to the scRNA-Seq data set.}
    \label{tab:real_data_est_k}
\end{table}

\begin{figure}[htp]
    \centering
    \includegraphics[width=.76\textwidth]{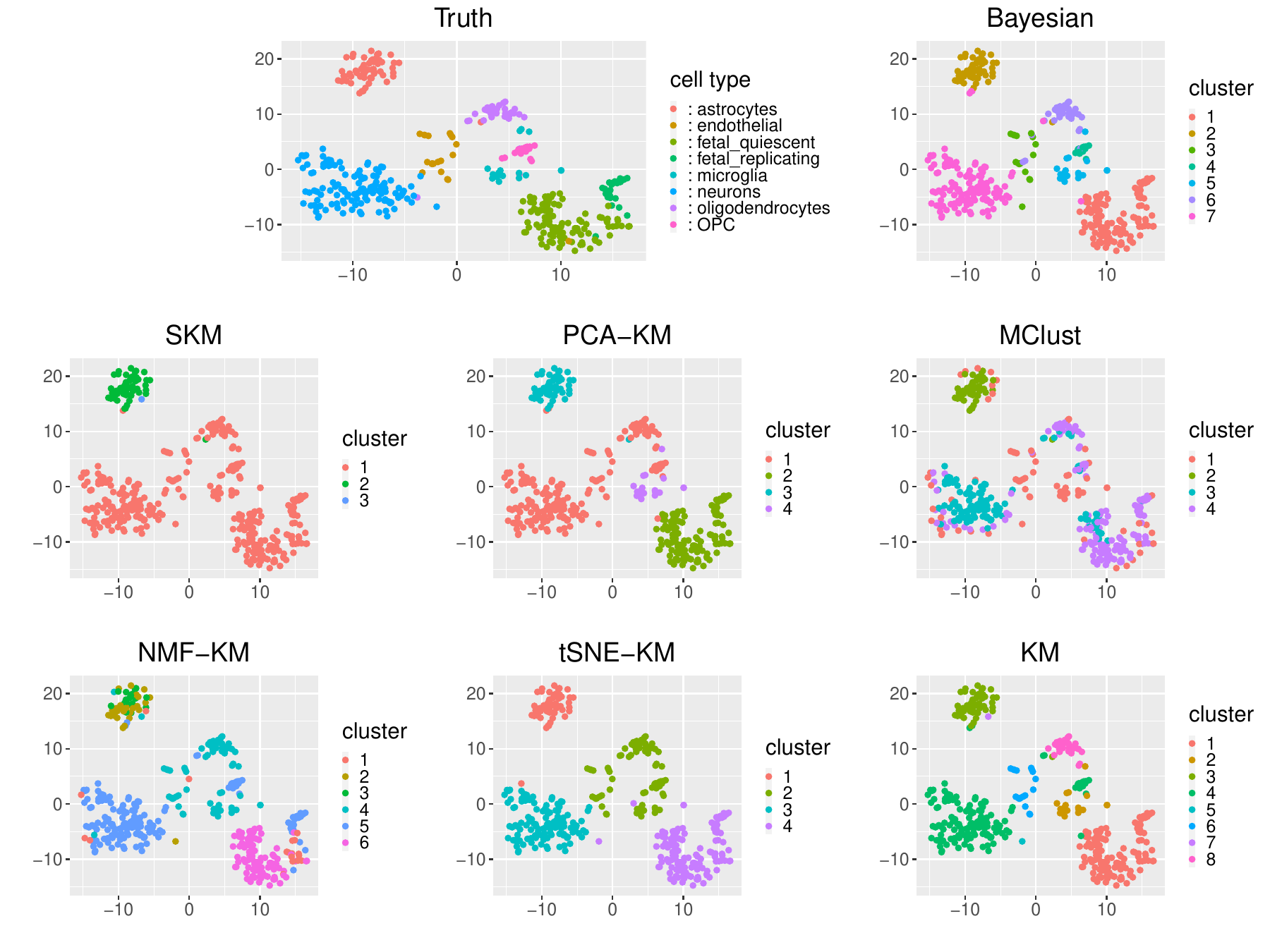}
    \caption{Clustering results of scRNA-Seq data corresponding to different methods. Data points are embedded into two-dimensional subspace by tSNE embedding.}
    \label{fig:real_data_est_k}
\end{figure}

We further examine the alternative methods when the number of clusters is set to be the true number of cell types ($K=8$). Appendix Figure \ref{fig:real_data_true_k} plots the estimated clustering memberships under all the alternative methods. As shown in Figure \ref{fig:real_data_true_k}, MClust cannot distinguish fetal quiescent and fetal replicating cell types and merges OPC and oligodendrocytes cell types into one cluster. PCA-KM and tSNE-KM return clusters with relatively similar sizes and hence their performance on small clusters are not satisfactory. SKM and NMF-KM perform significantly worse than others since they do not correctly identify any single cell type. 
Appendix Table \ref{tab:real_data_true_k} reports the resulting ARIs and NMIs, showing that the proposed Bayesian method still yields the highest ARI and NMI even though the number of clusters is correctly pre-specified for all alternative methods.

\section{Discussion}
We propose a Bayesian approach for high-dimensional Gaussian mixtures where the cluster mean vectors exhibit certain sparsity structure. We fully investigate the minimax risk for estimating the mean matrix, show that the posterior contraction rate is minimax optimal, and obtain an error bound for the mis-clustering error. Our approach demonstrates superior performance in both simulation study and real-world applications. 

There still exist challenges that need further research. 
 One extension is to consider scenarios where the cluster-specific covariance matrices have some structures, such as sparse spiked structures \citep{xie2022bayesian}. Exploring the theoretical properties of covariance matrix estimation could be an interesting future direction. 
On the implementation side, algorithms based on Markov chain Monte Carlo can be computationally expensive in ultra-high dimensions. Certain optimization-based alternatives, such as variational Bayes methods \citep{ray2021variational} can be attractive. Developing the underlying backbone theory for variational Bayes approaches can be a promising future research direction as well.


\section*{Acknowledgement}
The work of Xu was supported by NSF 1918854, NSF 1940107, and NIH R01MH128085.





\renewcommand{\thefigure}{A\arabic{figure}}
\renewcommand{\thetable}{A\arabic{table}}
\setcounter{figure}{0}    
\setcounter{table}{0}

\begin{appendices}
\section{Proofs for Section \ref{sec:model}}\label{sec:proof_model}

\begin{proof}[proof of Theorem \ref{thm2}]
We define three subspaces of $\Theta^*_K$ as follows. 
 
We consider the case where $\bSigma^* = \bI_p$.
Without loss of generality, we consider the case where $n/K$ is an integer. If not, let $n'=\lfloor n/K\rfloor K$. Then we get a lower bound of a smaller parameter space $\Theta'\subset\mathbb{R}^{p\times K}\times\mathbb{R}^{n'\times K}$ which is also a lower bound of the original parameter space $\Theta^*_K\subset\mathbb{R}^{p\times K}\times\mathbb{R}^{n\times K}$. 
We define the first subspace
$$
\Theta^*_{K1}=\left\{(\bmu ,\bl,\bSigma)\in\Theta^*_K:\bmu=[\mu_0\bw_1,\cdots,\mu_0\bw_K] ,\mu_0=\sqrt{\frac{c_1\log K}{s}}, \text{supp}(\bmu) = S,\bSigma^* = \bI_p\right\}
$$ 
for some fixed sparsity support $S\subset [p]$ with $|S| = s$, constant $c_1>0$, and $\bw_1, \dots, \bw_K \in \{0,1\}^p$. Specifically, we choose the vectors $\bw_1, \dots, \bw_K$ such that $\|(\bw_i)_S\|_0 = s$ for all $i\in [K]$ and $\|(\bw_i)_S - (\bw_j)_S\|_2^2 > s/2$ for $i\neq j$. By Lemma 4.10 in \cite{massart2007}, there exists $\{(\bw_1)_S, \dots, (\bw_K)_S\} \subset \{0,1\}^s$ satisfies these properties for $K \leq s$.
Next we define the second subspace. Note that each $\bl$ can be associated with a mapping $z:[n]\to[K]$ such that $l_i=e_{z(i)}$ where $e_i$ is the vector whose $i$th entry is 1 and 0 elsewhere. Then we define

\begin{align*}
\Theta^*_{K2}=\mathrel{\Bigg\{}
& (\bmu ,\bl, \bSigma)\in\Theta^*_K: z^{-1}(k)=\left\{\frac{(k-1)n}{K}+1,\dots,\frac{kn}{K},\bSigma^* = \bI_p\right\},
 |\mathrm{supp}(\bmu)| = s
\mathrel{\Bigg\}},
\end{align*}
where $\mathrm{supp}(\cdot)$ denotes the indices of the non-zero rows of a matrix. 
For the third subspace,
\begin{align*}
\Theta^*_{K3} = \{&(\bmu,\bl,\bSigma)\in\Theta^*_K: \mathrm{supp}(\bmu)=\mathrm{supp}(\tilde{\bmu}), \bl=\tilde{\bl},(\bmu)_{S}=(\tilde{\bmu})_{S}+\bE, \bE\in B_a^{s\times K}(0),\bSigma^* = \bI_p\},
\end{align*}
for fixed $(\tilde{\bmu},\tilde{\bl},\bI)\in\Theta_K^*$ where $B_a^{s\times K}(0)=\{\bmu\in\mathbb{R}^{s\times K}:\|\bmu\|_F\leq a\}$ represents the closed ball of the set of $s\times K$ matrices centered at $0$ with radius $a$ with respect to the Frobenius norm and $S = \mathrm{supp}(\tilde{\bmu})$. Here $a$ is a positive scalar which will be specified later.

\vspace*{2ex}
\noindent
$\blacksquare$ We first consider the minimax lower bound over $\Theta^*_{K1}$.
By Lemma 4.10 in \cite{massart2007} 
we know there exist $\{\bw_1,\cdots,\bw_K\}\subset\{0,1\}^s$ such that $\|\bw_i-\bw_j\|^2>s/2$. Consider an $\epsilon_1$-ball of $\Theta^*_{K1}$ with respect to the metric 
$d_1((\bmu,\bl_1),(\bmu,\bl_2)) = \|\bmu\bl_1^T - \bmu\bl_2^T\|_F/\sqrt{n}.$ 
Suppose $\bl_1$ and $\bl_2$ are associated with mappings $z_1,z_2:[n]\mapsto [K]$. 
We have
\begin{align*}
    n\epsilon_1^2 &>\|\bmu\bl_1^T-\bmu \bl_2^T\|^2_F \geq\mu_0^2\sum_{i=1}^n\|\bw_{z_1(i)}-\bw_{z_2(i)}\|^2_F
    \geq\mu_0^2|\{i:z_1(i)\neq z_2(i)\}|\frac{s}{2}.
\end{align*}
Let $\epsilon_1^2=(c_1\log K)/12$. Since $\mu_0^2=(c_1\log K)/s$, we have
$|\{i:z_1(i)\neq z_2(i)\}|\leq n/6$. Denote $B_\epsilon(\bmu\bl^T) := \{\bl_1\in\mathcal{L}_K:d_1((\bmu,\bl),(\bmu,\bl_1))\leq \epsilon\}$ for any $\bmu\in\mathbb{R}^{p\times K}$ and $\bl\in\mathcal{L}_K$.
Then for any $(\bmu, \bl)\in\Theta^*_{K1}$, we have
\begin{align*}
    &|B_{\epsilon_1}(\bmu \bl^T)|\leq{n\choose n/6}K^{n/6}
    \leq \exp\left(n\log 6-\frac{5n}{6}\log 5 + \frac{n}{6}\log K \right)\\
    &\quad= \exp\left(n\log 6 - \frac{5}{6}\frac{\log 5}{\log 6} n\log 6 + \frac{n}{6}\log K \right)
    \leq \exp\left(\frac{3}{10}n\log 6 + \frac{n}{6}\log K \right)\\
    &\quad\leq \exp\left(\frac{3}{10}\frac{\log 6}{\log 2}n\log K + \frac{1}{6}n\log K \right)
    \leq \exp\left( 0.95 n\log K \right),
\end{align*}
where the second inequality comes from the Stirling's formula and the fourth inequality is due to the fact that $K\geq 2$. 
Since $|\Theta^*_{K1}|=K^n$, we have
$$\log M\left(\epsilon_1,\Theta^*_{K1},d_1\right)\geq\log\frac{K^n}{{n\choose n/6}K^{n/6}} \geq \frac{1}{20} n\log K.$$ 
Note that
$D_{KL}(\mathbb{P}_{\bmu \bl_1^T}\|\mathbb{P}_{\bmu \bl_2^T})=\frac{1}{2}\|\bmu \bl_1-\bmu \bl_2^T\|^2_F\leq\frac{2\mu_0^2sn}{2}= c_1 n\log K.$
Therefore, by the generalized Fano's lemma \citep{Yu1997} 
$$
\inf_{\hat{\bmu},\hat{\bl}}\sup_{(\bmu^*, \bl^*)\in\Theta^*_{K1}}\prob_*\left( \|\hat{\bmu}\hat{\bl}^T-(\bmu^*) (\bl^*)^T\|_F^2\geq c_2n\log K\right) \geq 1 - \frac{c_1 n\log K + \log 2}{n\log K/20}\geq 0.9
$$ 
with some constant $c_2 = {c_1}/{48}>0$ for sufficiently small $c_1<{1}/{20}$ and sufficiently large $n$ where $\mathbb{P}_*$ represents the probability measure under $(\bmu^*,\bl^*,\bSigma^*)$. It follows from Markov's inequality that
\begin{align*}
&\inf_{\hat{\bmu},\hat{\bl}}\sup_{(\bmu^*, \bl^*)\in\Theta^*_{K1}}
\E_*\left(\|\hat{\bmu}\hat{\bl}^T-(\bmu^*) (\bl^*)^T\|_F^2\right)\\
&\quad\geq \inf_{\hat{\bmu},\hat{\bl}}\sup_{(\bmu^*, \bl^*)\in\Theta^*_{K1}}
c_2n\log K
\prob_*\left( \|\hat{\bmu}\hat{\bl}^T-(\bmu^*) (\bl^*)^T\|_F^2\geq c_2n\log K\right)\geq 0.9c_2n\log K.
\end{align*}

\vspace*{2ex}
\noindent
$\blacksquare$ We next consider the minimax lower bound over $\Theta_{K2}^*$. By the construction of $\Theta^*_{K2}$, we have
$\bmu \bl^T=
\begin{pmatrix}
    \bmu_1 & \dots & \bmu_1 & \bmu_2 & \dots & \bmu_2 & \dots & \bmu_K & \dots & \bmu_K
\end{pmatrix}
.$
Thus,
$\|\bmu \bl^T-\bmu'\bl^T\|^2=(n/K)\sum_{k=1}^K\|\bmu_k-\bmu'_k\|^2_2$.
Let $\bmu_k=
    (\lambda_k , a_0 \bv^T)^T
$ for $k\in [K]$, where $\lambda_k$'s are distinct scalars, $\bv\in\{0,1\}^{p-1}$, and $a_0 > 0$ is to be specified later. 
By Lemma 4.10 in \cite{massart2007} 
there exists $\{\bv^{(1)},\cdots,\bv^{(N)}\}\subset\{0,1\}^{p-1}$ such that
\begin{itemize}
    \item $\|\bv^{(i)}-\bv^{(j)}\|^2>\frac{s-1}{2}$ for $i\neq j$,
    \item $\|\bv^{(i)}\|_0=s-1$ for all $i\in [N]$,
    \item $N>\exp(c_4(s-1)\log \frac{p-1}{s-1} )$ for some $c_4\geq 0.233$.
\end{itemize}
For distinct pair $\bmu,\bmu'$, we choose 
$$\bmu=
\begin{pmatrix}
    \lambda_1 & \dots & \lambda_K \\
    a_0 \bv  & \dots & a_0 \bv
\end{pmatrix}\quad\mbox{and}\quad
\bmu'=
\begin{pmatrix}
    \lambda_1 & \dots & \lambda_K \\
    a_0 \bv'  & \dots & a_0 \bv'
\end{pmatrix}
$$
such that $\bv,\bv'\in\{\bv^{(1)},\ldots,\bv^{(N)}\}$, $\bv\neq \bv'$. 
Then we consider an $\epsilon_2$-ball in $\Theta^*_{K2}$ with respect to the metric  $d_1((\bmu,\bl_1),(\bmu,\bl_2)) = \|\bmu\bl_1^T - \bmu\bl_2^T\|/\sqrt{n},$
we have
\begin{align*}
    \frac{1}{n}\|\bmu \bl^T-\bmu'\bl^T\|_F^2&=\frac{1}{K}\sum_{k=1}^K\|\bmu_k-\bmu_k'\|_2^2
    =\frac{1}{K}\sum_{k=1}^Ka_0^2\|\bv-\bv'\|_2^2    
    >\frac{1}{K}a_0^2\frac{(s-1)K}{2}=\frac{a_0^2(s-1)}{2}.
\end{align*}
Let $a_0^2=2\{sc_3\log (p/s)\}/\{n(s-1)\}$ and $\epsilon_2^2=\{c_3s\log (p/s)\}/{n}$, where $c_3 > 0$ is a constant to be specified later. Then we have
\begin{align*}
    M(\epsilon_2,\Theta^*_{K2},d_1)\geq \exp\left(c_4(s-1)\log \frac{p-1}{s-1} \right) \geq \exp\left(\frac{c_4}{4}s\log \frac{p}{s} \right)
\end{align*}
for $s\geq 2$ because $(p - 1)/(s - 1)\geq \sqrt{p/s}$.
Note that for any $(\bmu, \bl),(\bmu', \bl)\in\Theta^*_{K2}$,
$$
D_{KL}(\mathbb{P}_{\bmu \bl^T}\|\mathbb{P}_{\bmu' \bl^T})=\frac{1}{2}\|\bmu \bl-\bmu' \bl^T\|^2_F\leq a_0^2sn.
$$ 
Without loss of generality we may assume that $s\geq 11$. 
It follows that $\log 2\leq \frac{10}{11}\frac{c_4}{4}s\log\left(\frac{p}{s}\right)$ because $c_4 \geq 0.2$. It follows that
\begin{align*}
    \frac{d_{KL}(\Theta^*_{K2})+\log 2}{\log M(\epsilon_2,\Theta^*_{K2},d_1)}\leq \frac{a_0^2sn+\log 2}{\frac{c_4}{4}s\log \frac{p}{s}}\leq \frac{16c_3}{c_4} + \frac{10}{11}.
\end{align*}
Therefore, by selecting $c_3 = c_4/(22\times 16)$, we have
$$\inf_{\hat{\bmu},\hat{\bl}}\sup_{(\bmu^*, \bl^*)\in\Theta^*_{K2}}\prob_*\left( \|\hat{\bmu}\hat{\bl}^T-(\bmu^*) (\bl^*)^T\|_F^2\geq c_3s\log \frac{p}{s}\right) \geq \frac{1}{22}$$ 
for $c_4\geq 0.233$, and for sufficiently large $n$. It follows from Markov's inequality that
\[
\inf_{\hat{\bmu},\hat{\bl}}\sup_{(\bmu^*, \bl^*)\in\Theta^*_{K2}}\E_*\left(\|\hat{\bmu}\hat{\bl}^T-(\bmu^*) (\bl^*)^T\|_F^2\right)\geq \frac{c_3}{22}s\log \frac{p}{s}.
\]

\vspace*{2ex}
\noindent
$\blacksquare$ We now consider the minimax lower bound over $\Theta^*_{K3}$. 
Let $\bmu=\tilde{\bmu}+\bE$ and $\bmu'=\tilde{\bmu}+\bE'$ where $\bE, \bE'\in B_a^{s\times K}(0)$.
Then we have 
\begin{align*}
d_{KL}(\Theta^*_{K3})
&=\frac{1}{2}\sup_{\bE,\bE'\in B_a^{s\times K}(0)}\|(\bE-\bE')(\bl^*)^T\|_F^2
\leq \frac{1}{2}\sup_{\bE,\bE'\in B_a^{s\times K}(0)} n\|\bE-\bE'\|_F^2= a^2n.
\end{align*}
Let $\epsilon_3^2=c_5{sK}n_{\min}/n^2$ for some $0<c_5<1$ where $n_{\min}$ is the size of the smallest cluster induced by $\tilde{\bl}$. Since
we have
\begin{align*}
d_1((\bmu,\tilde{\bl}),(\bmu',\tilde{\bl}))& = \frac{1}{\sqrt{n}}\|(\bE - \bE')\tilde{\bl}^T\|_F\geq \frac{1}{\sqrt{n}}\|\bE - \bE'\|_F\sigma_{\min}(\tilde{\bl})
= \sqrt{\kappa_{\min}}\|\bE - \bE'\|_F,
\end{align*}
where $\kappa_{\min} = n_{\min}/n$, it follows that
\begin{align*}
    M(\epsilon_3,\Theta^*_{K3}, d_1)&\geq M\left(\frac{\epsilon_3}{\sqrt{\kappa_{\min}}}, B_a^{s\times K}(0),\|\cdot\|_F\right)
    \geq \left(\frac{\sqrt{\kappa_{\min}}}{\epsilon_3}\right)^{Ks}\frac{\mathrm{Vol}(B_a^{s\times K}(0))}{\mathrm{Vol}(B_1^{s\times K}(0))}
    =\left(\frac{a\sqrt{\kappa_{\min}}}{\epsilon_3}\right)^{Ks}.
\end{align*}
Thus, by letting $a=\sqrt{sK/n}$ and using the generalized Fano's Lemma \citep{Yu1997} 
we obtain
\begin{align*}
    \inf_{\hat{\bmu},\hat{\bl}}\sup_{(\bmu^*, \bl^*)\in\Theta_{K3}^*}\prob_*\left( \|\hat{\bmu}\hat{\bl}^T-(\bmu^*) (\bl^*)^T\|_F^2\geq c_6\frac{sKn_{\min}}{n}\right) \geq 1-\frac{sK+\log 2}{sK\log \sqrt{\frac{1}{c_5}}}\geq 0.9
\end{align*}
for sufficiently large $s$ and $K$ and sufficiently small $c_5$ with $c_6=c_5/4$, because $sKn_{\min}/n^2\leq s/n$. Then by Markov's inequality, 
\[
\inf_{\hat{\bmu},\hat{\bl}}\sup_{(\bmu^*, \bl^*)\in\Theta^*_{K3}}\E_*\left(\|\hat{\bmu}\hat{\bl}^T-(\bmu^*) (\bl^*)^T\|_F^2\right)\geq\frac{0.9c_6sKn_{\min}}{n}.
\]

\vspace*{2ex}\noindent
$\blacksquare$ Now we combined the minimax lower bounds over $\Theta^*_{K1}$, $\Theta^*_{K2}$, and $\Theta^*_{K3}$:
\begin{align*}
    &\inf_{\hat{\bmu},\hat{\bl}}\sup_{(\bmu^*,\bl^*)\in\Theta_K^*}\E_*\left(\|\hat{\bmu} \hat{\bl}^T-(\bmu^*)(\bl^*)^T\|_F^2\right)
    \geq\inf_{\hat{\bmu},\hat{\bl}}\max_{j\in\{1,2,3,\}}\sup_{(\bmu^*,\bl^*)\in\Theta^*_{Kj}} \E_*\left(\|\hat{\bmu} \hat{\bl}^T-(\bmu^*)(\bl^*)^T\|_F^2\right)\\
    &\quad\geq\max\left\{0.9c_2n\log K, \frac{c_3}{22}s\log\frac{p}{s}, \frac{0.9c_6sKn_{\min}}{n}\right\}
    \geq C\left(s\log\frac{p}{s} + n\log K \right).
\end{align*}
\end{proof}

\begin{proof}[proof of Theorem \ref{thm3}]

Note that by a basic inequality, we have
\begin{align*}
     &\quad\|\by-(\bmu^*)(\bl^*)^T\|_F^2\geq
     \|\by-\widehat{\bmu}\widehat{\bl}^T\|_F^2
     = \|(\by - (\bmu^*)(\bl^*)^T) + ((\bmu^*)(\bl^*)^T - \widehat{\bmu}\widehat{\bl}^T)\|_F^2
     \\
     &\quad=\|\by-(\bmu^*)(\bl^*)^T\|_F^2+\|(\bmu^*)(\bl^*)^T-\widehat{\bmu}\widehat{\bl}^T\|_F^2
     +2\langle\by-(\bmu^*)(\bl^*)^T,(\bmu^*)(\bl^*)^T-\widehat{\bmu}\widehat{\bl}^T\rangle_F\\
     &\quad=\|\by-(\bmu^*)(\bl^*)^T\|_F^2+\|(\bmu^*)(\bl^*)^T-\widehat{\bmu}\widehat{\bl}^T\|_F^2\\
     &\quad\quad+2\left\langle\by-(\bmu^*)(\bl^*)^T,\frac{(\bmu^*)(\bl^*)^T-\widehat{\bmu}\widehat{\bl}^T}{\|(\bmu^*)(\bl^*)^T-\widehat{\bmu}\widehat{\bl}^T\|_F}\right\rangle_F\times \|(\bmu^*)(\bl^*)^T-\widehat{\bmu}\widehat{\bl}^T\|_F,
\end{align*}
where $\langle \cdot,\cdot\rangle_F$ is the Frobenius inner product defined by $\langle \bA,\bB\rangle_F = \mathrm{tr}(\bA^T\bB)$. After rearranging the terms on the both sides of the inequality, we obtain
\begin{align*}
\E_*\|(\bmu^*)(\bl^*)^{T} -  \widehat{\bmu}\widehat{\bl}^T\|_F 
\leq 2\E_*\left[\sup_{({\bmu}, {\bl})\in \Theta_K}\left\langle
 \bE, (\bSigma^*)^{\frac{1}{2}}\frac{{\bmu}{\bl}^T -  (\bmu^*) (\bl^*)^{T}}{\|{\bmu}{\bl}^T-(\bmu^*)(\bl^*)^T\|_F}
 \right\rangle_F\right],
\end{align*}
where $\bE=(\bSigma^*)^{-\frac{1}{2}}(\by-(\bmu^*)(\bl^*)^T)$ is the standardized noise matrix. 

Consider the set of matrices $$\widetilde{\Theta}_K=\left\{(\bSigma^*)^{\frac{1}{2}}\frac{\bmu\bl^T-(\bmu^*)(\bl^*)^T}{\|\bmu\bl^T-(\bmu^*)(\bl^*)^T\|_F}:\bmu\in\mathbb{R}^{p\times K},|\text{supp}(\bmu)|\leq s,\bl\in\mathcal{L}_K\right\}.$$
To obtain an upper bound of the right hand side of the inequality above, we use some tools of maximal inequality of empirical process. Specifically, we define a stochastic process $X(\bB)=\langle\bE,\bB\rangle_F$ indexed by a $p\times n$ matrix $\bB$. Since the entries of $\bE$ are i.i.d. standard Gaussian, it follows that $X(\bB)$ is sub-Gaussian. Then by Corollary 8.5 in \cite{kosorok2008}, 
\begin{align*}
    &\E_*\left[\sup_{({\bmu}, {\bl})\in \Theta_K}\left\langle
 \bE, (\bSigma^*)^{\frac{1}{2}}\frac{{\bmu}{\bl}^T -  (\bmu^*) (\bl^*)^{T}}{\|{\bmu}{\bl}^T-(\bmu^*)(\bl^*)^T\|_F}
 \right\rangle_F\right]
    \lesssim \int_0^{\mathrm{diam}(\widetilde{\Theta}_K)}\sqrt{\log N(\epsilon, \widetilde{\Theta}_K, \|\cdot\|_F)} \mathrm{d}\epsilon.
\end{align*}
Obtaining a sharp upper bound of $N(\epsilon, \widetilde{\Theta}_K, \|\cdot\|_F)$ is quite involved. We breakdown the computation of a sharp bound for the covering number of $\widetilde{\Theta}_K$ as follows.

\vspace*{1ex}
\noindent\textbf{Step 1: Decompose $\widetilde{\Theta}_K$ into unions of subspaces $\mathcal{E}_j^K$ where $\|\bmu\bl^T - (\bmu^*)(\bl^*)^T\|_F$ is bounded. }
Define a function $B:\mathbb{R}^{p\times K}\times \mathcal{L}_K\mapsto\widetilde{\Theta}_K$ as 
$B(\bmu,\bl)=(\bSigma^*)^{\frac{1}{2}}\frac{\bmu\bl^T-(\bmu^*)(\bl^*)^T}{\|\bmu\bl^T-(\bmu^*)(\bl^*)^T\|_F}.$
Then we have 
\begin{align*}
    &\|B(\bmu,\bl)-B(\bmu',\bl')\|_F \leq \left\|(\bSigma^*)^{\frac{1}{2}}\right\|_2 \left\|\frac{\bmu\bl^T-(\bmu^*)(\bl^*)^T}{\|\bmu\bl^T-(\bmu^*)(\bl^*)^T\|_F}-\frac{(\bmu')(\bl')^T-(\bmu^*)(\bl^*)^T}{\|(\bmu')(\bl')^T-(\bmu^*)(\bl^*)^T\|_F}\right\|_F\\
    &\quad=\lambda_{\max}(\bSigma^*)^{\frac{1}{2}} \bigg\|\frac{\bmu\bl^T\|(\bmu')(\bl')^T-(\bmu^*)(\bl^*)^T\|_F-(\bmu')(\bl')^T\|(\bmu^*)(\bl^*)^T-\bmu\bl^T\|_F}{\|\bmu\bl^T-(\bmu^*)(\bl^*)^T\|_F\|(\bmu')(\bl')^T-(\bmu^*)(\bl^*)^T\|_F}\\
    &\quad\quad\quad\quad -\frac{(\bmu^*)(\bl^*)^T(\|(\bmu')(\bl')^T-(\bmu^*)(\bl^*)^T\|_F-\|\bmu\bl^T-(\bmu^*)(\bl^*)^T\|_F)}{\|\bmu\bl^T-(\bmu^*)(\bl^*)^T\|_F\|(\bmu')(\bl')^T-(\bmu^*)(\bl^*)^T\|_F}\bigg\|_F\\
    &\quad\leq \lambda_{\max}(\bSigma^*)^{\frac{1}{2}} \left(\frac{\|\bmu\bl^T-(\bmu')(\bl')^T\|_F}{\|\bmu\bl^T-(\bmu^*)(\bl^*)^T\|_F}
    +\frac{\left|\|(\bmu')(\bl')^T-(\bmu^*)(\bl^*)^T\|_F-\|\bmu\bl^T-(\bmu^*)(\bl^*)^T\|_F\right|}{\|\bmu\bl^T-(\bmu^*)(\bl^*)^T\|_F}\right)\\
    &\quad\leq2 \lambda_{\max}(\bSigma^*)^{\frac{1}{2}} \frac{\|\bmu\bl^T-(\bmu')(\bl')^T\|_F}{\|\bmu\bl^T-(\bmu^*)(\bl^*)^T\|_F}.
\end{align*}
We split $\Theta_K$ as follows. Denote 
\[
\mathcal{E}_j^K=\left\{(\bmu,\bl):\bmu\in\mathbb{R}^{p\times K},|\text{supp}(\bmu)|\leq s,\bl\in\mathcal{L}_K,\|\bmu\bl^T-(\bmu^*)(\bl^*)^T\|_F\in(a_j,a_{j+1}]\right\}
\]
where the series $\{a_j\}_{j=-\infty}^\infty$ is to be determined later satisfying $\lim_{j\to\infty}a_j=\infty$ and $\lim_{j\to-\infty}a_j=0$. We also require that $a_j \leq a_{j + 1}/\sqrt{n}$. 
Suppose $\mathcal{N}_j^K$ is an $a_j\epsilon/2$-covering of $\mathcal{E}_j^K$ with respect to Frobenius norm, then we know that $\bigcup_{j=-\infty}^{\infty} B(\mathcal{N}_j^K)$ is an $\epsilon \lambda_{\max}(\bSigma^*)^{\frac{1}{2}}$-covering of $\widetilde{\Theta}_K$ with respect to Frobenius norm. Thus $N(\epsilon \lambda_{\max}(\bSigma^*)^{\frac{1}{2}},\widetilde{\Theta}_K,\|\cdot\|_F)\leq \sum_{j=-\infty}^{\infty}N(a_j\epsilon/2,\mathcal{E}_j^K,\|\cdot\|_F)$.

\vspace*{1ex}\noindent
\textbf{Step 2: Show that $\bl$ equals $\bl^*$ up to a permutation for $(\bmu, \bl)\in\mathcal{E}_j^K$ when $a_{j + 1}$ is small. }
denote $\Delta$ as the minimum distance among all cluster centers, that is, $\Delta=\min_{i,j\in[K]}\|\bmu^*_i-\bmu_j^*\|_2$. For fixed $\bl$ which is induced by $z:[n]\mapsto[K]$, we denote $n_{gh}\overset{\Delta}{=}|z^{-1}(g)\bigcap (z^*)^{-1}(h)|$, $n_g=\sum_{h=1}^Kn_{gh}$, and $n_h^*=\sum_{g=1}^Kn_{gh}$. We then have
\begin{align*}
    \|\bmu\bl^T-(\bmu^*)(\bl^*)^T\|_F^2
    &=\sum_{g=1}^{K}\sum_{h=1}^K n_{gh}\|\bmu_g-\bmu^*_h\|^2_2\\
    &=\sum_{g=1}^K \left\{n_g\left\|\bmu_g-\frac{\sum_{h=1}^Kn_{gh}\bmu^*_h}{n_g}\right\|_2^2+\sum_{h=1}^Kn_{gh}\|\bmu^*_h\|_2^2-\frac{\left\|\sum_{h=1}^Kn_{gh}\bmu^*_h\right\|_2^2}{n_g} \right\}\\
    &=\sum_{g=1}^K n_g\left\|\bmu_g-\frac{\sum_{h=1}^Kn_{gh}\bmu^*_h}{n_g}\right\|_2^2+C_{\bl}.
\end{align*}
 where $C_{\bl}\overset{\Delta}{=}\sum_{g=1}^K\left\{\sum_{h=1}^Kn_{gh}\|\bmu^*_h\|_2^2-\left\|\sum_{h=1}^Kn_{gh}\bmu^*_h\right\|_2^2/n_g\right\}$.

Note that by Cauchy-Schwarz inequality and triangle inequality, we have
\begin{align*}
    \left\|\sum_{h=1}^Kn_{gh}\bmu^*_h\right\|^2_2\leq \left(\sum_{h=1}^Kn_{gh}\left\|\bmu_h^*\right\|_2\right)^2\leq \sum_{h=1}^Kn_{gh}\sum_{h=1}^Kn_{gh}\|\bmu^*_h\|_2^2=n_g\sum_{h=1}^Kn_{gh}\|\bmu^*_h\|_2^2
\end{align*}
for every $g\in[K]$, which implies $C_{\bl}\geq 0$ for any $\bl$.
Note that if $\bl\bP=\bl^*$ for some permutation matrix $\bP\in\mathcal{S}^{K\times K}$ where $\mathcal{S}^{K\times K}$ is the set of all $K\times K$ permutation matrices and $\tau$ which is the permutation function induced by $\bP$, then for every $g\in[K]$, we have $n_g=n_{g,\tau(g)}$ and $n_{gh}=0$ for $h\neq \tau(g)$. Thus 
$\left\|\sum_{h=1}^Kn_{gh}\bmu^*_h\right\|_2^2=n_{g,\tau(g)}^2\|\bmu^*_{\tau(g)}\|_2^2=n_gn_{g,\tau(g)}\|\bmu_{\tau(g)}^*\|_2^2=n_g\sum_{h=1}^Kn_{gh}\|\bmu^*_h\|_2^2,$
which implies $C_{\bl}=0$. Next, by rearranging the terms in the expression of $C_{\bl}$ we have
\begin{align*}
    &\sum_{h=1}^Kn_gn_{gh}\|\bmu_h^*\|_2^2-\left\|\sum_{h=1}^Kn_{gh}\bmu^*_h\right\|_2^2
    = \sum_{h=1}^Kn_{gh}(n_g-n_{gh})\|\bmu^*_h\|_2^2-\sum_{h_1=1}^K\sum_{h_2\neq h_1}n_{gh_1}n_{gh_2} (\bmu_{h_1}^*)^T\bmu_{h_2}^*\\
    &\quad = \sum_{h_1=1}^K\sum_{h_2\neq h_1}\left\{n_{gh_1}n_{gh_2}\|\bmu^*_{h_1}\|_2^2-n_{gh_1}n_{gh_2}(\bmu^*_{h_1})^T\bmu^*_{h_2}\right\}
    = \sum_{h_1=1}^K\sum_{h_2\neq h_1} \frac{1}{2} n_{gh_1}n_{gh_2}\|\bmu^*_{h_1}-\bmu^*_{h_2}\|_2^2.
\end{align*}
Suppose there is no permutation matrix $\bP\in \mathcal{S}^{K\times K}$ such that $\bl\bP = \bl^*$. Then there exists some $g'\in[K]$ such that for some $h_1',h_2'\in[K]$ and $h_1'\neq h_2'$, $n_{g'h_1'}\geq 1$ and $n_{g'h_2'}\geq 1$. 
Furthermore, for such $g'\in [K]$ and $h_1'$, $h_2'$, $n_{g'h_1'}n_{g'h_2'}\geq (n_{g'}-1)/(K-1)$.
If this is not true, then we obtain $n_{g'h_1'}(n_{g'}-n_{g'h_1'})=\sum_{h_2\neq h_1'}n_{g'h_1'}n_{g'h_2}<n_g-1$ for the aforementioned $h_1'$, and this contradicts to the fact that $n_{g'h_1'}(n_{g'}-n_{g'h_1'})\geq n_{g'}-1$ for $1\leq n_{g'h_1'}\leq n_{g'}-1$. The reason is that $n_{g'h_1'}(n_{g'} - n_{g'h_1'})$ is a quadratic function of $n_{g'h_1'}$ and the minimum is achieved when $n_{g'h_1'} = 1$ or $n_{g'h_1'} = n_{g'} - 1$. 
Therefore, we know that if $\bl$ is not identical to $\bl^*$ up to permutation, then there exists some $g'\in[K]$ and $h_1'\neq h_2'$ such that
\begin{align*}
    C_{\bl}&=\sum_{g=1}^K\left\{\sum_{h=1}^Kn_{gh}\|\bmu^*_h\|_2^2-\frac{\left\|\sum_{h=1}^Kn_{gh}\bmu^*_h\right\|_2^2}{n_g}\right\}
    = \sum_{g=1}^K\sum_{h_1=1}^K\sum_{h_2\neq h_1} \frac{n_{gh_1}n_{gh_2}\|\bmu^*_{h_1}-\bmu^*_{h_2}\|_2^2}{2n_g} \\
    &\geq \frac{n_{g'h_1'}n_{g'h_2'}\|\bmu^*_{h_1'}-\bmu^*_{h_2'}\|_2^2}{n_{g'}}
    \geq \frac{n_{g'}-1}{n_{g'}(K-1)}\Delta^2
    \geq\frac{\Delta^2}{2(K-1)}.
\end{align*}
Thus for $\mathcal{E}_j^K$, we can see that if $a_{j+1}\leq \Delta/\sqrt{2(K-1)}$, then every $(\bmu,\bl)\in\mathcal{E}_j^K$ satisfies $\bl\bP=\bl^*$ for some permutation matrix $\bP\in\mathcal{S}^{K\times K}$. 

\vspace*{1ex}\noindent
\textbf{Step 3: Reduction of covering numbers of $\mathcal{E}_j^K$'s for small $a_{j + 1}$. }
For $(\bmu, \bl)\in\mathcal{E}_j^K$ with $a_{j + 1}\leq \Delta/\sqrt{2(K - 1)}$, we have $\|\bmu\bl^T-(\bmu^*)(\bl^*)^T\|_F=\|(\bmu\bP-\bmu^*)(\bl^*)^T\|_F\leq \sqrt{n}\|\bmu\bP-\bmu^*\|_F$
for some permutation matrix $\bP\in\mathcal{S}^{K\times K}$. 
For a fixed $\bP\in\mathcal{S}^{K\times K}$, denote
$
\mathcal{G}_{j}^K(\bP)\overset{\Delta}{=}\{\bmu\in\mathbb{R}^{p\times K}: |\text{supp}(\bmu)|\leq s, \|\bmu\bP-\bmu^*\|_F\in (a_j,a_{j+1}/\sqrt{n}]\}, 
\mathcal{G}_j^K\overset{\Delta}{=}\{\bmu\in\mathbb{R}^{p\times K}: |\text{supp}(\bmu)|\leq 2s, \|\bmu\|_F\in(a_j,a_{j+1}/\sqrt{n}]\}.
$
Then for every $\bP\in\mathcal{S}^{K\times K}$, there is an injective function $f_{\bP}:\mathcal{G}_j^K(\bP)\to\mathcal{G}_j^K$ such that $f_{\bP}(\bmu)=\bmu\bP-\bmu^*$. 
Thus $f$ is a bijective function between $\mathcal{G}_j^K(\bP)$ and $R_f(\mathcal{G}_j^K(\bP))$ where $R_f(\mathcal{G}_j^K(\bP))\subset \mathcal{G}_j^K$ is the image of function $f$. Note that $f$ is also distance-preserving with respect to the Frobenius norm, i.e., 
$\|f(\bmu)-f(\bmu')\|_F=\|\bmu-\bmu'\|_F.$
Thus for any $\eta>0$, 
$
N(\eta,\mathcal{G}_j^K(\bP),\|\cdot\|_F)=N(\eta,R_f(\mathcal{G}_j^K(\bP)),\|\cdot\|_F)\leq N(\eta,\mathcal{G}_j^K,\|\cdot\|_F).
$
We know that for every $(\bmu,\bl)\in\mathcal{E}_j^K$ such that $a_{j+1}\leq \Delta/\sqrt{2(K-1)}$, there exists $\bP\in\mathcal{S}^{K\times K}$ such that $\bl\bP=\bl^*$. Suppose $\mathcal{N}_j^K(\bP)$ is a $\frac{\zeta}{\sqrt{n}}$-covering of $\mathcal{G}_j^K(\bP)$. Then there exists $\tilde{\bmu}\in\mathcal{N}_j^K(\bP)$ such that
$\|\tilde{\bmu}\bl-(\bmu^*)(\bl^*)^T\|_F=\|(\tilde{\bmu}\bP-\bmu^*)(\bl^*)^T\|_F\in(a_j,a_{j+1}]$, i.e., $(\tilde{\bmu},\bl)\in\mathcal{E}_j^K$, and $\|\bmu\bl-\tilde{\bmu}\bl\|_F\leq \sqrt{n}\|\bmu-\tilde{\bmu}\|_F\leq \zeta$, which means $\bigcup_{\bP\in\mathcal{S}^{K\times K}}\mathcal{N}_j^K(\bP)$ is a $\zeta$-covering of $\mathcal{E}_j^K$.
Then we have
\[
N(a_j\epsilon/2,\mathcal{E}_{j}^K,\|\cdot\|_F)\leq \sum_{\bP\in\mathcal{S}^{K\times K}}N(a_j\epsilon/(2\sqrt{n}),\mathcal{G}_{j}^K(\bP),\|\cdot\|_F)\leq K!N(a_j\epsilon/(2\sqrt{n}),\mathcal{G}_j^K,\|\cdot\|_F)
\]
when $a_{j+1}\leq \Delta/\sqrt{2(K-1)}$. Note that for the covering number of the space $\mathcal{E}_j^K$, we abuse the notation and use $\|\cdot\|_F$ to denote the metric $d((\bmu_1,\bl_1),(\bmu_2,\bl_2)) = \|\bmu_1\bl_1^T - \bmu_2\bl_2^T\|_F$. 

\vspace*{1ex}\noindent
\textbf{Step 4: Reduction of covering numbers of $\mathcal{E}_j^K$'s for large $a_{j + 1}$. }
Next we consider the case when $a_j$ is relatively large. Specifically, when $a_j\geq 2\|(\bmu^*)(\bl^*)^T\|_F$, and $(\bmu, \bl)\in\mathcal{E}_j^K$, we know that 
$\|\bmu\bl^T\|_F\in (a_j-\|(\bmu^*)(\bl^*)^T\|_F,a_{j+1}+\|(\bmu^*)(\bl^*)^T\|_F]\subset (a_j/2,2a_{j+1}].$
For $\mathcal{E}_j^K$, we can write it as $\mathcal{E}_j^K=\bigcup_{m=0}^{K - 1}\mathcal{E}_{j,m}^K$ where $\mathcal{E}_{j,m}^K$ is the subset of $\mathcal{E}_j^K$ whose $\bl$ is induced by a clustering with $m$ empty clusters. Then for $\mathcal{E}_{j,m}^K$, it suffices to consider $\mathcal{F}_{j,m}^K\overset{\Delta}{=}\{(\bmu_{-m},\bl_{-m}):(\bmu,\bl)\in\mathcal{E}_{j,m}^K\}$ where $\bmu_{-m}$ and $\bl_{-m}$ are the sub-matrices of $\bmu$ and $\bl$ by deleting the columns that correspond to the empty clusters. 
For those $j$'s with $a_j\geq 2\|(\bmu^*)(\bl^*)^T\|_F$, we further have
\begin{align*}
\mathcal{F}_{j,m}^K&\subset\{(\bmu_{-m},\bl_{-m}):\|\bmu_{-m}\bl_{-m}^T\|_F\in (a_j/2, 2a_{j + 1}],|\mathrm{supp}(\bmu_{-m})|\leq s\}\\
&\subset\bigcup_{\bl_{-m}\in\mathcal{L}_{K - m}}\mathcal{H}_{j,m}^K\times\{\bl_{-m}\},
\end{align*}
because the singular values of $\bl_{-m}$ are between $\sqrt{n}$ and $1$, where 
\[
\mathcal{H}_{j,m}^K\overset{\Delta}{=}\{\bmu_{-m}\in\mathbb{R}^{p\times (K-m)}:|\text{supp}(\bmu_{-m})|\leq s, \|\bmu_{-m}\|_F\in (a_j/2,2a_{j+1}/\sqrt{n}]\}.
\]
Since for any $(\bmu_{-m},\bl_{-m})$ and $(\bmu_{-m}',\bl_{-m})\in \mathcal{H}_{j,m}^K\times\{\bl_{-m}\}$ we have
$
\|\bmu_{-m}\bl_{-m}^T - \bmu_{-m}'\bl_{-m}^T\|_F\leq \|\bl_{-m}\|_2\|\bmu_{-m} - \bmu_{-m}'\|_F\leq \sqrt{n}\|\bmu_{-m} - \bmu_{-m}'\|_F,
$
it follows that
\[
N(a_j\epsilon/2,\mathcal{E}_j^K,\|\cdot\|_F)\leq\sum_{m=0}^{K-1} N(a_j\epsilon/2,\mathcal{F}_{j,m}^K,\|\cdot\|_F)\leq \sum_{m=0}^{K-1}|\mathcal{L}_{K-m}|N(a_j\epsilon/(2\sqrt{n}),\mathcal{H}_{j,m}^K,\|\cdot\|_F).
\]

\vspace*{1ex}\noindent
\textbf{Step 5: Computing covering numbers of $\mathcal{E}_j^K$'s for small and large $a_{j + 1}$. }
We denote $a_{-1}=\Delta/\sqrt{2(K-1)}$, $a_1=\Delta/\sqrt{2(K-1)}$, $a_2=2\|(\bmu^*)(\bl^*)^T\|_F$, $\mathcal{E}_0^K=\mathcal{E}_{-1}^K=\emptyset$ and 
\begin{align*}
    \frac{a_{j+1}}{a_j}=
    \begin{cases}
    \frac{\sqrt{n}}{4}\left(1+\frac{1}{j^2}\right)^{\frac{1}{sK}} & \text{ for } j>2\\
    \sqrt{n}\left(1+\frac{1}{(-j-1)^2}\right)^{\frac{1}{2sK}} & \text{ for } j<-1
    \end{cases}.
\end{align*}
Note that without loss of generality we may assume $a_2 > a_1$. 
We have $a_{j + 1}/\sqrt{n}\geq a_j$ and 
\begin{align*}
    \lim_{j\to\infty}a_j&=\lim_{j\to\infty}\left(\frac{n}{16}\right)^{\frac{j-1}{2}}\left(\prod_{i=1}^{j-1}\left(1+\frac{1}{i^2}\right)\right)^{\frac{1}{sK}}2\|(\bmu^*)(\bl^*)^T\|_F=\infty\\
    \lim_{j\to\infty}a_{-j}&=\lim_{j\to\infty}n^{-\frac{j-1}{2}}\left(\prod_{i=1}^{j-1}\left(1+\frac{1}{i^2}\right)\right)^{-\frac{1}{2sK}}\frac{\Delta}{\sqrt{2(K-1)}}=0.
\end{align*}
For $j>2$, we have
\begin{align*}
    &N\left(\frac{a_j\epsilon}{2\sqrt{n}},\mathcal{H}_{j,m}^K,\|\cdot\|_F\right)\leq {p\choose s} \left(\frac{3\sqrt{n}}{a_j\epsilon}\right)^{s(K-m)}\left\{\left(\frac{4a_{j+1}}{\sqrt{n}}\right)^{s(K-m)}-a_j^{s(K-m)}\right\}\\
    &\quad\leq {p\choose s}\left(\frac{3\sqrt{n}}{\epsilon}\right)^{sK}\left\{\left(\frac{4a_{j+1}}{a_j\sqrt{n}}\right)^{sK}-1\right\}
    \leq {p\choose s}\left(\frac{3\sqrt{n}}{\epsilon}\right)^{sK}\frac{1}{j^2}.
\end{align*}
Then 
\begin{align*}
    \sum_{j=2}^\infty N(a_j\epsilon/2,\mathcal{E}_j^K,\|\cdot\|_F) &\leq \sum_{j=1}^\infty\sum_{m=0}^{K-1} |\mathcal{L}_{K-m}| N\left(\frac{a_j\epsilon}{2\sqrt{n}},\mathcal{H}_{j,m}^K,\|\cdot\|_F\right)
    \leq KK^n{p\choose s}\left(\frac{3\sqrt{n}}{\epsilon}\right)^{sK}\frac{\pi^2}{6}.
\end{align*}
Similarly, 
\begin{align*}
    \sum_{j=-\infty}^{-1} N(a_j\epsilon/2,\mathcal{E}_j^K,\|\cdot\|_F)&\leq \sum_{j=-\infty}^{-1} K!N\left(\frac{a_j\epsilon}{2\sqrt{n}},\mathcal{G}_{j}^K,\|\cdot\|_F\right)
    \leq K!{p\choose 2s}\left(\frac{6\sqrt{n}}{\epsilon}\right)^{2sK}\frac{\pi^2}{6}.
\end{align*}

\vspace*{1ex}\noindent
\textbf{Step 6: Computing the covering number of $\mathcal{E}_1^K$. }
Denote $\mathcal{F}_{1,m}^K(\bl)\overset{\Delta}{=}\{(\bmu_{-m},\bl_{-m}):(\bmu_{-m},\bl_{-m})\in\mathcal{E}_{1,m}^K\}$ for fixed $\bl$ which induces $m$ empty clusters. Then we have $\mathcal{F}_{1,m}^K=\bigcup_{\bl\in\mathcal{L}_{K-m}}\mathcal{F}_{1,m}^K(\bl)$. By previous derivation we get $\|\bmu_{-m}\bl_{-m}^T-(\bmu^*)(\bl^*)^T\|_F^2\leq a_{j+1}^2$ is equivalent to 
$\sum_{g=1}^{K-m} n_g\left\|\bmu_g-\frac{\sum_{h=1}^Kn_{gh}\bmu^*_h}{n_g}\right\|_2^2+C_{\bl}\leq a_{j+1}^2.$

Denote $(\bmu^*_h)^{\text{supp}(\bmu^*)\setminus S},(\bmu^*_h)^{\text{supp}(\bmu^*)\cap S}\in\mathbb{R}^p$ as the vectors which have the same values as $\bmu_h^*$ on coordinates $\text{supp}(\bmu^*)\setminus S$ and $\text{supp}(\bmu^*)\cap S$ respectively, and $0$ elsewhere. Then we have $\bmu^*_h=(\bmu^*_h)^{\text{supp}(\bmu^*)\setminus S}+(\bmu^*_h)^{\text{supp}(\bmu^*)\cap S}$ and $\langle \bmu_g,(\bmu^*_h)^{\text{supp}(\bmu^*)\setminus S} \rangle = \langle (\bmu^*_h)^{\text{supp}(\bmu^*)\cap S},(\bmu^*_h)^{\text{supp}(\bmu^*)\setminus S} \rangle = 0$. Thus, for $\mathcal{E}_1^K$ we have
\begin{align*}
    \sum_{g=1}^{K-m}n_g\left\|\bmu_g-\frac{\sum_{h=1}^Kn_{gh}(\bmu^*_h)^{\text{supp}(\bmu^*)\cap S}}{n_g}\right\|_2^2 + \sum_{g=1}^{K-m}n_g\left\|\frac{\sum_{h=1}^Kn_{gh}(\bmu^*_h)^{\text{supp}(\bmu^*)\setminus S}}{n_g}\right\|_2^2 + C_{\bl}&\leq a_{2}^2.
\end{align*}
We then denote 
$C_{\bl}'\overset{\Delta}{=}C_{\bl}+\sum_{g=1}^{K-m}n_g\left\|\frac{\sum_{h=1}^Kn_{gh}(\bmu^*_h)^{\text{supp}(\bmu^*)\setminus S}}{n_g}\right\|_2^2.$

Denote $\mathcal{F}_{1,m,S}^K(\bl)\overset{\Delta}{=}\{(\bmu_{-m},\bl_{-m})\in\mathcal{F}_{1,m}^K(\bl):\text{supp}(\bmu)=S,|S|\leq s\}$ for fixed $S\subset[p]$. Then we have $\mathcal{F}_{1,m}^K(\bl)=\bigcup_{S\subset[p],|S|\leq s}\mathcal{F}_{1,m,S}^K(\bl)$ and therefore $$N(a_1\epsilon/2,\mathcal{E}_1^K,\|\cdot\|_F)\leq\sum_{m=0}^K\sum_{\bl\in\mathcal{L}_{K-m}}\sum_{S\subset[p],|S|\leq s}N(a_1\epsilon/2,\mathcal{F}_{1,m,S}^K(\bl),\|\cdot\|_F).$$
Let
$$
\mathcal{I}_{1,m,S}^K(\bl)\overset{\Delta}{=}\left\{(\bmu,\bl_{-m})\in\mathbb{R}^{p\times (K-m)}\times\mathcal{L}_{K-m}:\sum_{g=1}^{K-m}n_g\|\bmu_g\|_2^2+C_{\bl}'\leq a_{2}^2, \text{supp}(\bmu)=S, |S|\leq s\right\}
$$
for fixed $\bl$ which induces $m$ empty clusters.
Note that there is an injective function $f:\mathcal{F}_{1,m,S}^K(\bl)\to\mathcal{I}_{1,m,S}^K(\bl)$ such that $\bmu_g\mapsto \bmu_g-\sum_{h=1}^Kn_{gh}(\bmu^*_h)^{^{\text{supp}(\bmu^*)\cap S}}/n_g$
for $g\in[K-m]$ and we know that $\mathcal{I}_{1,m,S}^K(\bl)$ is contained in an $s(K-m)$-dimensional ellipsoid with center $0$ and length of semi-axes $\{(a_{2}^2-C_{\bl}')/n_g\}_{g=1}^{K-m}$. Thus the volume of $\mathcal{F}_{1,m,S}^K(\bl)$ can be bounded 
\begin{align*}
    |\mathcal{F}_{1,m,S}^K(\bl)|&\leq |\mathcal{I}_{1,m,S}^K(\bl)|
    \leq \frac{\pi^{\frac{s(K-m)}{2}}}{\Gamma\left(\frac{s(K-m)}{2}+1\right)}\prod_{g=1}^{K-m}\frac{(a_{2}^2-C_{\bl}')^{s/2}}{n_g^{s/2}}
    \leq \frac{\pi^{\frac{s(K-m)}{2}}}{\Gamma\left(\frac{s(K-m)}{2}+1\right)}\frac{a_{2}^{s(K-m)}}{\prod_{g=1}^{K-m}n_g^{s/2}}
\end{align*}
where $\Gamma$ is the Euler's Gamma function.

Suppose 
$\mathcal{M}_{1,m,S}^K(\bl)$ is a maximal $a_1\epsilon/2$-packing of $\mathcal{F}_{1,m,S}^K(\bl)$ for fixed $\bl$ and $S\subset[p]$. Then for every $(\tilde{\bmu}_{-m},\bl_{-m})\in\mathcal{M}_{1,m,S}^K(\bl)$, consider $U_{1,m,S}((\tilde{\bmu}_{-m},\bl_{-m}),a_1\epsilon/4)\subset \mathcal{F}_{1,m,S}^K$.
We have $\|\bmu\bl^T-\tilde{\bmu}\bl^T\|_F^2=\sum_{g=1}^{K-m} n_g\|\bmu_g-\tilde{\bmu}_g\|_2^2\leq \frac{a_1^2\epsilon^2}{16}.$
Let 
$$\mathcal{U}_{1,m,S}(\bl)=\left\{(\bmu,\bl_{-m})\in \mathbb{R}^{p\times (K-m)}\times \mathcal{L}_{K-m}: \sum_{g=1}^{K-m}16n_g\|\bmu_g\|_2^2/(a_1^2\epsilon^2)\leq 1, \text{supp}(\bmu)=S,|S|\leq s\right\}$$
for fixed $\bl$ and $S\subset[p]$. Since $\bmu$ shares the same support as $\tilde{\bmu}_{-m}$ for $(\bmu,\bl_{-m})\in U_{1,m,S}((\tilde{\bmu}_{-m}\\
,\bl_{-m}),\frac{a_1\epsilon}{4})$, there exists an bijective function $f':U_{1,m,S}((\tilde{\bmu},\bl),a_1\epsilon/4)\to \mathcal{U}_{1,m,S}(\bl)$ such that $\bmu_g\mapsto\bmu_g-\tilde{\bmu}_g$ for $g\in[K-m]$. In addition, we know that $\mathcal{U}_{1,m,S}(\bl)$ is essentially an $s(K-m)$-dimensional ellipsoid with center $0$ and length of semi-axes $\{(a_1^2\epsilon^2)/16n_g\}_{g=1}^{K-m}$. Therefore, the volume
$|U_{1,m,S}(\left(\tilde{\bmu},\bl),\frac{a_1\epsilon}{4}\right)|= \frac{\pi^{\frac{s(K-m)}{2}}}{\Gamma\left(\frac{s(K-m)}{2}+1\right)}\prod_{g=1}^{K-m}\left(\frac{a_1\epsilon}{4\sqrt{n_g}}\right)^{s}.$
Note that the sets $U_{1,m,S}\left((\tilde{\bmu}_{-m},\bl_{-m}),\frac{a_1\epsilon}{4}\right)$ are disjoint when $\tilde{\bmu}_{-m}$ varies since $\mathcal{M}_{1,m,S}^K(\bl)$ is a packing. Then we have, 
\begin{align*}
    &N\left(\frac{a_1\epsilon}{2},\mathcal{F}_{1,m,S}(\bl),\|\cdot\|_F\right)\leq|\mathcal{M}_{1,m,S}^K(\bl)|\leq\frac{|\mathcal{F}_{1,m}^K|}{|U_{1,m,S}((\tilde{\bmu},\bl),a_1\epsilon/4)|}\\
    &\quad\leq\frac{\pi^{\frac{s(K-m)}{2}}/\Gamma\left(\frac{s(K-m)}{2}+1\right)}{\pi^{\frac{s(K-m)}{2}}/\Gamma\left(\frac{s(K-m)}{2}+1\right)} \left(\frac{2a_2}{\epsilon a_1}\right)^{s(K-m)}
    \leq \left(\frac{2}{\epsilon}\right)^{sK}\left(\frac{2\|(\bmu^*)(\bl^*)^T\|_F}{\Delta}\right)^{sK}.
\end{align*}
and therefore,
\begin{align*}
    &N(a_1\epsilon/2,\mathcal{E}_{1}^K,\|\cdot\|_F)\leq\sum_{m=0}^{K-1} \sum_{\bl\in\mathcal{L}_{K-m}}\sum_{S\in[p],|S|\leq s}|\mathcal{M}_{1,m,S}^K(\bl)|\\
    &\quad\leq KK^n{p\choose s}\left(\frac{2}{\epsilon}\right)^{sK}\left(\frac{2\|(\bmu^*)(\bl^*)^T\|_F}{\Delta}\right)^{sK}
    \lesssim KK^n{p\choose s}\left(\frac{2}{\epsilon}\right)^{sK}(sn)^{sK/2}.
\end{align*}
Note that $\Delta\geq\frac{1}{n^q}$ and $\|(\bmu^*)(\bl^*)^T\|_F^2=O(sn)$ by assumptions.
Then combining with the previous results we have
\begin{align*}
    \log\left(\sum_{j=-\infty}^{\infty}N(a_j\epsilon/2,\mathcal{E}_j^K,\|\cdot\|_F)\right)&\lesssim n\log K + s\log \frac{p}{s} + sK\log n+ sK\log \frac{6}{\epsilon}
\end{align*}
 
Therefore, by Corollary 8.5 in \cite{kosorok2008}, we have
\begin{align*}
    &\E_*\left[\sup_{(\widehat{\bmu}, \widehat{\bl})\in \Theta_K}\left\langle\bE, (\bSigma^*)^{\frac{1}{2}}\frac{\widehat{\bmu}\widehat{\bl}^T -  (\bmu^*) (\bl^*)^{T}}{\|\widehat{\bmu}\widehat{\bl}^T-(\bmu^*)(\bl^*)^T\|_F}\right\rangle_F\right]
    \lesssim \int_0^{\mathrm{diam}(\widetilde{\Theta}_K)}\sqrt{\log N(\epsilon, \widetilde{\Theta}_K, \|\cdot\|_F)} \mathrm{d}\epsilon\\
    &\quad\lesssim \int_0^2\sqrt{\log K+ s\log \frac{p}{s}+n\log K+\frac{sK}{2}\log n+sK\log\frac{6\lambda_{\max}(\bSigma^*)^{\frac{1}{2}}}{\epsilon}}d\epsilon\\
    &\quad\leq \sqrt{s\log \frac{p}{s}+n\log K+sK\log n} + 2\sqrt{sK}\int_0^1\sqrt{-\log u+\log (12\lambda_{\max}(\bSigma^*)^{\frac{1}{2}})}du\\
    &\quad\lesssim \sqrt{s\log \frac{p}{s}+n\log K}.
\end{align*}

\end{proof}

\section{Proofs for Section \ref{sec:theory}}\label{sec:proof_theory}
\subsection{Proof architecture}\label{sub:proof_architecture}

%
%
We first sketch the proof of Theorem \ref{thm1} by providing technical lemmas below.


\begin{lemma}\label{lemma2}
Under the conditions of Theorem \ref{thm1}, we have 
\[
\Pi\{\|\bmu \bl^T-(\bmu^*)(\bl^*)^T\|_F^2<s\log p\}\geq\exp\{-c(s\log p+n\log K^* )\}.
\]
\end{lemma}


Note that the prior of $\bmu$ is absolutely continuous with respect to the Lebesgue measure, which implies $|\mathrm{supp}(\bmu)|=p$ with probability 1. 
However, we expect most rows of $\bmu$ come from the ``spike" distribution \emph{a priori}, which implies the ``magnitude" of these rows are quite small with high prior probability. This motivates us to define a generalized notation of the support. Formally, for $\delta>0$, we define $\mathrm{supp}_\delta(\bmu)\overset{\Delta}{=}\{j\in [p]:\|\bmu_{j*}\|_1\leq \delta\}$ as the soft support of $\bmu$ with threshold $\delta$, where $\bmu_{j*}$ represents the $j$th row of $\bmu$. Let $(\bmu_k)_{S_\delta}=(\mu_{jk}:j\in\mathrm{supp}_\delta(\bmu))\in\mathbb{R}^{|\mathrm{supp}_\delta(\bmu)|}$ denotes the sub-vector of $\bmu_k$ whose coordinates are in $\mathrm{supp}_\delta(\bmu)$. It is conceivable that for small $\delta$, the size of the soft support of $\bmu$ is small compared with $p$ with high prior probability. This heuristics is formalized through the following lemma.

\begin{lemma}\label{lemma_soft_supp}
Given $K$, suppose $\bmu\in\mathbb{R}^{p\times K}$ follows the prior specification (\ref{prior2}) and (\ref{prior6}) with some hyperparameters $\kappa>0$, $\alpha>1$, $\lambda_0\gg \lambda_1>0$ and let $\delta=(1+\kappa)\log p/\lambda_0$.
Assume $K\log\log p\leq \log p$. Then given $K$, we have, for $\bmu\in\mathbb{R}^{p\times K}$,
\[
\Pi\left(|\mathrm{supp}_{\delta}(\bmu)|\geq \beta\left(s+\frac{n\log K}{\log p}\right)\right)\leq \exp\left(-c(s\log p+n\log K)\right)
\]
for some constant $\beta, c>0$.

\end{lemma}


\begin{lemma}\label{lemma1}
Let $(\bmu',\bl')\in\Theta = \bigcup_{K = 1}^{K_{\max}}\mathbb{R}^{p\times K}\times\mathcal{L}_K$ be such that $(\bmu')(\bl')^T\neq (\bmu^*)(\bl^*)^T$ and consider 
\[
\mathcal{E}=\{(\bmu, \bl):\|\bmu \bl^T-(\bmu')(\bl')^T\|_F\leq\delta\|(\bmu')(\bl')^T-(\bmu^*)(\bl^*)^T\|_F\},
\] 
for some sufficiently small constant $\delta$ such that $0 < \delta < \frac{\|\bSigma^*\|_2}{2(\|\bSigma^*\|_2+2)}$. Assume the conditions of Theorem \ref{thm1} hold. 
Let $p_0(\by_i)$ and $p^*(\by_i)$ be the density function of $N(((\bmu^*)(\bl^*)^T)_i, \bI_p)$ and $N(((\bmu^*)(\bl^*)^T)_i, \bSigma^*)$ respectively. Denote $p_0(\by) = \prod_{i=1}^n p_0(\by_i)$ and $p^*(\by) = \prod_{i=1}^n p^*(\by_i)$.
Then there exists a test function $\phi_n$ such that 
\begin{align*}
\E_*\phi_n &\leq\exp\{-c_1\|(\bmu')(\bl')^T-(\bmu^*)(\bl^*)^T\|_F^2\},\\
\sup_{(\bmu, \bl)\in \mathcal{E}}\E_{(\bmu, \bl, \bI_p)}\left(\frac{p^*(\by)}{p_0(\by)}(1-\phi_n)\right) &\leq \exp\left\{-c_2 \|(\bmu')(\bl')^T-(\bmu^*)(\bl^*)^T\|_F^2\right\}
\end{align*}
where $c_1,c_2 > 0$ are some positive constants that are independent of $n$.
\end{lemma}
\begin{lemma}
\label{lemma4}
Let 
\begin{align*}
\mathcal{F}_n = \bigcup_{K = 1}^{K_{\max}}\bigg\{\bmu \bl^T:&\bmu\in\mathbb{R}^{p\times K}, |\mathrm{supp}_\delta(\bmu)|\leq \beta \left(s+\frac{n\log K}{\log p}\right),
\max_{k\in [K]}\|(\bmu_k)_{S_\delta}\|_\infty\leq a_n, \bl\in\mathcal{L}_K \bigg\}
\end{align*} 
where $\beta$ and $\delta$ are defined as in Lemma \ref{lemma_soft_supp}
and $a_n=(s\log p+n\log K_{\max})n^{\gamma}$ for some constant $\gamma>0$. Denote $N(\epsilon_n,\mathcal{F}_n,d)$ as the covering number of $\mathcal{F}_n$ with respect to the metric $d(\bA,\bB)={\|\bA-\bB\|_F}/{\sqrt{n}}$. Suppose 
$\epsilon_n^2=(s\log p+n\log K_{\max})/n.$ Then we have
$$N(\epsilon_n,\mathcal{F}_n,d)\leq \exp(cn\epsilon_n^2)$$ for some constant $c$.
\end{lemma}

\begin{lemma}
\label{lemma3}
Let $\delta$ be defined as in Lemma \ref{lemma_soft_supp} and $\mathcal{F}_n$ be defined as in Lemma \ref{lemma4}.
Assume the conditions of Theorem \ref{thm1} hold. Then we have
\[
\Pi(\mathcal{F}_n^c)\leq\exp\{-c(s\log p+n\log K^*)\}
\] for some constant $c$.
\end{lemma}

\subsection{Proofs of the auxiliary lemmas}

In this subsection, we provide the detailed proofs of the lemmas appearing in Section \ref{sub:proof_architecture}. 

\begin{proof}[proof of Lemma \ref{lemma2}]
The proof of this lemma is based on a mofidication of that of Lemma 3.1 in \cite{xie2022bayesian}. 
denote $\epsilon^2 = (s\log p)/n$. 
First by conditioning on the event $\{(\bmu,\bl)\in\Theta:\bl = \bl^*, K = K^*\}$, we have that 
\begin{align*}
    &\Pi\{\|\bmu \bl^T-(\bmu^*)(\bl^*)^T\|_F^2<n\epsilon^2\mid \bl = \bl^*, K = K^*\}
    \geq
    \Pi\left(\bigcap_{k = 1}^{K^*}\left\{\|\bmu_k-\bmu_k^*\|_2^2\leq\epsilon^2\right\}\right),
\end{align*}
where $n_k^*$ is the number of observations assigned to the $k$th cluster according to the cluster assignment matrix $\bl^*$. Now we focus on the prior distribution of $\bmu_k$. 
denote $S_0$ the true sparsity of $\bmu^*$. Note that for each $k\in [K^*]$, $
\|\bmu_k-\bmu_k^*\|_2\leq\|(\bmu_k)_{S_0}-(\bmu_k^*)_{S_0}\|_2+\|(\bmu_k)_{S_0^c}\|_2,
$
\begin{align*}
\prob\left(\bigcap_{k = 1}^{K^*}\left\{\|\bmu_k-\bmu_k^*\|_2^2<\epsilon^2\right\}\right)
\geq\Pi\left(
\bigcap_{k = 1}^{K^*}\left\{
\|(\bmu_k)_{S_0}-(\bmu_k^*)_{S_0}\|_2<\frac{\epsilon}{2}\right\}\cap\left\{\|(\bmu_k)_{S_0^c}\|_2<\frac{\epsilon}{2}\right\}\right).
\end{align*}
Now we introduce the latent random variable $\xi_j\sim \text{Bernoulli}(\theta)$ such that 
$|\mu_{jk}|\mid\xi_j\sim(1-\xi_j)\text{Exp}(\lambda_0)+\xi_j\text{Exp}(\lambda_1)$ independently for all $j\in[p]$ and $k\in[K^*]$,
where $\mu_{jk}$ is the $j$th coordinate of $\bmu_k$. 
Recall that a Laplace distribution can be represented as a scale-mixture of normals as 
$(\mu_{jk}\mid\phi_j,\xi_j )\sim\mathcal{N}\left(0,\phi_j/\lambda_{\xi_j}^2\right)$ with $\phi_j\sim\text{Exp}\left(1/2\right).$
Define the event 
$\mathcal{A} = \bigcap_{j\in S_0}\{\xi_j=1\}\bigcap_{j\in S_0^c}\{\xi_j=0\}\bigcap_{j\in S_0}\{1<\phi_j<4\}.$
Note that given $\mathcal{A}$, the entries of $\bmu$ are independent. 
Conditioning on $\mathcal{A}$, we have $|\mu_{jk}|\sim\text{Exp}(\lambda_0)$ for $j\in S_0^c$, which implies
\begin{align*}
    &\prod_{k = 1}^{K^*}\Pi\left(\|(\bmu_k)_{S_0^c}\|_2<\frac{\epsilon}{2}\mathrel{\Big|}\mathcal{A}\right)
    \geq \prod_{k = 1}^{K^*}\prod_{j\in S_0^c}\Pi\left(|\mu_{jk}|<\frac{\epsilon}{2\sqrt{p}}\mathrel{\Big|}\mathcal{A}\right)\\
    &\quad=\left\{1-\exp\left(-\frac{\lambda_0\epsilon}{2\sqrt{p}}\right)\right\}^{K^*(p-s)}
    \geq\left(1-\frac{s}{p}\right)^{pK^*}
    \geq\exp\{-\log(2e)sK^*\}.
\end{align*}
Here we use the inequality $(1-x)^{1/x}\geq1/(2e)$ for $x\in(0,1/2)$ and the fact 
$\lambda_0\geq 2\log\frac{p}{s}\sqrt{\frac{np}{s\log p}}\Longrightarrow \frac{\lambda_0}{2}\sqrt{\frac{s\log p}{np}}\geq\log\frac{p}{s}.$
Next, by Anderson's lemma (see Lemma 1.4 in supporting document for \cite{pati2014posterior}), for each $k\in [K^*]$, conditioning on $\mathcal{A}$ (which guarantees that $1\leq\phi_j\leq 4$ for all $j\in S_0$), we have
\begin{align*}
    &\Pi\left(\|(\bmu_k)_{S_0}-(\bmu_k^*)_{S_0}\|_2<\frac{\epsilon}{2}\mathrel{\Big|}\mathcal{A}\right)
    \geq\exp\left(-\frac{1}{2}\sum_{j\in S_0}\frac{|\mu_{jk}^*|^2\lambda_1^2}{\phi_j}\right)\Pi\left(\|(\bmu_k)_{S_0}\|_2<\frac{\epsilon}{2}\mathrel{\Big|}\mathcal{A}\right)
    \\&\quad
    \geq\exp\left(-\frac{1}{2}\sum_{j\in S_0}\frac{|\mu_{jk}^*|^2\lambda_1^2}{\phi_j}\right)\prod_{j\in S_0}\left\{2\Phi\left(\frac{\epsilon\lambda_1}{2\sqrt{s\phi_j}}\right)-1\right\}\\
    &\quad\geq
    \exp\left(-\frac{1}{2}\lambda_1^2\|(\bmu_k^*)_{S_0}\|^2_2\right)\left\{2\Phi\left(\frac{\epsilon\lambda_1}{4\sqrt{s}}\right)-1\right\}^s
    \geq\exp\left\{-\frac{1}{2}\lambda_1^2\|\bmu_k^*\|_2^2 - s\left(1 + \left|\log\frac{\epsilon\lambda_1}{4\sqrt{s}}\right|\right)\right\},
\end{align*}
where we use $\log(2\Phi(x)-1)\geq-1-|\log x|$ for small $x>0$ in the last inequality.
Since 
\begin{align*}
sK^*\left|\log\frac{\epsilon\lambda_1}{4\sqrt{s}}\right| 
& = sK^*\left|\log\left(\frac{\lambda_1}{4}\sqrt{\frac{\log p}{n}}\right)\right|
\leq sK^*\left|\log\frac{\lambda_1}{4}\right| + \frac{sK^*}{2}\left|\log\frac{\log p}{n}\right|
\leq c_1'sK^*\log n\leq c_1s\log p
\end{align*}
for some constant $c_1',c_1 > 0$,
it follows that
$\prod_{k = 1}^{K^*}\Pi\left(\|\bmu_k-\bmu_k^*\|_2^2 < \epsilon\mathrel{\Big|}\mathcal{A}\right)\geq\exp(-c_2s\log p)$
for some constant $c_2\geq \max(c_1,C)>0$ given
$\sum_{k = 1}^{K^*}\lambda_1^2\|\bmu_k^*\|_2^2\leq Cs\log p$.
We next consider the prior probability of the event $\mathcal{A}$. First note that
$\prod_{j\in S_0}\Pi(1<\phi_j<4)\geq\exp(-c_3s)$
for some constant $c_3>0$ by the definition of the exponential distribution. Since the prior of $\theta$ is $\text{Beta}(1,\beta_\theta)$ where $\beta_\theta=p^{1+\kappa}\log p$, for some constant $c_4>1+\kappa>0$,

\begin{align*}
    &\Pi\left(\bigcap_{i\in S_0}\{\xi_i=1\}\bigcap_{i\in S_0^c}\{\xi_i=0\}\right)=
    \int_0^1\theta^s(1-\theta)^{p-s}\Pi(d\theta)\\
    &\quad= \frac{\Gamma(\beta_\theta+1)}{\Gamma(\beta_\theta)} \int_0^1\theta^s(1-\theta)^{p+\beta_\theta-s}d\theta 
    = \frac{\Gamma(s+1)\Gamma(p+\beta_\theta-s)}{\Gamma(p+\beta_\theta+1)}\frac{\Gamma(\beta_\theta+1)}{\Gamma(\beta_\theta)}\\
    &\quad\geq \exp\left(-s\log(p+\beta_\theta)+\log \beta_\theta\right)
    \geq \exp(-s\log (2\beta_\theta))\\
    &\quad\geq \exp\left(-s\log\log p -(1+\kappa)s\log p-s\log 2\right)
    \geq \exp(-c_4s\log p).
\end{align*}

Hence, we obtain that
$\Pi(\mathcal{A})\geq\exp(-c_3s - c_4s\log p)$
and therefore,
\begin{align*}
\Pi\left(\bigcap_{k = 1}^{K^*}\left\{\|\bmu_k-\bmu_k^*\|_2^2 < \frac{\epsilon^2}{K^*}\right\}\right)
&\geq\prod_{k = 1}^{K^*}\Pi\left(\|\bmu_k-\bmu_k^*\|_2<\frac{\epsilon}{\sqrt{K^*}}\mathrel{\Big|}\mathcal{A}\right)\Pi(\mathcal{A})\geq\exp(-c_5s\log p)
\end{align*}
for some constant $c_5>\max(c_2,c_4)>0$.
Thus, we have
\begin{align*}
\Pi\left(\|\bmu \bl^T-(\bmu^*)(\bl^*)^T\|_F^2\leq s\log p\mid \bl=\bl^*,K=K^*\right)\geq \exp(-c_5s\log p).
\end{align*}
Then consider $\Pi(\bl=\bl^*\mid K=K^*)$ for Multinomial-Dirichlet model. Let $\bw=(w_1,\dots,w_{K^*})^T$. Define integers $\alpha_r=\lfloor \alpha\rfloor$ and $\beta_r = \lfloor K^*\alpha \rfloor$. Then we know $\lceil \alpha \rceil = \alpha_r +1$ and $\lceil K^*\alpha \rceil = \beta_r+1$. 
Note that the gamma function $\Gamma(x)$ is strictly increasing for $x> 2$ and
we have $\Gamma(x)\leq 1$ for $1\leq x\leq 2$. 
Thus, we have, 
    for $K^*\alpha> 2$,  
    \[
    \frac{\Gamma(K^*\alpha)}{\Gamma(K^*\alpha+n)}\geq
    {
    \frac{\Gamma(\beta_r)}{\Gamma(\beta_r + 1 + n)}
    =
    \frac{(\beta_r - 1)!}{(\beta_r + n)!}}
    ;
    \]
    for $1\leq K^*\alpha\leq 2$, 
    \[
    \frac{\Gamma(K^*\alpha)}{\Gamma(K^*\alpha+n)}\geq{\frac{\min_{x\in[1,2]}\Gamma(x)}{(\beta_r+n)!}=\left(\min_{x\in[1,2]}\Gamma(x)\right)\frac{\beta_r!}{(\beta_r+n)!}}
    .
    \]
Similarly, we also have,
    for $\alpha>2$,  
    \[
    \frac{\Gamma(\alpha+n_i)}{\Gamma(\alpha)}
    {
    \geq\frac{\Gamma(\alpha_r + n_i)}{\Gamma(\alpha_r + 1)}
    \geq\frac{(\alpha_r+n_i - 1)!}{\alpha_r!}
    }
    ;
    \]
    for $1\leq\alpha\leq2$, 
    \[
    \frac{\Gamma(\alpha+n_i)}{\Gamma(\alpha)}
    {
    \geq \frac{(\alpha_r+n_i - 1)!}{1}=\frac{(\alpha_r+n_i - 1)!}{\alpha_r!}
    }.
    \]
Therefore, we conclude that for $\alpha\geq 1$, some constants $C_1, C_2 > 0$,
\[
\frac{\Gamma(K^*\alpha)}{\Gamma(K^*\alpha+n)}{\geq C_1\frac{(\beta_r - 1)!}{(\beta_r+n)!}}\quad\mbox{and}\quad
\frac{\Gamma(\alpha+n_i)}{\Gamma(\alpha)}{\geq C_2\frac{(\alpha_r+n_i - 1)!}{\alpha_r!}}.
\]

By integrating out $\bw$ we have
\begin{align*}
    &\Pi(\bl=\bl^*\mid K=K^*)
    =\int \Pi(\bl=\bl^*\mid \bw,K=K^*)d\Pi(\bw\mid K=K^*)\\
    &\quad=\int_0^1\cdots\int_0^1 \prod_{i=1}^nw_{z_i^*}\frac{\Gamma(K^*\alpha)}{(\Gamma(\alpha))^{K^*}}\prod_{j=1}^{K^*} w_j^{\alpha-1}dw_1\cdots dw_{K^*}
    =\frac{\Gamma(K^*\alpha)}{(\Gamma(\alpha))^{K^*}}\frac{\prod_{i=1}^{K^*}\Gamma(\alpha+n_i)}{\Gamma(K^*\alpha+n)}\\
    &\quad\geq {C_1\frac{(\beta_r - 1)!}{(\beta_r+n)!}(C_2)^{K^*}\prod_{i=1}^{K^*}\frac{(\alpha_r+n_i - 1)!}{\alpha_r!}}\\
    &\quad=C_1(C_2)^{K^*}{\frac{(\beta_r - 1)!}{(\beta_r+n)!}\frac{(K^*\alpha_r+n - K^*)!}{(\alpha_r!)^{K^*}}\frac{1}{{K^*\alpha_r+n - K^*\choose \alpha_r+n_1 - 1,\cdots,\alpha_r+n_{K^*} - 1}}}
    .
\end{align*}
where 
${a\choose a_1,\cdots,a_{K^*}}=\frac{a!}{a_1!a_2!\cdots a_{K^*}!}$ is the multinomial coefficient for $a_1+a_2+\cdots+a_{K^*}=a$. We know that the function 
$(a_1,\cdots,a_{K^*})\mapsto {a\choose a_1,\cdots,a_{K^*}}$
achieves the maximum when
$a_1,\dots,a_{K^*}$ are as close as possible to each other. 
Formally, let $a=qK^*+r$ where $q\overset{\Delta}{=}a\mod K^*=\lfloor {a}/{K^*}\rfloor$ and $0\leq r<K^*$, then the multinomial coefficient is maximized when $a_1=\dots=a_r=q+1$ and $a_{r+1}=\dots=a_{K^*}=q$, and hence the maximal value is ${a!}/\{q!^{K^*-r}(q+1)!^r\}$.
Then the preceding expression achieves the minimum when $n_1=\dots=n_{r}=\lfloor {n}/{K^*}\rfloor +1$ and $n_{r+1}=\dots=n_{K^*}=\lfloor {n}/{K^*}\rfloor$ where $r=n-\lfloor {n}/{K^*}\rfloor K^*$. Note that 
$
\frac{\alpha_r}{\beta_r}=\frac{\lfloor \alpha\rfloor}{\lfloor K^*\alpha\rfloor}\geq \frac{\lfloor\alpha\rfloor}{(\lfloor\alpha\rfloor+1)K^*}\geq\frac{1}{2K^*}.
$
So 
\begin{align*}
    &\Pi(\bl=\bl^*\mid K=K^*)
    \geq C_1(C_2)^{K^*} {\frac{\{(\alpha_r+1)(\alpha_r+2)\cdots(\alpha_r+\lfloor\frac{n}{K^*}\rfloor - 1)\}^{K^*}(\lfloor\frac{n}{K^*}\rfloor + \alpha_r)^r}{\beta_r(\beta_r+1)(\beta_r+2)\cdots(\beta_r+n)}}\\
    &\quad\geq C_1(C_2)^{K^*}
    \frac{(\alpha_r+1)^{K^*}(\alpha_r+2)^{K^*}\cdots(\alpha_r+\lfloor\frac{n}{K^*}\rfloor - 1)^{K^*}}{(\beta_r+2K^*)^{K^*}(\beta_r+4K^*)^{K^*}\cdots(\beta_r+2\lfloor\frac{n}{K^*}\rfloor K^* - 2K^*)^{K^*}}
    \times\frac{1}{(\beta_r + n)^{r+K^*}}
    \\
    &\quad\geq C_1(C_2)^{K^*}\left(\frac{1}{2K^*}\right)^{\lfloor\frac{n}{K^*}\rfloor K^*}{\times\frac{1}{(\beta_r + n)^{r+K^*}}}.
\end{align*}
Also note the prior of $K$ is a truncated Poisson, so we have
\begin{align*}
    \Pi(K=K^*)&\geq\frac{e^{-\lambda}\lambda^{K^*}}{K^*!}
    \geq \exp(-\lambda+K^*\log\lambda-K^*\log K^*)
    \geq \exp(-2K^*\log K^*).
\end{align*}
Therefore, for some constant $c>0$,
\begin{align*}
    &\Pi\left(\|\bmu \bl^T-(\bmu^*)(\bl^*)^T\|_F^2\leq s\log p\right)\\
    &\quad\geq \Pi\left(\|\bmu \bl^T-(\bmu^*)(\bl^*)^T\|_F^2\leq s\log p\mid \bl=\bl^*,K=K^*\right)
    \times \Pi(\bl=\bl^*\mid K=K^*)\Pi(K=K^*)\\
    &\quad\geq \exp\{-c(s\log p+n\log K^*)\}.
\end{align*}
\end{proof}

\begin{proof}[proof of Lemma \ref{lemma_soft_supp}]
We denote $\bmu_{j*}$  the $j$th row of $\bmu$ for $j=1,\dots,p$.
Note that from the prior model, we have $|\mu_{ji}|\mid \xi_j\sim(1-\xi_j)\text{Exp}(\lambda_0)+\xi_j\text{Exp}(\lambda_1)$
where $\text{Exp}(\lambda)$ is the exponential distribution with parameter $\lambda$. Then we have
$\|\bmu_{j*}\|_1\mid \xi_j\sim (1-\xi_j)\text{Gamma}(K,\lambda_0)+\xi_j\text{Gamma}(K,\lambda_1)$
where $\text{Gamma}(K,\lambda)$ is the Gamma distribution with shape $K$ and rate $\lambda$. Thus, by the change of variable $u=\lambda_0 x$ and conditioning on the event $A=\{\theta\leq c_1(s+n\log K/\log p)/p^{1+\kappa}\}$ for some constant $c_1>0$, we have
\begin{align*}
    &\Pi\left(\|\bmu_{j*}\|_1>\delta\right) \leq (1-\theta)\int_{\delta}^\infty\frac{\lambda_0^K}{\Gamma(K)}x^{K-1}e^{-\lambda_0 x}dx+\theta\\
    &\quad = (1-\theta)\frac{\lambda_0^K}{\Gamma(K)}\frac{1}{\lambda_0^K}\int_{\lambda_0\delta}^\infty u^{K-1}e^{-u}du + \theta
    <\frac{1}{\Gamma(K)}(\lambda_0\delta)^{K}e^{-\lambda_0\delta} + c_1\frac{s+\frac{n\log K}{\log p}}{p^{1+\kappa}}\\
    &\quad<(\lambda_0\delta)^Ke^{-\lambda_0\delta} + c_1\frac{s+\frac{n\log K}{\log p}}{p^{1+\kappa}}
    \leq \exp(K\log(1+\kappa) + K\log\log p - (1+\kappa)\log p) + c_1\frac{s+\frac{n\log K}{\log p}}{p^{1+\kappa}}\\
    &\quad\leq\exp(-\kappa\log p)+c_1\frac{s+\frac{n\log K}{\log p}}{p^{1+\kappa}}
    \leq (1+c_1)\frac{s+\frac{n\log K}{\log p}}{p^{1+\kappa}}
\end{align*}
for sufficiently large $n$.
Note that the second inequality comes form the result of \cite{natalini2000inequalities} about upper incomplete gamma function for $\lambda_0\delta = (1+\kappa)\log p>K+1$. 
The fourth inequality is due to the fact that the function $x\mapsto x^Ke^{-x}$ is decreasing when $x>K$. 
Note that \cite{hagerup1990guided} stated a version of Chernoff's inequality for binomial distributions that 
    $\prob(X>ap)\leq\left(\left(\frac{q}{a}\right)^a\exp(a)\right)^p$ if  $X\sim\text{Binomial}(p,q)$ and $q\leq a<1$. 
Then over the event $A=\{\theta\leq c_1(s+n\log K/\log p)/p^{1+\kappa}\}$ we have
\begin{align*}
    &\Pi\left(|\text{supp}_\delta(\bmu)|>\beta\left(s+\frac{n\log K}{\log p}\right)\mid A\right)\\
    &\quad\leq \exp\left(-\beta \left(s+\frac{n\log K}{\log p}\right)\log \frac{\beta \left(s+\frac{n\log K}{\log p}\right)}{pq}+\beta \left(s+\frac{n\log K}{\log p}\right)\right)\\
    &\quad\leq \exp\left(-\beta \left(s+\frac{n\log K}{\log p}\right)\log (\beta (1+c_1) p^{\kappa})\right)\\
    &\quad\leq \exp\left(-\beta\left(\kappa s\log p+s\log (\beta (1+c_1))+\kappa n\log K + \frac{n\log K}{\log p}\log (\beta (1+c_1))\right)\right)\\
    &\quad\leq \exp\left(-c_2(s\log p+n\log K)\right)
\end{align*}
for some constant $0<c_2\leq \beta \kappa$.
Consider the event $A$, we calculate the prior probability of $A^c$. Let $\beta_\theta=p^{1+\kappa}\log p$, we have
\begin{align*}
    &\Pi\left(\theta>c_1\frac{s\log p+n\log K}{p^{1+\kappa}\log p}\right)=\int_{c_1\frac{s\log p+n\log K}{p^{1+\kappa}\log p}}^1\frac{\Gamma(\beta_\theta+1)}{\Gamma(\beta_\theta)}(1-\theta)^{\beta_\theta-1}d\theta\\
    &\quad=\left(1-c_1\frac{s\log p+n\log K}{p^{1+\kappa}\log p}\right)^{\beta_\theta}
    \leq \exp\left(-\beta_\theta c_1\frac{s\log p+n\log K}{p^{1+\kappa}\log p}\right)
    \leq\exp\left(-c_1(s\log p+n\log K)\right).
\end{align*}
Therefore, for some constant $0<c\leq\min(c_1,c_2)$ we have
\begin{align*}
    &\Pi\left(|\text{supp}_\delta(\bmu)|>\beta\left(s+ \frac{n\log K}{\log p}\right)\right)=\int_0^1\Pi\left(|\text{supp}_\delta(\bmu)|>\beta \left(s+\frac{n\log K}{\log p}\right)\mid A\right)\Pi(d\theta)\\
    &\quad\leq\int_0^{c_1\frac{s\log p+n\log K}{p^{1+\kappa}\log p}}\Pi(|\text{supp}_\delta(\bmu)|>\beta s\mid A)\Pi(d\theta)
    +\Pi\left(\theta>c_1\frac{s\log p+n\log K}{p^{1+\kappa}\log p}\right)\\
    &\quad\leq \exp\left(-c(s\log p+n\log K)\right).
\end{align*}
\end{proof}

\begin{proof}[proof of Lemma \ref{lemma1}]
Construct a likelihood ratio test $\phi_n=\mathbbm{1}\{\|\by-(\bmu^*) (\bl^*)^T\|_F^2-\|\by-(\bmu')(\bl')^T\|_F^2\geq\rho\}$ where 
$\rho= \left(2m - 1\right)\|(\bmu^*)(\bl^*)^T-(\bmu') (\bl')^T\|_F^2$ for constant $m = \frac{\|\bSigma^*\|_2 - \sqrt{\|\bSigma^*\|_2^2 - \|\bSigma^*\|_2}}{2} > 0$ if $\|\bSigma^*\|_2 \geq 1$ and $m = \frac{\|\bSigma^*\|_2}{2}$ if $\|\bSigma^*\|_2 < 1$. 
For the type I error probability, we have under the true model, $\by_i-((\bmu^*)(\bl^*)^T)_i\sim N(0,\bSigma^*)$ for each $i$. Thus, by Hoeffding's inequality for sub-Gaussian random variables,
\begin{align*}
    &\prob_{*}\left(\|\by-(\bmu^*)(\bl^*)^T\|_F^2-\|\by-(\bmu')(\bl')^T\|_F^2\geq\rho\right)\\
    &\quad=\prob_{*}{\Bigg(}\sum_{i=1}^n (\by_i-((\bmu^*)(\bl^*)^T)_i)^T((\bmu')(\bl')^T-(\bmu^*)(\bl^*)^T)_i\geq
    m \|(\bmu^*)(\bl^*)^T-(\bmu') (\bl')^T\|_F^2{\Bigg)}\\
    &\quad\leq2\exp\left(-\frac{c_1 m^2 \|(\bmu')(\bl')^T-(\bmu^*)(\bl^*)^T\|_F^4}{\|(\bSigma^*)^{\frac{1}{2}}((\bmu')(\bl')^T-(\bmu^*)(\bl^*)^T)\|_F^2} \right)\\
    &\quad \leq\exp\left(-c_1'\|(\bmu')(\bl')^T-(\bmu^*)(\bl^*)^T\|_F^2\right)
\end{align*}
for some constant $c_1, c_1' > 0$ since $\frac{m^2}{\|\bSigma^*\|_2} \geq \frac{1}{\|\bSigma^*\|_2} \geq \frac{1}{M_{\bSigma}} > 0$. 
For the type II error probability, we have 
\begin{align*}
    C&\overset{\Delta}{=}\int \frac{p^*(\by)p(\by)}{p_0(\by)}d\by = \E_*\frac{p(\by)}{p_0(\by)} = \E_* \exp\left(-\frac{1}{2}(\|\by - \bmu\bl^T\|_F^2 - \|\by - (\bmu^*)(\bl^*)^T\|_F^2)\right)\\
    &= \E_* \exp \left( \sum_{i=1}^n (\by_i - ((\bmu^*)(\bl^*)^T)_i)^T( \bmu\bl^T - (\bmu^*)(\bl^*)^T )_i - \frac{1}{2} \|\bmu\bl^T - (\bmu^*)(\bl^*)^T\|_F^2 \right) \\
    &= \exp\left( \frac{1}{2}\| (\bSigma^*)^{\frac{1}{2}}(\bmu\bl^T - (\bmu^*)(\bl^*)^T)\|_F^2 - \frac{1}{2} \|\bmu\bl^T - (\bmu^*)(\bl^*)^T\|_F^2 \right) \\
    &= \exp\left( \frac{1}{2}\| (\bSigma^*)^{\frac{1}{2}}(\bmu\bl^T - (\bmu^*)(\bl^*)^T)\|_F^2) - \frac{1}{2} \|\bmu\bl^T - (\bmu^*)(\bl^*)^T\|_F^2 \right) \\
    & \leq \exp\left( \frac{1}{2}\left(\| \bSigma^*\|_2 - 1\right) \|\bmu\bl^T - (\bmu^*)(\bl^*)^T\|_F^2  \right)\\
    & \leq \exp\left( \frac{1}{2}\left(\| \bSigma^*\|_2 - 1\right) (\|\bmu\bl^T - (\bmu')(\bl')^T\|_F + \|(\bmu')(\bl')^T - (\bmu^*)(\bl^*)^T\|_F)^2  \right)\\
    & \leq \exp\left( \frac{1}{2}(1+\delta)^2 \left(\| \bSigma^*\|_2 - 1\right)\|(\bmu')(\bl')^T - (\bmu^*)(\bl^*)^T\|_F^2)  \right) < \infty .
\end{align*}
Therefore, $\E_{(\bmu, \bl, \bI_p)}\left(\frac{p^*(\by)}{p_0(\by)}(1-\phi_n)\right) = C\prob_{\bZ}\left( \|\bZ-(\bmu^*)(\bl^*)^T\|_F^2-\|\bZ-(\bmu')(\bl')^T\|_F^2<\rho \right)$
where $\bZ$ has density $\frac{p(\bZ)p^*(\bZ)}{C p_0(\bZ)}$. We then have
\begin{align*}
    &\frac{p(\bZ)p^*(\bZ)}{p_0(\bZ)}\\
    &\quad\propto \exp\left( -\frac{1}{2}\left( \|(\bSigma^*)^{-\frac{1}{2}}(\bZ - (\bmu^*)(\bl^*)^T)\|_F^2 + \|\bZ - \bmu\bl^T\|_F^2 - \|\bZ - (\bmu^*)(\bl^*)^T\|_F^2 \right) \right)\\
    &\quad= \exp{\Bigg(} -\frac{1}{2}{\Big(} \|(\bSigma^*)^{-\frac{1}{2}}(\bZ - (\bmu^*)(\bl^*)^T)\|_F^2 + \|(\bmu^*)(\bl^*)^T - \bmu\bl^T\|_F^2\\ 
    &\quad\quad\quad\quad\quad+ 2\langle (\bSigma^*)^{-\frac{1}{2}}(\bZ - (\bmu^*)(\bl^*)), (\bSigma^*)^{\frac{1}{2}}((\bmu^*)(\bl^*)^T - \bmu\bl^T) \rangle_F 
    {\Big)}{\Bigg)}\\
    &\quad\propto \exp\left( -\frac{1}{2}\left( \|(\bSigma^*)^{-\frac{1}{2}}\left(\bZ - (\bmu^*)(\bl^*)^T + \bSigma^*((\bmu^*)(\bl^*)^T - \bmu\bl^T)\right)\|_F^2 \right) \right)
\end{align*}
Thus $\bZ_i \sim N( ((\bmu^*)(\bl^*)^T)_i - \bSigma^*((\bmu^*)(\bl^*)^T - \bmu\bl^T)_i, \bSigma^*)$ and $\bZ_i - (\bmu\bl^T)_i \sim N((\bI_p - \bSigma^*)((\bmu^*)(\bl^*)^T - \bmu\bl^T)_i, \bSigma^*)$ for $i\in [n]$.
Then, by the same argument as type I error,
\begin{align*}
    &\prob_{\bZ}(\|\bZ-(\bmu^*)(\bl^*)^T\|_F^2-\|\bZ-(\bmu')(\bl')^T\|_F^2<\rho)\\
    &\quad=\prob_{\bZ}\left(\sum_{i=1}^n(\bZ_i-((\bmu^*)(\bl^*)^T)_i)^T((\bmu')(\bl')^T-(\bmu^*)(\bl^*)^T)_i<
    m \|(\bmu^*)(\bl^*)^T-(\bmu') (\bl')^T\|_F^2\right)\\
    &\quad=\prob_{\bZ}\left(\sum_{i=1}^n(\bZ_i - (\bmu \bl^T)_i)^T((\bmu') (\bl')^T-(\bmu^*)(\bl^*)^T)_i<
    m \|(\bmu^*)(\bl^*)^T-(\bmu') (\bl')^T\|_F^2
    \right.\\
    &\quad\quad\quad\quad\quad-\sum_{i=1}^n(\bmu \bl^T-(\bmu^*)(\bl^*)^T)_i^T ((\bmu')(\bl')^T-(\bmu^*)(\bl^*)^T)_i{\Bigg)}\\
    &\quad\leq\prob_{\bZ}\left(\sum_{i=1}^n(\bZ_i - (\bmu \bl^T)_i)^T ((\bmu') (\bl')^T-(\bmu^*)(\bl^*)^T)_i < \left(\delta - 1 + m \right)\|(\bmu^*)(\bl^*)^T-(\bmu')(\bl')^T\|_F^2\right)
\end{align*}
Note that the inequality comes from  Cauchy-Schwarz inequality,
\begin{align*}
    &\langle\bmu \bl^T-(\bmu^*)(\bl^*)^T, (\bmu^*)(\bl^*)^T-(\bmu')(\bl')^T\rangle_F\\
    &\quad=\langle\bmu \bl^T-(\bmu')(\bl')^T, (\bmu^*)(\bl^*)^T-(\bmu')(\bl')^T\rangle_F-\|(\bmu^*)(\bl^*)^T-(\bmu')(\bl')^T\|_F^2\\
    &\quad\leq(\delta-1)\|(\bmu^*)(\bl^*)^T-(\bmu')(\bl')^T\|_F^2.
\end{align*}

Denote $T = \sum_{i=1}^n(\bmu \bl^T - (\bmu^*)(\bl^*)^T)_i^T (\bSigma^* - \bI_p) ((\bmu') (\bl')^T-(\bmu^*)(\bl^*)^T)_i$ as the mean of random variable $\sum_{i=1}^n(\bZ_i - (\bmu \bl^T)_i)^T ((\bmu') (\bl')^T-(\bmu^*)(\bl^*)^T)_i$. Similarly we have 
\begin{align*}
    &-\langle\bmu \bl^T-(\bmu^*)(\bl^*)^T, (\bmu^*)(\bl^*)^T-(\bmu')(\bl')^T\rangle_F\\
    &\quad=\langle\bmu \bl^T-(\bmu')(\bl')^T, (\bmu')(\bl')^T - (\bmu^*)(\bl^*)^T \rangle_F + \|(\bmu^*)(\bl^*)^T-(\bmu')(\bl')^T\|_F^2\\
    &\quad\leq(\delta+1)\|(\bmu^*)(\bl^*)^T-(\bmu')(\bl')^T\|_F^2.
\end{align*}
Thus,
\begin{align*}
    -T &= \sum_i (\bmu \bl^T - (\bmu^*)(\bl^*)^T)_i^T (\bSigma^* - \bI_p) ((\bmu^*)(\bl^*)^T - (\bmu') (\bl')^T)_i \\
    &\leq \|\bSigma^*\|_2 \langle (\bmu \bl^T-(\bmu^*)(\bl^*)^T), (\bmu^*)(\bl^*)^T-(\bmu')(\bl')^T\rangle_F - \langle\bmu \bl^T-(\bmu^*)(\bl^*)^T, (\bmu^*)(\bl^*)^T-(\bmu')(\bl')^T\rangle_F \\
    &\leq \left(\|\bSigma^*\|_2 (\delta - 1) + (\delta + 1)\right)\|(\bmu^*)(\bl^*)^T-(\bmu')(\bl')^T\|_F^2.
\end{align*}
Therefore, by Hoeffding's inequality,
\begin{align*}
    &\prob_{\bZ}(\|\bZ-(\bmu^*)(\bl^*)^T\|_F^2-\|\bZ-(\bmu')(\bl')^T\|_F^2<\rho)\\
    &\quad \leq \prob_{\bZ}{\Bigg(} \sum_{i=1}^n(\bZ_i - (\bmu \bl^T)_i)^T ((\bmu') (\bl')^T-(\bmu^*)(\bl^*)^T)_i - T \\
    &\quad\quad\quad\quad\quad < \left(\delta - 1 + m + \|\bSigma^*\|_2(\delta - 1) + \delta + 1 \right)\|(\bmu^*)(\bl^*)^T-(\bmu')(\bl')^T\|_F^2 {\Bigg)}\\
    &\quad\leq 2\exp\left( -\frac{c_2^2 \|(\bmu')(\bl')^T-(\bmu^*)(\bl^*)^T\|_F^4 }{2\|(\bSigma^*)^{\frac{1}{2}}(\bmu')(\bl')^T-(\bmu^*)(\bl^*)^T\|_F^2} \right)
    \leq 2\exp\left(-c_2' \|(\bmu')(\bl')^T-(\bmu^*)(\bl^*)^T\|_F^2\right)
\end{align*}
for constant $c_2' = c_2^2 / (2\|\bSigma^*\|_2) > 0$
where we use the fact that $c_2 = \delta - 1 + m + \|\bSigma^*\|_2(\delta - 1) + \delta + 1 = (\|\bSigma^*\|_2 + 2)\delta - (\|\bSigma^*\|_2 - m) < 0$ since $\delta < \frac{\|\bSigma^*\|_2}{2(\|\bSigma^*\|_2+2)} < \frac{\|\bSigma^*\|_2 - m}{ \|\bSigma^*\|_2 + 2}$. In addition, $\frac{c_2^2}{\|\bSigma^*\|_2} \geq \frac{(\|\bSigma^*\|_2-m)^2}{\|\bSigma^*\|_2} \geq \frac{\|\bSigma^*\|_2}{4} \geq \frac{m_{\bSigma}}{4} > 0$.
As a result, 
\begin{align*}
    \E_{(\bmu, \bl, \bI_p)}\left(\frac{p^*(\by)}{p_0(\by)}(1-\phi_n)\right) &= C\prob_{\bZ}\left( \|\bZ-(\bmu^*)(\bl^*)^T\|_F^2-\|\bZ-(\bmu')(\bl')^T\|_F^2<\rho \right) \\
    &\leq 2\exp\left( -c_3 \|\bmu\bl^T - (\bmu^*)(\bl^*)^T\|_F^2 \right) 
\end{align*}
for some constant 
\begin{align*}
    c_3 &= c_2' - \frac{1}{2}(1+\delta)^2 \left(\| \bSigma^*\|_2 - 1\right)\\
    &= \frac{((\|\bSigma^*\|_2+2)\delta + m - \|\bSigma^*\|_2)^2 - \|\bSigma^*\|_2(\|\bSigma^*\|_2 - 1)(1+\delta)^2}{2\|\bSigma^*\|_2}\\
    &= \frac{(5\|\bSigma^*\|_2+4)\delta^2 + (-4\|\bSigma^*\|_2^2 + 2\|\bSigma^*\|_2m + 4m - 2\|\bSigma^*\|_2)\delta + (m^2 - 2\|\bSigma^*\|_2m + \|\bSigma^*\|_2)}{2\|\bSigma^*\|_2}.
\end{align*}
Denote $c_4(\delta) = 2\|\bSigma^*\|_2c_3$ as a quadratic function of $\delta$. 
We have $c_4(0) = m^2 - 2\|\bSigma^*\|_2m + \|\bSigma^*\|_2 > 0$. 
When $\|\bSigma^*\|_2 \geq 1$, we have $m = \frac{\|\bSigma^*\|_2 - \sqrt{\|\bSigma^*\|_2^2 - \|\bSigma^*\|_2}}{2} \leq \frac{1}{2}$. When $\|\bSigma^*\|_2 < 1$, $m = \frac{\|\bSigma^*\|_2}{2} \leq \frac{1}{2}$. Thus $c_4(0) \geq \frac{1}{4}$ and therefore there exists $\delta > 0$ such that $c_3 \geq \frac{1}{8\|\bSigma^*\|_2} \geq \frac{1}{8M_{\bSigma}} > 0$.
\end{proof}

\begin{proof}[proof of Lemma \ref{lemma4}]
denote $\mathcal{F}_{nK}=\{(\bmu, \bl):\bmu\in\mathbb{R}^{p\times K}, 
(\bmu_k)_{S_\delta}\in [-a_n,a_n]^{\beta (s+n\log K/\log p)} \\\text{ for $k\in[K]$},
|\mathrm{supp}_{\delta}(\bmu)|\leq \beta (s+n\log K/\log p),
\bl=\mathcal{L}_K,K\leq K_{\max}\}$. Then 
$N(\epsilon_n,\mathcal{F}_n,d)\leq \sum_{K=1}^{K_{\max}} N(\epsilon_n,\mathcal{F}_{nK},d)$
since $\mathcal{F}_{nK_1}$ and $\mathcal{F}_{nK_2}$ are disjoint for $K_1\neq K_2$.
Consider for fixed $\bl$, 
$$\frac{\|\bmu\bl^T-\bmu'\bl^T\|_F^2}{n}\leq \frac{\|\bmu-\bmu'\|_F^2\|\bl\|_F^2}{n}=\|\bmu-\bmu'\|_F^2.$$
denote $\mathcal{G}_{nK1}\overset{\Delta}{=}\{\bmu\in\mathbb{R}^{p\times K}:(\bmu_k)_{S_\delta}\in[-a_n,a_n]^{\beta (s+n\log K/\log p)}, k\in [K]\}$ and $\mathcal{G}_{nK2}\overset{\Delta}{=}\mathcal{L}_K$.
We know that the cardinality of $\mathcal{G}_{nK2}$ is $K^n$.
Therefore,
$\log N(\epsilon_n,\mathcal{F}_{nK},d)\leq n\log K+\log N(\epsilon_1,\mathcal{G}_{nK1},\|\cdot\|_F)$
where $\epsilon_1^2=(s\log p+n\log K_{\max})/n$. 
Let $\mathcal{G}_{nK1\delta^c}\overset{\Delta}{=}\{\bmu\in\mathcal{G}_{nK1}:\bmu_{S_\delta}=0\}$ and $\mathcal{G}_{nK1\delta}\overset{\Delta}{=}\{\bmu\in\mathcal{G}_{nK1}:\bmu_{S_\delta^c}=0\}$. Suppose $\mathcal{N}_{\delta^c}$ and $\mathcal{N}_{\delta}$ are the minimal $\epsilon_1/2$-covering of $\mathcal{G}_{nK1\delta^c}$ and $\mathcal{G}_{nK1\delta}$ respectively. Then for any $\bmu\in\mathcal{G}_{nK1}$, there exists $\tilde{\bmu}\in\mathcal{N}_{\delta^c}$ and $\bar{\bmu}\in\mathcal{N}_{\delta}$ such that $\|\bmu-(\tilde{\bmu}+\bar{\bmu})\|_F\leq \|\bmu_{S_\delta^c}-\tilde{\bmu}_{S_\delta^c}\|_F+\|\bmu_{S_\delta}-\bar{\bmu}_{S_\delta}\|_F\leq \epsilon_1$. Thus, we have $\log N(\epsilon_1,\mathcal{G}_{nK1},\|\cdot\|_F)\leq \log N(\epsilon_1/2,\mathcal{G}_{nK1\delta^c},\|\cdot\|_F) + \log N(\epsilon_1/2,\mathcal{G}_{nK1\delta},\|\cdot\|_F)$.

Note that for $\mathcal{G}_{nK1\delta^c}$, we have $|\mathcal{G}_{nK1\delta^c}|<\delta^{pK}$. Since $\delta\lesssim \log p/(pK_{\max}\sqrt{n/\log p})$, we know that $|\mathcal{G}_{nK1\delta^c}|\lesssim |B_{\epsilon_1}^{p\times K}(0)|$ where $B_{\epsilon_1}^{p\times K}(0)$ is an $\epsilon_1$-ball in $\mathbb{R}^{p\times K}$. Thus $N(\epsilon_1/2,\mathcal{G}_{nK1\delta^c},\|\cdot\|_F)$ is bounded above by some constant.
We know that for a subset of Euclidean space,
$$\log N(\epsilon_1,\mathcal{G}_{nK1\delta},\|\cdot\|_F)\leq \beta\left(s+\frac{n\log K}{\log p}\right)K\log\frac{3a_n}{\epsilon_1}+\log {p\choose \beta \left(s+\frac{n\log K}{\log p}\right)}.$$
Note that since $(s+n\log K/\log p)/p\to 0$ as $n\to\infty$, by Stirling's formula we have $\log {p\choose \beta\left(s+\frac{n\log K}{\log p}\right)}\lesssim s\log p+n\log K.$
By letting $a_n=(s\log p+n\log K_{\max})n^{\gamma}$ we will have 
$\log N(\epsilon_1,\mathcal{G}_{nK1},\|\cdot\|_F)\lesssim s\log p + n\log K$ since $K\log n\lesssim\log p$.
Thus, 
$\log N(\epsilon_n,\mathcal{F}_{nK},d)\lesssim s\log p+n\log K.$
Therefore, for some constant $c'$
\begin{align*}
    N(\epsilon_n,\mathcal{F}_n,d)&\leq\sum_{K=1}^{K_{\max}} \exp(c(s\log p+n\log K))
    \leq\exp(c'(s\log p+n\log K_{\max}))
    =\exp(c'n\epsilon_n^2).
\end{align*}

\end{proof}

\begin{proof}[proof of Lemma \ref{lemma3}]
We have
\begin{align*}
    \Pi(\mathcal{F}_n^c)&\leq\sum_{k=1}^{K_{\max}}\Pi\left(\bigcup_{i=1}^k\left\{\|\bmu_i\|_\infty>a_n\right\}\big\vert K=k\right)\Pi(K=k)+\Pi\left(|\mathrm{supp}_\delta(\bmu)|>\beta \left(s+\frac{n\log K_{\max}}{\log p}\right)\right)\\
    &\leq \sum_{k=1}^{K_{\max}}kp\Pi(|\mu_{11}|>a_n)\Pi(K=k)+\Pi\left(|\mathrm{supp}_\delta(\bmu)|>\beta \left(s+\frac{n\log K_{\max}}{\log p}\right)\right)\\
    &\leq p\Pi(|\mu_{11}|>a_n)\sum_{k = 1}^{K_{\max}}k\Pi(K = k)+\Pi\left(|\mathrm{supp}_\delta(\bmu)|>\beta \left(s+\frac{n\log K_{\max}}{\log p}\right)\right)\\
    &\leq \frac{\lambda}{1-e^{-\lambda}}p\Pi(|\mu_{11}|>a_n)+\Pi\left(|\mathrm{supp}_\delta(\bmu)|>\beta \left(s+\frac{n\log K_{\max}}{\log p}\right)\right)
\end{align*}
By lemma \ref{lemma_soft_supp}, the last term on the right hand side of the inequality is bounded above
$$\Pi\left(|\mathrm{supp}_\delta(\bmu)|>\beta \left(s+\frac{n\log K_{\max}}{\log p}\right)\right)\leq \exp\left(-c(s\log p+n\log K_{\max})\right)$$ for some constant $c>0$.
By spike-and-slab lasso prior, we know that $|\mu_{11}|\sim (1-\theta)\text{Exp}(\lambda_0)+\theta\text{Exp}(\lambda_1).$
Let $X=|\mu_{11}|$, 
we have $X\mid \theta\sim (1-\theta)\text{Exp}(\lambda_0)+\theta\text{Exp}(\lambda_1)$, then
\begin{align*}
    \sup_{m\geq 1}\frac{(\E[|X|^m])^{1/m}}{m} &= \sup_{m\geq 1}\frac{(\E[\E[|X|^m\mid \theta]])^{1/m}}{m}
    =\sup_{m\geq 1}\frac{\E\left[(1-\theta)\frac{m!}{\lambda_0^m}+\theta\frac{m!}{\lambda_1^m}\right]^{1/m}}{m}\\
    &= \sup_{m\geq 1}\frac{1}{m}\left\{\left(1-\frac{1}{1+\beta_\theta}\right)\frac{m!}{\lambda_0^m}+\frac{1}{1+\beta_\theta}\frac{m!}{\lambda_1^m}\right\}^{1/m}\\
    &\leq \sup_{m\geq 1}\frac{1}{m}\frac{(m!)^{1/m}}{\lambda_1}
    \leq \sup_{m\geq 1}\frac{1}{m}\frac{(em)^{1/m}m}{e\lambda_1}
    \leq \frac{2}{\lambda_1}\leq2n^{\gamma}\leq\infty
\end{align*}
for any $n>0$.
Note that the first inequality is due to the power mean inequality for $\lambda_0\geq \lambda_1$.
Thus by Bernstein inequality, we have
\begin{align*}
    &\Pi(X\geq a_n)=\Pi(X-\E X\geq a_n-\E X)
    \leq\Pi(X-\E X\geq \frac{a_n}{2})\\
    &\quad\leq 2\exp\left(-c\min\left(\frac{a_n^2}{4\|X\|_{\psi_1}},\frac{a_n}{2\|X\|_{\psi_1}}\right)\right)
    \leq\exp(-c(s\log p+n\log K_{\max}))
\end{align*}
Note that here $\E X=\frac{1-\theta}{\lambda_0}+\frac{\theta}{\lambda_1}\leq\frac{a_n}{2}$ and $a_n=(s\log p+n\log K_{\max})n^{\gamma}$.

\end{proof}

\subsection{Proof of Theorem \ref{thm1}}
\begin{proof}[proof of Theorem \ref{thm1}]
denote $\Bar{\epsilon}_n^2=(s\log p+n\log K^*)/n$ and $\Theta = \bigcup_{K = 1}^{K_{\max}}\mathbb{R}^{p\times K}\times\mathcal{L}_K$. Let $d$ be a metric on $\Theta$ with $d((\bmu, \bl), (\bmu', \bl')) = \|\bmu \bl^T-\bmu'\bl'^T\|_F$.
Let $$U_n=\left\{(\bmu, \bl)\in\Theta: d((\bmu, \bl), (\bmu^*, \bl^*))< M\sqrt{n}\epsilon_n\right\}.$$ 
 
By Bayes rule,
$\Pi(U_n^c\mid \by)=\frac{\int_{U_n^c}\frac{p_n(\by)}{p_0(\by)}d\Pi}{\int\frac{p_n(\by)}{p_0(\by)}d\Pi}:=\frac{N_n}{D_n}$
where $$p_0(\by)=(2\pi)^{-\frac{np}{2}}\exp\left(-\frac{1}{2}\|\by-(\bmu^*)(\bl^*)^T\|_F^2\right), p_n(\by)=(2\pi)^{-\frac{np}{2}}\exp\left(-\frac{1}{2}\|\by-\bmu \bl^T\|_F^2\right).$$
We also denote
$
p^*(\by)=(2\pi)^{-\frac{np}{2}}\det(\bSigma^*)^{-\frac{n}{2}}\exp\left(-\frac{1}{2}\|(\bSigma^*)^{-\frac{1}{2}}\left(\by-(\bmu^*)(\bl^*)^T\right)\|_F^2\right).
$
 
By Lemma \ref{lemma2} we know that 
$
\Pi\left(\|\bmu \bl^T-(\bmu^*)(\bl^*)^T\|_F^2\leq s\log p \right)
\geq\exp\left(-c_1n\bar{\epsilon}_n^2\right)
$
for some constant $c_1 > 0$.
Denote $A_n=\{D_n>\exp(-c_2n\Bar{\epsilon}_n^2)\}$ for some $c_2>c_1>0$.
Thus,
$$A_n\supset\{D_n>\Pi(\|\bmu \bl^T-(\bmu^*)(\bl^*)^T\|_F^2\leq s\log p)\exp(-(c_2-c_1)n\Bar{\epsilon}_n^2)\}.$$ 

Let $\mathbbm{1}(A)$ denote the indicator random variable whose expected value is the probability of event $A$.
 
Then we can write 
\begin{align*}
    &\E_*\Pi(U_n^c\mid Y)=\E_*\{(\phi_n + 1 - \phi_n)\Pi(U_n^c\mid \by)\}\\
    &\quad
    =\E_*\{\phi_n\Pi(U_n^c\mid \by)\}+\E_*\{(1 - \phi_n)\mathbbm{1}(A_n)\Pi(U_n^c\mid \by)\}+ \E_*\{(1 - \phi_n)\mathbbm{1}(A_n^c)\Pi(U_n^c\mid \by)\}\\
    &\quad
    \leq\E_*\phi_n+\E_*\left\{(1-\phi_n)\exp(c_2n\Bar{\epsilon}_n^2)\int_{U_n^c}\frac{p_n(\by)}{p_0(\by)}d\Pi\right\}
    +\prob_*(A_n^c).
    \end{align*}
 
We treat the three terms on the right-hand side of the last equality separately. Denote
$$\mathcal{F}_n = \bigcup_{K = 1}^{K_{\max}}\left\{(\bmu, \bl):\bmu\in\mathbb{R}^{p\times K},|\mathrm{supp}_{\delta_{\bmu}}(\bmu)|\leq \beta \left(s+\frac{n\log K}{\log p}\right), \max_{k\in [K]}\|(\bmu_k)_{S_\delta}\|_\infty \leq a_n, \bl\in\mathcal{L}_K\right\}$$ for $\delta_{\bmu}$ and $\beta$ defined in Lemma \ref{lemma_soft_supp}, $a_n=(s\log p+n\log K_{\max})n^{\gamma}$.
Let $U_{n,j}=\{(\bmu,\bl)\in\mathcal{F}_n:d((\bmu, \bl), (\bmu^*, \bl^*)) \in\left[j^2n\epsilon_n^2,(j+1)^2n\epsilon_n^2\right)\}$.
Let $N_{n,j}$ be the maximal $\epsilon_n$-nets $U_{n,j,1},\dots,U_{n,j,N_j}$ that covers $U_{n,j}$ with respect to metric $d$.
By Lemma \ref{lemma1} we have that for each $U_{n,j,h}\in N_{n,j}$, there exists a test $\phi_{n,j,h}$ such that
$\E_*\phi_{n,j,h}\leq \exp\left(-c_3nj^2\epsilon_n^2\right).$ Denote $\phi_n=\max_{j=M}^\infty\max_{h=1}^{|N_{n,j}|}\phi_{n,j,h}$. Then $\E_*\phi_n\leq \sum_{j=M}^\infty\sum_{h=1}^{|N_{n,j}|}\E_*\phi_{n,j,h}$.
By Lemma \ref{lemma1}, we have $\E_*\phi_n\leq N(\epsilon_n,\mathcal{F}_n,d)\frac{\exp(-c_3nM^2\epsilon_n^2)}{1-\exp(-c_3n\epsilon_n^2)}.$ 
By lemma \ref{lemma4} we have $N(\epsilon_n,\mathcal{F}_n,d) \leq\exp(c_4n\epsilon_n^2)$, so the first term goes to 0 under $\E_*$ as $n$ tends to infinity and sufficiently large $M>0$.

For the third term, by definition we have
$
    \prob_*(A_n^c) \leq 1-\Pi\{D_n>\Pi(\|\bmu \bl^T-(\bmu^*)(\bl^*)^T\|_F^2\leq s\log p)\exp(-(c_2-c_1)n\Bar{\epsilon}_n^2)\}.
$
Consider the event in the probability on the right hand side. By dividing both sides of the inequality by $\Pi(\|\bmu \bl^T-(\bmu^*)(\bl^*)^T\|_F^2\leq s\log p)$, we can rewrite the inequality in terms of $\Pi'$, which is the restricted and renormalized probability measure of prior $\Pi$ conditioning on the event 
$E_n=\{\|\bmu \bl^T-(\bmu^*)(\bl^*)^T\|_F^2\leq s\log p\}$.
Next by Jensen's inequality,  $$\sum_iZ_{ni}\overset{\Delta}{=}\int\sum_i\log\frac{p_n(\by_i)}{p_0(\by_i)}d\Pi'\leq\log\int\prod_i\frac{p_n(\by_i)}{p_0(\by_i)}d\Pi'.$$
Then we have that 
$
    \sum_{i}\E_* Z_{ni}
    = -\frac{1}{2}\int\|\bmu \bl^T-(\bmu^*)(\bl^*)^T\|_F^2 d\Pi' \geq-\frac{1}{2}s\log p.
$
So for the event $\{D_n'>\exp(-(c_2-c_1)n\Bar{\epsilon}_n^2)\}$ where $D_n'=\int\prod_i\frac{p_n(\by_i)}{p_0(\by_i)}d\Pi'$, we have
\begin{align*}
    \Pi(\log D'_n\geq-(c_2-c_1)n\Bar{\epsilon}_n^2)&\geq\Pi\left(\sum_iZ_{ni}\geq-(c_2-c_1)n\Bar{\epsilon}_n^2\right)
    \geq\Pi\left(\sum_i(Z_{ni}-\E_* Z_{ni})\geq-c'_2n\Bar{\epsilon}_n^2\right).
\end{align*}
Therefore, by Hoeffding inequality for sub-Gaussian random variable:
\begin{align*}
    &\Pi\left(\left|\sum_i Z_{ni} - \E_* Z_{ni} \right| \geq \frac{c_2'n\Bar{\epsilon}_n^2}{2}\right)
    \leq 2\exp\left( -c_5\frac{n^2\Bar{\epsilon}_n^4}{\|(\bSigma^*)^{\frac{1}{2}}((\bmu^*)(\bl^*)^T - \bmu\bl^T)\|_F^2} \right)\\
    &\quad \leq 2\exp\left( -c_5\frac{n^2\Bar{\epsilon}_n^4}{\|(\bSigma^*)^{\frac{1}{2}}\|_2^2\|((\bmu^*)(\bl^*)^T - \bmu\bl^T)\|_F^2} \right)
    \leq \exp\left(-c_5'n\Bar{\epsilon}_n^2\right)
\end{align*}
for some constant $c_5, c_5' > 0$ since $\|\bSigma^*\|_2$ is upper bounded from infinity.
Thus, as $n\to\infty$,
$\E_*\mathbbm{1}(A_n^c)\leq 1-\Pi(\log D'_n\geq-(c_2-c_1)n\Bar{\epsilon}_n^2)\leq\exp(-c_5'n\Bar{\epsilon}_n^2)\to 0.$

For the second term, we have, by Fubini's theorem,
\begin{align*}
    &\E_* (1-\phi_n)\exp(c_2n\Bar{\epsilon}_n^2)\int_{U_n^c}\frac{p_n(\by)}{p_0(\by)}d\Pi \leq \exp(c_2n\Bar{\epsilon}_n^2)\left(\int_{U_n^c\cap\mathcal{F}_n}\E_n\left( (1-\phi_n) \frac{p^*(\by)}{p_0(\by)}\right)d\Pi+\Pi(\mathcal{F}_n^c)\right).
\end{align*}
By Lemma \ref{lemma3}, we know that $\log\Pi\left(\mathcal{F}_{n}^c\right) \lesssim -n\Bar{\epsilon}_n^2$.
By Lemma \ref{lemma1}, we have 
$\E_n\left((1-\phi_n)\frac{p^*(\by)}{p_0(\by)}\right)\leq \exp(-c_5nM^2\epsilon_n^2)$ so the above term goes to 0 as $n$ tends to infinity.

\end{proof}

\subsection{Proof of Theorem \ref{thm:cluster_accuracy}}

\begin{proof}[proof of Theorem \ref{thm:cluster_accuracy}]
We find the singular value decomposition (SVD) of $\bmu\bl^T=\bU\Sigma \bV^T$ for some diagonal matrix $\Sigma\in\mathbb{R}^{p\times n}$ and $\bU\in\mathbb{O}^p$ and $\bV\in\mathbb{O}^n$ where $\mathbb{O}^m$ denotes the set of $m$ by $m$ orthogonal matrices. Consider the matrix $\bl$, let $\Sigma_L=\bl^T\bl$ be the diagonal matrix whose $k$th diagonal entry is the size of cluster $k$. Then denote $\bl_N=\bl\Sigma_L^{-1/2}$ as the ``normalization" of $\bl$ since it is orthogonal. We should note that the reason why we call it ``normalization" is that for each row of $\bl$, we divide it by the square root of the size of the cluster it belongs to. On the other hand, for matrix $\bmu$, we suppose the corresponding QR decomposition is $\bmu=\bQ\bR$ for some $\bQ\in\mathbb{O}^{p}$ and upper triangular matrix $\bR\in\mathbb{R}^{p\times K}$. Then suppose the SVD of $\bR\Sigma_L^{1/2}$ is $\bR\Sigma_L^{1/2}=\bU_1\Sigma_1\bV_1^T$ for some $\bU_1\in\mathbb{O}^p$ and $\bV_1\in\mathbb{O}^K$. Therefore we obtain 
$\bmu\bl^T=\bQ\bU_1\Sigma_1\bV_1^T\bl_N^T$ and we know that $v_i$, which denotes the $i$th column of $\bl_N\bV_1$, satisfies $\bmu\bl^Tv_i=\sigma_i u_i$ for $i=1,\cdots,K$ where $u_i$ is the $i$th column of $\bQ\bU_1$ and $\sigma_i$ is the $i$th singular value of $\bmu\bl^T$.

Then we can use a variant of Davis-Kahan theorem (\citealp{yu2014}). Suppose $\bmu\bl^T=\bQ\bU_1\Sigma_1\bV_1^T\bl_N^T$ and $(\bmu^*)(\bl^*)^T=\bQ^*\bU_1^*\Sigma_1^*(\bV_1^*)^T(\bl_N^*)^T$. 
Denote
$D_F(\bO_1,\bO_2)=\inf_{\bV\in\mathbb{\bO}^r}\|\bO_1-\bO_2\bV\|_F$ for $\bO_1,\bO_2\in\mathbb{O}^r$
and let $\|\sin\Theta(\bO_1,\bO_2)\|_F$ be the (Frobenius) sine-theta distance between $\mathrm{span}(\bO_1)$ and $\mathrm{span}(\bO_2)$. Then the relationship between the metric $D_F$ and sine-theta distance holds:
$\|\sin\Theta(\bO_1,\bO_2)\|_F\leq D_F(\bO_1,\bO_2)\leq\sqrt{2}\|\sin\Theta(\bO_1,\bO_2)\|_F.$
Note that for the right singular subspace of $\bmu\bl^T$, we have
\begin{align*}
    \sqrt{2}\|\sin\Theta(\bl_N\bV_1,\bl_N^*\bV_1^*)\|_F&\geq \inf_{\bV\in\mathbb{O}^K}\|\bl_N\bV_1-\bl_N^*\bV_1^*\bV\|_F\\
    &=\inf_{\bV\in\mathbb{O}^K}\|\left(\bl_N-\bl_N^*\bV_1^*\bV(\bV_1)^{-1}\right)\bV_1\|_F\\
    &=\inf_{\bV\in\mathbb{O}^K}\|\bl_N-\bl_N^*\bV_1^*\bV(\bV_1^*)^{-1}\|_F
    \geq \|\sin\Theta(\bl_N,\bl_N^*)\|_F.
\end{align*}
Then by theorem 3 in \cite{yu2014}, we have
\begin{align*}
    \|\sin\Theta(\bl_N,\bl_N^*)\|_F&\leq \frac{2\sqrt{2}(2\sigma_{\max}+\|\bmu\bl^T-(\bmu^*)(\bl^*)^T\|_{2})}{\sigma_{\min}^2}
    \times \|\bmu\bl^T-(\bmu^*)(\bl^*)^T\|_F
\end{align*}
where $\sigma_{\max}$ and $\sigma_{\min}$ represent the max and min singular value of $(\bmu^*)(\bl^*)^T$ respectively.

We write $\bl_N=[(l_N)_1,\dots,(l_N)_n]^T$ and $\bl_N^*=[(l_N)_1^*,\dots,(l_N)_n^*]^T$. Note that $\bl_N^*$ and $\bl_N$ have at most $K$ distinct rows. Let $\zeta$ be the minimum distance among these $K$ distinct rows of $\bl_N^*$ with respect to $\ell$-2 norm: $\zeta=\min_{(l_N)^*_i\neq (l_N)^*_j}\|(l_N)^*_i-(l_N)^*_j\|_2$. Let $\bO=\arg\inf_{\bV\in\mathbb{O}^K}\|\bl_N-\bl_N^*\bV\|_F$. Define the set 
$\mathcal{I}=\left\{i:\|(l_N)_i-\bO^T(l_N)^*_i\|_2\geq\frac{\zeta}{2}\right\}.$
Assume that the event $\mathcal{E}_n = \{\|\sin\Theta(\bl_N,\bl_N^*)\|_F\leq\eta_n\}$ occurs \textit{a posteriori}, where $\eta_n=\frac{8\sqrt{2}\sigma_{\max}}{\sigma_{\min}^2}(M(s\log p+n\log K_{\max}))^{1/2}$
and $M$ is the constant in Theorem \ref{thm1}. By Theorem \ref{thm1}, we know that $\E_*\{\Pi(\mathcal{E}_n)\}\to 1$ as $n\to\infty$. 
This implies that $|\mathcal{I}|\leq {4\eta_n^2}/{\zeta^2}$ since otherwise we have
    $\|\bl_N-\bl_N^*\bO\|_F^2\geq ({\zeta^2}/{4})({4\eta_n^2}/{\zeta^2}) = \eta_n^2$,
which contradicts with the definition of $\mathcal{E}_n$. Thus, for any $i,j\in\mathcal{I}^c$ with $(l_N)_i=(l_N)_j$,
    $\|(l_N)^*_i-(l_N)^*_j\|_2
    \leq \|(l_N)_i-\bO^T(l_N)^*_i\|_2+\|(\bm{l}_N)_j-\bO^T(l_N)^*_j\|_2<\zeta,$
which implies $(l_N)^*_i=(l_N)^*_j$ since $\zeta$ is the minimum distance between pair of distinct rows of $\bl_N^*$.
On the other hand, note that $\zeta^2=1/n_{\max}^*+1/n_2$ where $n_2$ is the second largest cluster size. Consequently, since 
\begin{align*}
\sigma_{\min}((\bmu^*)(\bl^*)^T) &= \sqrt{\lambda_{K^*}((\bmu^*)(\bl^*)^T(\bl^*)(\bmu^*)^T)}
\geq\sqrt{n_{\min}^*\lambda_{K^*}((\bmu^*)(\bmu^*)^T)}
\geq n_{\min}^*\sigma_{\min}(\bmu^*),\\    
\sigma_{\max}((\bmu^*)(\bl^*)^T) &= \sqrt{\lambda_{\max}((\bmu^*)(\bl^*)^T(\bl^*)(\bmu^*)^T)}
\leq\sqrt{n_{\max}^*\lambda_{\max}((\bmu^*)(\bmu^*)^T)}
\leq n_{\max}^*\sigma_{\max}(\bmu^*),
\end{align*}
we have 
\[
|\mathcal{I}|\leq \frac{4\eta_n^2}{\zeta^2}\asymp\frac{(n_{\max}^*)^3\sigma_{\max}(\bmu^*)^2}{(n_{\min}^*)^4\sigma_{\min}(\bmu^*)^4}(s\log p+n\log K_{\max}).
\]
Note that $n_k^*>|\mathcal{I}|$ for all $k$,
namely, $\{(l_N)_i^*:i\notin\mathcal{I}\}$ consists of all $K^*$ distinct rows of $\bl_N^*$. 
Therefore, each of the unique $B_{\|\cdot\|_2}(\bO^T(l_N)^*_i,\zeta/2)$ as $i$ ranges over $[n]$, which is the $\ell$-2 ball centered at $\bO^T(l_N)^*_i$ with radius $\zeta/2$, contains at least one element of $\{(l_N)_i:i\in\mathcal{I}^c\}$.
Recall that $\zeta$ is the minimum distance between pair of distinct rows of $\bl_N^*$, so these open balls are disjoint for distinct rows of $\bl_N^*$. 
It follows from the pigeonhole principle that each ball contains exactly one element of $\{(l_N)_i:i\in\mathcal{I}^c\}$. 
As a result, if $(l_N)^*_i=(l_N)^*_j$ for $i,j\in\mathcal{I}^c$, then $(l_N)_i,(l_N)_j\in B_{\|\cdot\|_2}(\bO^T(l_N)^*_i,\zeta/2)$, implying that $(l_N)_i = (l_N)_j$ by the fact that every such ball contains exactly one row of $\bl$.

Therefore we prove that for any $i\in\mathcal{I}^c$, $(l_N)_i=(l_N)_j$ if and only if $(l_N)^*_i=(l_N)^*_j$. So the number of mis-clustered points are at most $|\mathcal{I}|$, which gives us the result because $\E_*\{\Pi(\mathcal{E}_n)\}\to 1$ as $n\to\infty$ by Theorem \ref{thm1}.
\end{proof}

\section{Posterior Inference via Gibbs Sampling}\label{sec:sampling}

In this section, we introduce a Gibbs sampler for posterior inference of the proposed Bayesian sparse Gaussian mixture model. 
We design a sampler based on the algorithm proposed in  \cite{miller2018} to avoid cross-dimensional sampling as in
RJ-MCMC. Let $\mathcal{C}$ denote the partition of $[n]$ according to the cluster memberships 
$\bz$. Formally,
$\mathcal{C}=\{E_k:|E_k|>0\}$ where $E_k=\{i:z_i=k \text{ for }i\in[n]\}$. 
Let $\mathcal{C}_{-i}$ be the partition of $[n]\setminus\{i\}$ according to the cluster memberships $\{z_j\}_{j\neq i}$. We also denote $n_c$ as the number of data points in $c \in \mathcal{C}$ and $n_c^-$ 
as the number of data points in $c \in \mathcal{C}_{-i}$. 
We can derive an urn representation for the mixture model from the exchangeable partition distribution: 
$\pi(\mathcal{C})=V_n(|\mathcal{C}|)\prod_{c\in\mathcal{C}}\alpha^{(|c|)}$, where 
$V_n(|\mathcal{C}|)=\sum_{k=1}^\infty p_K(k)k_{(|\mathcal{C}|)}/(\alpha k)^{(n)}
$
and $p_K$ is the prior of $K$. Here $x^{(m)}\overset{\Delta}{=}x(x+1)\cdots(x+m-1)$, $x_{(m)}\overset{\Delta}{=}x(x-1)\cdots(x-m+1)$.


To address the non-conjugacy of the Laplace distribution,
we re-write the SSL($\lambda_0$, $\lambda_1$, $\theta$) prior through the normal-scale-mixture representation of Laplace as follows: for $j\in[p]$,  
\begin{align*}
    (x_j\mid\phi_j,\xi_j=a)&\sim N\left(0,\frac{\phi_j}{\lambda_a^2}\right), a=0, 1, \quad
    \phi_j\sim\text{Exp}\left(\frac{1}{2}\right),\quad
    \xi_j\sim \text{Bernoulli}(\theta).
\end{align*}
We obtain the following closed-form full conditional posterior distributions of $\bmu_c$, $\bm{\phi}_c$ and $\xi$,
\begin{align*}
    ((\bmu_c)_j\mid-)&\sim N\left(\sum_{l\in c}(\by_l)_j\left(n_c+\frac{\lambda_{\xi_j}^2}{(\phi_c)_j}\right)^{-1},\left(n_c+\frac{\lambda_{\xi_j}^2}{(\phi_c)_j}\right)^{-1}\right),\\
    ((\phi_c)_j\mid-)&\sim\text{GiG}(0.5,(\bmu_c)_j^2\lambda_{\xi_j}^2,1),\\
    (\xi_j\mid-)&\sim\text{Bernoulli}(\theta'),\quad\text{where}\quad\theta'\propto \prod_{c\in \mathcal{C}}\frac{\lambda_1}{\sqrt{(\phi_c)_j}}\exp\left(-\frac{1}{2}\lambda_1^2\frac{(\bmu_c)_j^2}{(\phi_c)_j}\right)\theta,\\
    (\theta\mid-)&\sim\text{Beta}\left(1+\sum_{j=1}^p\xi_j,\beta_\theta+p-\sum_{j=1}^p\xi_j\right).
\end{align*}
Here, GiG$(\zeta,\chi,\tau)$ denotes the generalized inverse Gaussian distribution whose probability density function is 
$f(x)=x^{\zeta-1}\exp\left(-\left(\chi/x+\tau x\right)/2\right).$ 
We also remark that there exists a potential label switching phenomenon when sampling centers $\bmu_c$ and auxiliary variables $\bm{\phi}_c$ for all clusters. This can be prevented by the following alignment process. 
\begin{enumerate}[(i)]
    \item Collect 
    $B$ post-burn-in MCMC samples $\bmu^{(b)}=\left[\bmu_1^{(b)},\dots,\bmu^{(b)}_{K^{(b)}}\right]$ and 
    $\bl^{(b)}=\left[\bm{l}_1^{(b)},\dots,\bm{l}^{(b)}_{K^{(b)}}\right]^T$ for $b=1,\dots,B$, where  $K^{(b)}\overset{\Delta}{=}|\mathcal{C}^{(b)}|$ is the number of clusters in $b$-th iteration.
    \item Find the index $b^*$ that corresponds to the maximizer of the log-likelihood function:
    $b^*=\argmin_{b\in[B]} \left\|\by-(\bmu^{(b)})(\bl^{(b)})^T\right\|_\mathrm{F}^2.$
    \item 
    For $b=1,\dots,B$, find $\mathcal{P}^{(b)}=\argmin_{\mathcal{P}\in\mathcal{S}^{K^{(b)}\times K^{(b)}}}\left\|\bmu^{(b^*)}-\bmu^{(b)}\mathcal{P}\right\|_\mathrm{F}^2$, where $\mathcal{S}^{K^{(b)}\times K^{(b)}}$ is the set of all $K^{(b)}\times K^{(b)}$ permutation matrices.
  \item   For $b=1,\dots,B$, replace $\bmu^{(b)}$ by $\bmu^{(b)}\mathcal{P}^{(b)}$ and $\bl^{(b)}$ by $\bl^{(b)}\mathcal{P}^{(b)}$.
\end{enumerate}

We provide the detailed Gibbs sampler  in Algorithm \ref{alg:Gibbs_sampler} below. The R code can be found at \url{https://github.com/YanxunXu/HighDimClustering}.

\begin{algorithm}[H]
    \caption{The Gibbs sampler}
    \label{alg:Gibbs_sampler}
    \begin{algorithmic}[1]
        \Require{Initialization of $\mathcal{C}$,$\{\bmu_c:c\in\mathcal{C}\}$, $\bm{\xi}$, $\{\bm{\phi}_c:c\in\mathcal{C}\}$}
        \For{$b \gets 1$ to $B$}
            \For{$i \gets 1$ to $n$}
                \State{$t \gets |\mathcal{C}_{-i}|$}
                \If{$z_i\neq z_l$ for all $l\neq i$}
                    \State{$\bmu_{t+1}\gets\bmu_{z_j}$}
                    \ElsIf{$z_i=z_l$ for some $l\neq i$}
                        \State{Sample $\bm{\phi}_{t+1}\sim p_{\bm{\phi}}(\bm{\phi}_{t+1})$}
                        \State{Sample $\bmu_{t+1} \sim  p_{\bmu\mid\bm{\xi},\bm{\phi}}(\bmu_{t+1}\mid \bm{\xi},\bm{\phi}_{t+1})$}
                    \EndIf
                    \For{$k\gets 1$ to $t$}
                        \State{$m_k\gets (n_k^-+\alpha)p(\by_i\mid\bmu_c)$ where $n^-_k$ is the size of cluster $k$ in $\mathcal{C}_{-j}$}
                    \EndFor
                    \State{$V_n(t)\gets\frac{t!}{n!}\frac{\Gamma(\alpha t)}{n^{\alpha t-1}}p_K(t)$}
                    \State{$V_n(t+1)\gets\frac{(t+1)!}{n!}\frac{\Gamma(\alpha (t+1))}{n^{\alpha (t+1)-1}}p_K(t+1)$}
                    \State{$m_{t+1}\gets \alpha\frac{V_{n}(t+1)}{V_n(t)} p(\by_i\mid\bmu_{t+1})$}
                    \State{Sample $z_i\sim \text{Categorical}\left(\frac{m_1}{\sum_{k=1}^{t+1}m_k},\cdots,\frac{m_{t+1}}{\sum_{k=1}^{t+1}m_k}\right)$}
            \EndFor
            \For{$c\gets 1$ to $|\mathcal{C}|$}
                \For{$j\gets 1$ to $p$}
                    \State{Sample $(\bmu_c)_j\sim N\left(\sum_{l\in c}(\by_l)_j\left(n_c+\frac{\lambda_{\xi_j}^2}{(\bm{\phi}_c)_j}\right)^{-1},\left(n_c+\frac{\lambda_{\xi_j}^2}{(\bm{\phi}_c)_j}\right)^{-1}\right)$}
                \EndFor
            \EndFor

            \For{$c\gets 1$ to $|\mathcal{C}|$}
                \For{$j\gets 1$ to $p$}
                    \State{Sample $(\bm{\phi_c})_j\sim\text{GiG}(0.5,(\bmu_c)_j^2\lambda_{\xi_j}^2,1)$}
                \EndFor
            \EndFor
            \For{$j\gets 1$ to $p$}
                \State{$\theta'\gets \frac{\prod_{c\in \mathcal{C}}\lambda_1\exp\left(-\frac{1}{2}\lambda_1^2\frac{(\bmu_c)_j^2}{(\bm{\phi}_c)_j}\right)\theta}{\prod_{c\in \mathcal{C}}\lambda_1\exp\left(-\frac{1}{2}\lambda_1^2\frac{(\bmu_c)_j^2}{(\bm{\phi}_c)_j}\right)\theta+\prod_{c\in \mathcal{C}}\lambda_0\exp\left(-\frac{1}{2}\lambda_0^2\frac{(\bmu_c)_j^2}{(\bm{\phi}_c)_j}\right)(1-\theta)}$}
                \State{Sample $\xi_j\sim\text{Bernoulli}(\theta')$}
            \EndFor
            \State Sample $\theta\sim\text{Beta}\left(1+\sum_{j=1}^p\xi_j,\beta_\theta+p-\sum_{j=1}^p\xi_j\right)$
        \EndFor
    \end{algorithmic}
\end{algorithm}

\section{Additional Numerical Studies}\label{sec:add_result}
We conduct three additional simulation studies to evaluate the performance of the proposed method under various setups: (a) when the covariance matrix of the underlying data generating process is diagonal, with different diagonal entries; (b) when the true sampling distribution exhibits skewness; and (c) when the true coordinates do not exhibit sparsity. We use the same hyperparameters as in the simulation studies, and run the proposed method with 1000 burn-in samples and 4000 post-burn-in samples. Each additional simulation setup is replicated 100 times.

In (a), we modify Scenario I in the simulation study with $s=12$ and $K=3$ by setting $\bSigma = \text{diag}(\sigma_1^2, \dots, \sigma_p^2)$, where $\sigma_j^2$ are i.i.d. sampled from $\text{Gamma}(100, 100)$ for $j\in [p]$.
In all 100 replicated simulations, the proposed method successfully identifies the three clusters, yielding an average ARI of 0.97. In contrast, the four competitors tend to merge the two overlapping clusters, resulting in average ARIs of 0.54. These results highlight the robustness of the proposed Bayesian method, consistently outperforming alternative methods in terms of clustering accuracy.

In (b), we consider the sampling distribution to be a skewed t-distribution, with the same location and scale parameters as Scenario III in the manuscript. We set the skewness parameter to 10 and the degrees of freedom to 2. The average ARI across these 100 replicates is 0.84 for our proposed method, compared to 0.34 for PCA-KM. MClust and SKM fail to cluster in almost all simulation replicates, returning an average ARI of less than 0.01. 

In (c), we set $n=200$, $p=400$, $K=2$, $(\bmu_1^*)_S = r\mathbf{1}_s$, $(\bmu_2^*)_S = -r\mathbf{1}_s$, and assume $\bSigma_1^* = \bSigma_2^* = \bI_p$. We fix the ``signal-to-noise ratio" by setting $\|\bmu_1^* - \bmu_2^*\|_2^2 = 36$ and the sparsity support is defined as $S = [s]$, with $s$ varying over $\{4, 10, 100, 400\}$.  
Our proposed Bayesian model successfully recovers the two clusters when $s=4$ and $10$. However, when $s = 100$ and $400$, our model returns only one cluster, indicating poorer performance when the underlying truth is denser. Intuitively, this discrepancy arises because the spike-and-slab prior inherently introduces bias when the true model is not sparse. This bias affects the estimation of cluster centers, thereby influencing the update of cluster membership during subsequent MCMC iterations.

%

\begin{table}[htbp]
    \centering
    \begin{tabular}{c|c|c}
         & \multicolumn{2}{c}{Scenario I} \\
         & $K=3$, $s=6$ & $K=3$, $s=12$\\
        \hline
        Bayesian & 369 & 368 \\
        PCA-KM & 0.04 & 0.03\\
        MClust & 1.27 & 1.51\\
        SKM & 22.61 & 21.73\\
        CHIME & 11.28 & 11.37
    \end{tabular}
    \caption{Empirical running times (seconds) for all methods in simulation studies.}
    \label{tab:running_time}
\end{table}

\begin{figure}[htbp]
    \centering
    \includegraphics[width=.8\textwidth]{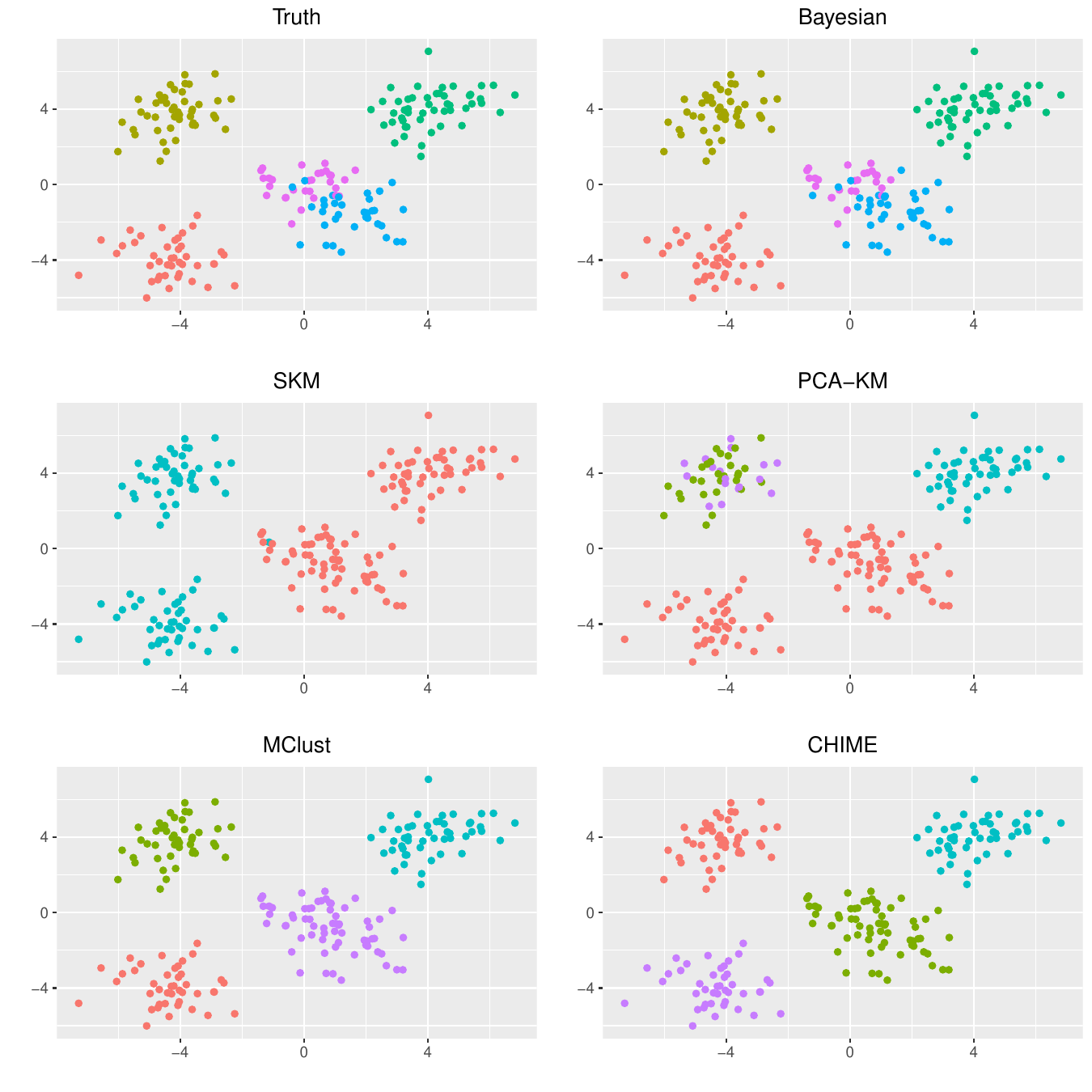}
    \caption{Clustering results of different methods compared with the true cluster memberships in the simulation Scenario I with $K^*=5$ and $s=6$ in a randomly selected simulation replicate. Data points are projected onto the subspace of the first two coordinates and different colors correspond to different estimated cluster memberships of the data points.}
    \label{fig:clustering_k5s6}
\end{figure}


\begin{figure}[tbp]
    \centering
    \includegraphics[width=.8\textwidth]{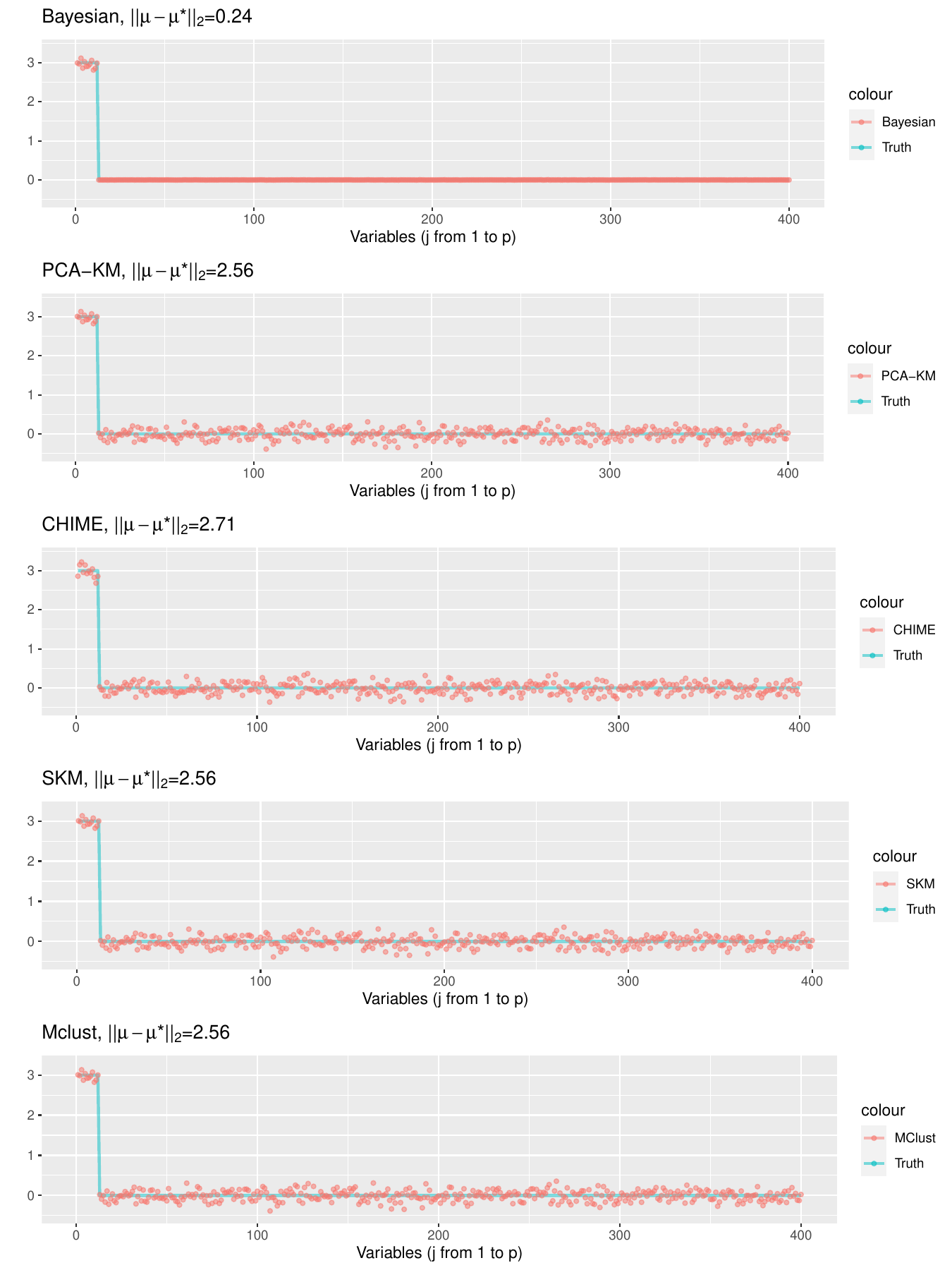}
    \caption{Estimation of mean vector $\bmu_1^*$ of different methods in Scenario I with $K^*=3$ and $s=12$ in a randomly selected simulation replicate.}
    \label{fig:sparsity_k3s12}
\end{figure}

\begin{figure}[tbp]
    \centering
    \includegraphics[width = .76\textwidth ]{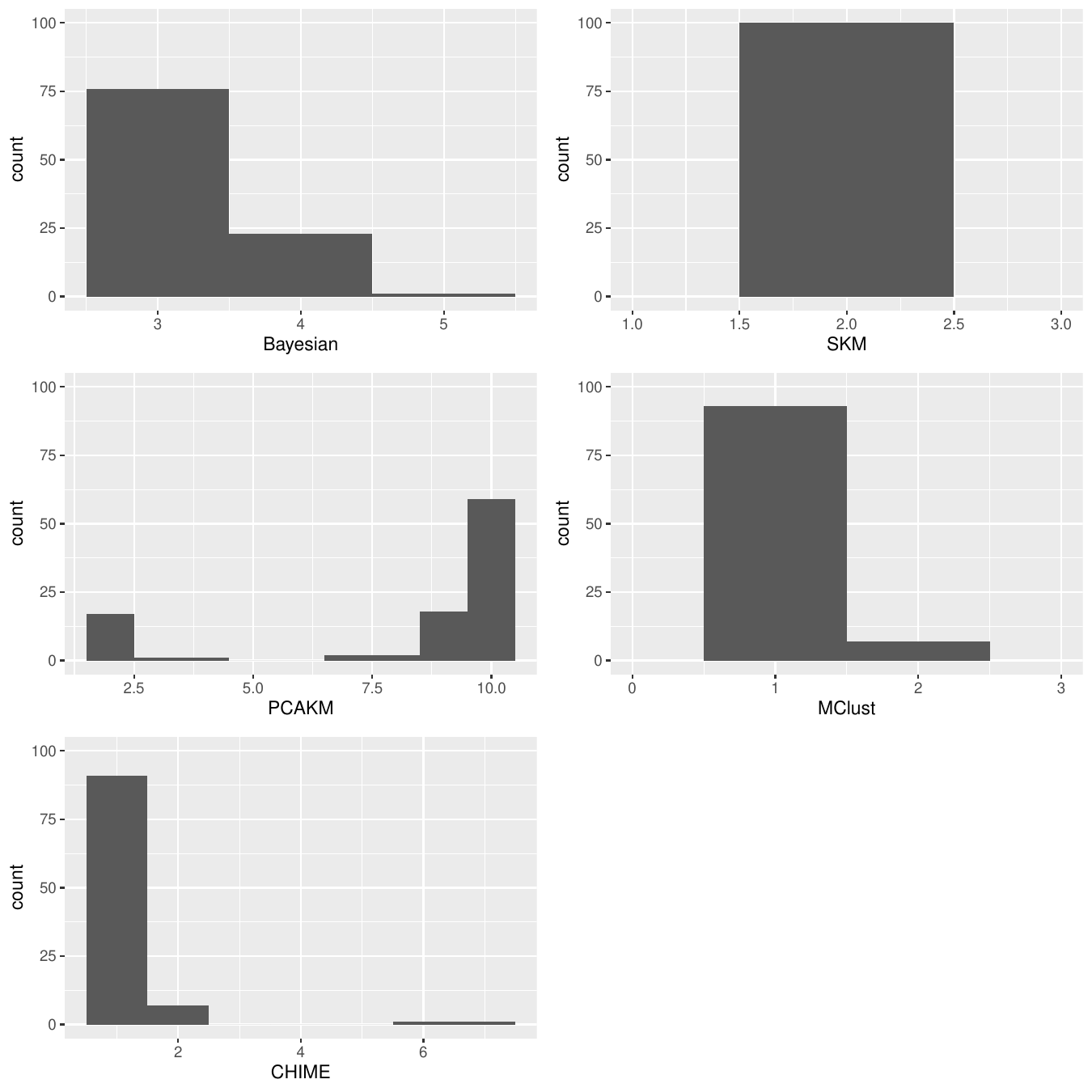}
    \caption{Histograms of estimated number of clusters under different methods in Scenario III of simulation studies.}
    \label{fig:hist_scen3}
\end{figure}


\begin{figure}[tbp]
    \centering
    \includegraphics[width = .8\textwidth ]{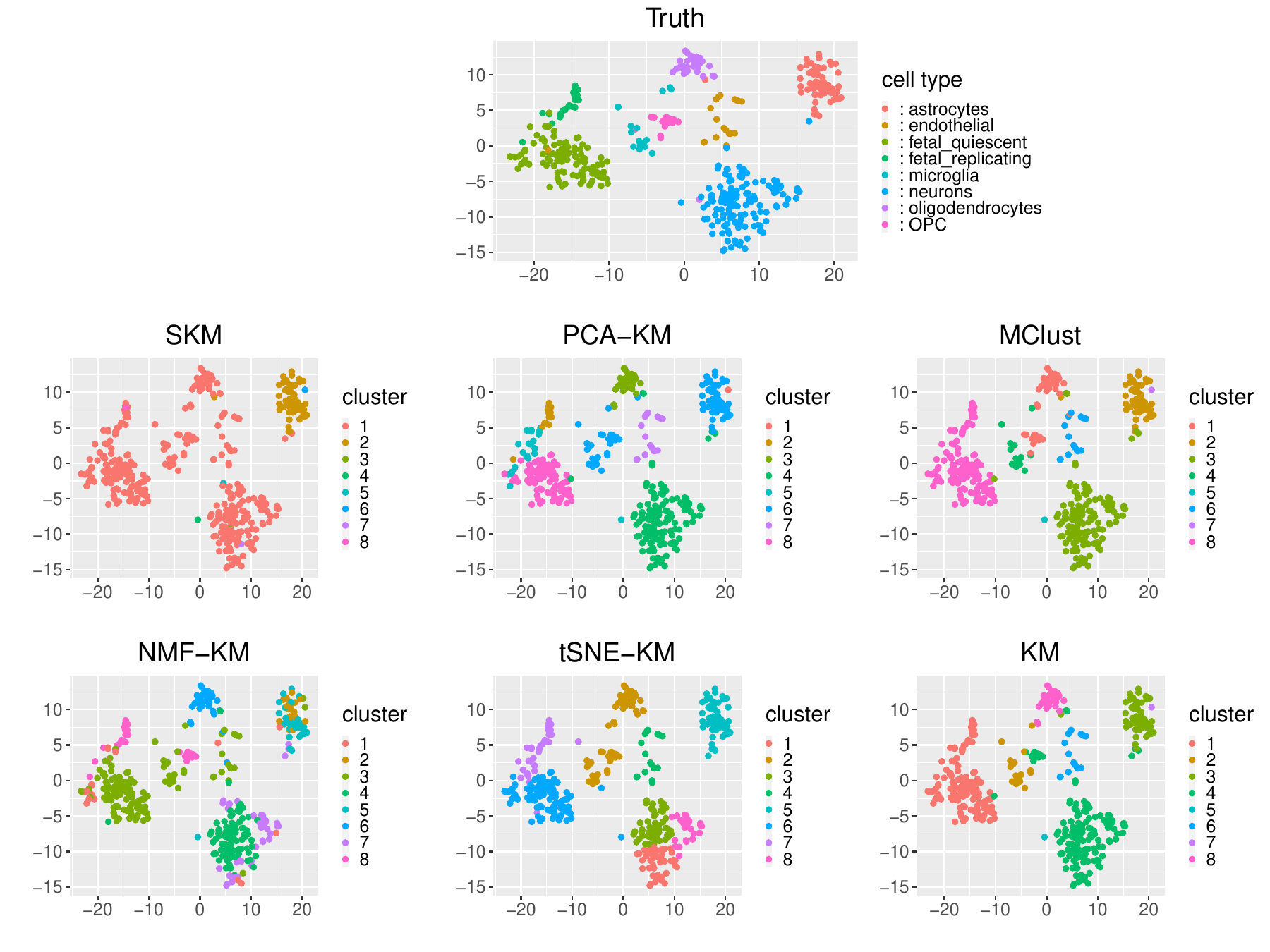}
    \caption{Clustering results of scRNA-Seq data corresponding to alternative methods. The number of clusters is set to be truth ($K=8$) for all methods.}
    \label{fig:real_data_true_k}
\end{figure}

\begin{table}[htbp]
    \centering
    \begin{tabular}{|c|c|c|}
        \hline Methods & ARI & NMI \\
        \hline KM & 0.79 & 0.77 \\
        \hline tSNE-KM & 0.63 & 0.73 \\
        \hline PCA-KM & 0.81 & 0.79 \\
        \hline NMF-KM &  0.77 & 0.78 \\
        \hline SKM & 0.15 &  0.23 \\
        \hline MClust & 0.83 & 0.79 \\
        \hline 
    \end{tabular}
    \caption{ARIs and NMIs of different methods on scRNA-Seq data. The number of clusters is set to be truth ($K=8$) for all methods.}
    \label{tab:real_data_true_k}
\end{table}



\end{appendices}

\clearpage

\bibliographystyle{apalike}
\bibliography{ref}

\end{document}